\documentclass[a4paper]{amsart}

\usepackage[utf8]{inputenc}
\usepackage[english]{babel}
\usepackage[babel]{csquotes}
\usepackage{amsmath, amsthm, amssymb, upref}
\usepackage{color}
\usepackage{exscale}

\renewcommand{\labelenumi}{\alph{enumi})}

\renewcommand{\emph}[1]{\textbf{#1}}

\newcommand{\alignDotsSpace}{~\,}
\newcommand{\alignLongFormula}{\quad}
\newcommand{\DegRevLex}{\texttt{DegRevLex}}
\newcommand{\oline}[1]{\overline{#1}}
\newcommand{\restrict}[2]{#1 \vert_{#2}}
\newcommand{\RedR}[1]{\stackrel{#1}{\longrightarrow}}
\newcommand{\RedL}[1]{\stackrel{#1}{\longleftarrow}}
\newcommand{\RedLR}[1]{\stackrel{#1}{\longleftrightarrow}}

\DeclareMathOperator{\BF}{BF}
\DeclareMathOperator{\codim}{codim}
\DeclareMathOperator{\ind}{ind}
\DeclareMathOperator{\lcm}{lcm}
\DeclareMathOperator{\LT}{LT}
\DeclareMathOperator{\Mat}{Mat}
\DeclareMathOperator{\NF}{NF}
\DeclareMathOperator{\NR}{NR}
\DeclareMathOperator{\Pos}{Pos}
\DeclareMathOperator{\SV}{S}
\DeclareMathOperator{\Supp}{Supp}
\DeclareMathOperator{\Syz}{Syz}
\DeclareMathOperator{\tr}{tr}

\newtheoremstyle
    {proposition}
    {2ex plus0.4ex minus0.2ex}
    {2ex plus0.4ex minus0.2ex}
    {\itshape}
    {}
    {\bfseries}
    {}
    {0.5em}
    {\thmname{#1} \thmnumber{#2.} \thmnote{(#3)\newline}}
\theoremstyle{proposition}
\newtheorem{thm}{Theorem}[section]
\newtheorem{lem}[thm]{Lemma}
\newtheorem{prop}[thm]{Proposition}
\newtheorem{cor}[thm]{Corollary}

\newtheoremstyle
    {definition}
    {2ex plus0.4ex minus0.2ex}
    {2ex plus0.4ex minus0.2ex}
    {}
    {}
    {\bfseries}
    {}
    {0.5em}
    {\thmname{#1} \thmnumber{#2.} \thmnote{(#3)\newline}}
\theoremstyle{definition}
\newtheorem{defn}[thm]{Definition}

\newtheorem{exmp}[thm]{Example}
\newtheorem{rem}[thm]{Remark}

\usepackage{algorithm, algorithmic}
\newcommand{\CHOOSE}{\textbf{choose }}
\newcommand{\assign}{:=}

\definecolor{linkcolor}{rgb}{0.0, 0.0, 0.7}
\usepackage{hyperref}
\hypersetup{
    pdfstartview = FitH,
    breaklinks = true,
    colorlinks = true,
    frenchlinks = false,
    bookmarksopen = true,
    allcolors = linkcolor,
    pdftitle = {Module Border Bases},
    pdfauthor = {Markus Kriegl},
    pdfsubject = {MSC2010. Primary 13P10},
    pdfkeywords = {border bases, module border bases, quotient module border
    bases, subideal border bases, characterization, computation}
}

\author{Markus Kriegl}
\address{Fakultät für Informatik und Mathematik, Universität Passau, D-94030
Passau, Germany}
\email{markus.kriegl@uni-passau.de}
\date{\today}

\title{Module Border Bases}
\subjclass[2010]{Primary 13P10}
\keywords{border bases, module border bases, quotient module border bases,
subideal border bases, characterization, computation}
\thanks{Special thanks go to Martin Kreuzer and Tobias Kaiser, both Universität
Passau (Germany), who have corrected my master thesis---which brought me to
the ideas in this paper---and with whom I have had helpful discussions that led
to the final version of this paper.}

\begin{document}

%
%

\begin{abstract}
In this paper, we generalize the notion of border bases of zero-dimensional
polynomial ideals to the module setting. To this end, we introduce order
modules as a generalization of order ideals and module border bases of
submodules with finite codimension in a free module as a generalization of
border bases of zero-dimensional ideals in the first part of this paper. In
particular, we extend the division algorithm for border bases to the module
setting, show the existence and uniqueness of module border bases, and
characterize module border bases analogously like border bases via the special
generation property, border form modules, rewrite rules, commuting matrices, and
liftings of border syzygies. Furthermore, we deduce Buchberger's Criterion for
Module Border Bases and give an algorithm for the computation of module border
bases that uses linear algebra techniques. In the second part, we further
generalize the notion of module border bases to quotient modules. We then show
the connection between quotient module border bases and special module border
bases and deduce characterizations similar to the ones for module border bases.
Moreover, we give an algorithm for the computation of quotient module border
bases using linear algebra techniques, again. At last, we prove that subideal
border bases are isomorphic to special quotient module border bases. This
isomorphy immediately yields characterizations and an algorithm for the
computation of subideal border bases.
\end{abstract}

\maketitle

%
%

\section{Introduction}
\label{sect:intro}
Border bases of zero-dimensional ideals have turned out to be a good alternative
for Gröbner bases as they have nice numerical properties, see for instance
\cite{Stetter2004}, \cite{OilHilbert}, \cite{Mourrain-pythagore} and
\cite{Mourrain-stableNF}. In the paper \cite{SubBB}, border bases have been
generalized to subideal border bases. The main difference is that the authors
additionally ensure that all the computation takes place in a so-called
\enquote{subideal}. This allows them to take external knowledge about the
system into account, e.\,g.\ if some physical properties are known to be
satisfied. We use a similar construction in this paper in order to generalize
border bases and subideal border bases of zero-dimensional ideals, i.\,e.\
ideals with finite codimension, to the module setting. Therefore, we introduce
new concepts, namely module border bases and quotient module border bases of
submodules with finite codimension.

In Part~\ref{part:freeMod} of this paper, we introduce the notion of module
border bases of free modules with finite rank as a generalization of border
bases to the module setting. More precisely, we regard a polynomial ring~$P$ in
finitely many indeterminates and determine bases of $P$-submodules $U \subseteq
P^r$ with finite codimension similar to border bases, where $r \in \mathbb N$.
We show that we can reuse the concepts and ideas of border bases introduced
in \cite[Section~6.4]{KR2}, \cite{CharBB}, and \cite{CompBB} in an almost
straightforward way to have corresponding results for module border bases. We
firstly introduce the basic concepts and a division algorithm for module border
bases (Theorem~\ref{thm:divAlg}) in Section~\ref{sect:divAlg}, then we prove the
existence and uniqueness of module border bases
(Proposition~\ref{thm:existUnique}) in Section~\ref{sect:existUnique}. In
Section~\ref{sect:chars}, we characterize module border bases via the special
generation property (Theorem~\ref{thm:specGen}), via border form modules
(Theorem~\ref{thm:BFMod}), via rewrite rules (Theorem~\ref{thm:rewrite}), via
commuting matrices (Theorem~\ref{thm:commMat}), and via liftings of border
syzygies (Theorem~\ref{thm:liftings}). Finally, we determine Buchberger's
Criterion for Module Border Bases (Theorem~\ref{thm:buchbCrit}), which allows us
to check easily whether a given module border prebasis is a module border basis,
or not. At last, we give an algorithm for the computation of module border bases
which uses linear algebra techniques in Section~\ref{sect:moduleBBAlg}.\\
In Part~\ref{part:quotMod}, we further generalize the notion of module border
bases of free modules to module border bases of quotient modules. More
precisely, we regard a polynomial ring~$P$ in finitely many indeterminates and
determine bases of $P$-submodules $U^S \subseteq P^r / S$ with finite
codimension similar to border bases, where $r \in \mathbb N$ and $S \subseteq
P^r$ is a given $P$-submodule. In Section~\ref{sect:quotModuleBB}, we introduce
the basic concepts of quotient module border bases. For a quotient module
border prebasis, we then define characterizing module border prebases
(Definition~\ref{defn:charModuleBB}) and use these characterizing module
border prebases to characterize quotient module border bases
(Theorem~\ref{thm:char-quot}). As an immediate consequence, we give an algorithm
for the computation of quotient module border bases with linear algebra
techniques (Corollary~\ref{thm:quotModuleBBAlg}) and characterize quotient
module border bases similar to module border bases via the special generation
property (Corollary~\ref{thm:specGen-quot}), via border form modules
(Corollary~\ref{thm:BFMod-quot}), via rewrite rules
(Corollary~\ref{thm:rewrite-quot}), via commuting matrices
(Corollary~\ref{thm:commMat-quot}), via liftings of border syzygies
(Corollary~\ref{thm:liftings-quot}), and via Buchberger's Criterion for Quotient
Module Border Bases (Corollary~\ref{thm:buchbCrit-quot}).\\
At last, we show that subideal border bases---which have recently been
introduced in \cite{SubBB}---can be regarded as special quotient module border
bases in Section~\ref{sect:subBB}. Therefore, the characterizations of quotient
module border bases in Section~\ref{sect:quotModuleBB} immediately yield
corresponding characterizations of subideal border bases
(Remark~\ref{rem:char-subBB}). In particular, we get another proof for the
characterization of subideal border bases via the special generation property,
which has already been proven in \cite{SubBB}, and we get the characterizations
of subideal border bases via border form modules, via rewrite rules, via
commuting matrices, via liftings of border syzygies, and via Buchberger's
Criterion for Subideal Border Bases which have not been proven by now. Moreover,
the algorithm for the computation of quotient module border bases can be used to
determine an algorithm for the computation of subideal border bases with linear
algebra techniques (Corollary~\ref{thm:subBBcompute}). The only way to compute
subideal border bases by now uses Gröbner bases techniques and thus needs, in
general, much more computational effort (cf.\ \cite[Section~6]{SubBB}).
Furthermore, we give some remarks how the further results about subideal border
bases in \cite{SubBB}, e.\,g.\ a division algorithm or an index, coincide with
our construction of subideal border bases as special quotient module border
bases.

For the whole paper, we use the notation and results of \cite{KR1} and
\cite{KR2}. In particular, we let $K$ be a field and $P = K[x_1, \ldots, x_n]$
be the polynomial ring in the indeterminates $x_1, \ldots, x_n$. Furthermore, we
let $r \in \mathbb N$ and $\{ e_1, \ldots, e_r \} \subseteq P^r$ be the
canonical $P$-module basis of the free $P$-module $P^r$. The monoid of all terms
$x_1^{\alpha_1} \cdots x_n^{\alpha_n} \in P$ with $\alpha_1, \ldots, \alpha_n
\in \mathbb N$ is denoted by $\mathbb T^n$ and the monoid of all terms $t e_k
\in P^r$ with $t \in \mathbb T^n$ and $k \in \{ 1, \ldots, r \}$ is denoted by
$\mathbb T^n \langle e_1, \ldots, e_r \rangle$. The set of all terms in $\mathbb
T^n$ with degree $d \in \mathbb N$ is denoted by $\mathbb T^n_d$ and we use
similar constructions, e.\,g.\ $\mathbb T^n_{\leq d} \langle e_1, \ldots, e_r
\rangle$ denotes the set of all terms $t e_k \in \mathbb T^n \langle e_1,
\ldots, e_r \rangle$ such that $\deg(t) \leq d$. Moreover, for any term ordering
$\sigma$ on the set of terms $\mathbb T^n \langle e_1, \ldots, e_r \rangle$ and
$P$-submodule $U \subseteq P^r$, we let $\LT_\sigma \{ U \} = \{
\LT_\sigma(\mathcal V) \mid \mathcal V \in U \setminus \{ 0 \} \}$ be the
monomodule of all leading terms of the vectors in $U \setminus \{ 0 \}$, and we
let $\mathcal O_\sigma(U) = \mathbb T^n \langle e_1, \ldots, e_r \rangle
\setminus \LT_\sigma \{ U \}$. At last, for every $P$-module~$M$ (respectively
$K$-vector space) and every $P$-submodule $U \subseteq M$ (respectively
$K$-vector subspace), we let
\begin{align*}
\varepsilon_U: M \twoheadrightarrow M / U, \quad m \mapsto m + U
\end{align*}
be the canonical $P$-module epimorphism ($K$-vector space epimorphism).

%
%

\part{Module Border Bases of Free Modules}
\label{part:freeMod}

In the first part of this paper, we generalize the results about border bases of
zero-dimensional ideals in $P$, i.\,e.\ ideals which have finite $K$-codimension
in~$P$, in \cite[Section~6.4]{KR2}, \cite{CharBB}, and \cite{CompBB} to the
module setting. We imitate the definitions and propositions there in order to
have similar results for $P$-submodules $U \subseteq P^r$ of finite
$K$-codimension in the free $P$-module $P^r$. At first, we generalize the
concepts of border bases and the corresponding division algorithm to module
border bases in Section~\ref{sect:divAlg}. Section~\ref{sect:existUnique} then
shows the existence and uniqueness of module border bases. In
Section~\ref{sect:chars}, we characterize module border bases the same way as
border bases have been characterized. At last we determine an algorithm that
computes module border bases in Section~\ref{sect:moduleBBAlg}.

%
%

\section{Module Border Division}
\label{sect:divAlg}

In this section, we generalize the concept of border basis of zero-dimensional
ideals from \cite[Section~6.4]{KR2} to the module setting in a straightforward
way. We see that the analogs of order ideals, namely order modules, cf.\
Definition~\ref{defn:orderModule}, are also subsets of $\mathbb T^n \langle e_1,
\ldots, e_r \rangle$ which are closed under forming divisors. This allows us to
define the concepts of borders in Definition~\ref{defn:border}, of the index of
a order module in Definition~\ref{defn:index} and module border bases in
Definition~\ref{defn:moduleBB}. At last, we determine a division algorithm for
module border bases in Theorem~\ref{thm:divAlg}. Similar to Gröbner bases and
border bases, this division algorithm plays a central role in many proofs. We
end up this section with some immediate consequences of this division algorithm.

We start with the definition and properties of order ideals. Our definition of
order ideals generalizes the order ideals defined in \cite[Defn.~6.4.3]{KR2},
namely we also regard the empty set as an order ideal. Though this seems to be
quite a technical act, we see in Remark~\ref{rem:needEmptyOrderModule} that this
generalization is necessary.

\begin{defn}
\label{defn:orderIdeal}
A set $\mathcal O \subseteq \mathbb T^n$ is called an \emph{order ideal} if it
is closed under forming divisors.
\end{defn}

\begin{defn}
\label{defn:orderIdealBorder}
Let $\mathcal O \subseteq \mathbb T^n$ be an order ideal.
\begin{enumerate}
  \item We call the set
  \begin{align*}
  \partial^1 \mathcal O = \partial \mathcal O = ((\mathbb T^n_1 \cdot \mathcal
  O) \cup \{ 1 \}) \setminus \mathcal O \subseteq \mathbb T^n
  \end{align*}
  the \emph{(first) border} of~$\mathcal O$. The \emph{(first) border closure}
  of~$\mathcal O$ is the set
  \begin{align*}
  \oline{\partial^1 \mathcal O} = \oline{\partial \mathcal O} = \mathcal O \cup
  \partial \mathcal O \subseteq \mathbb T^n.
  \end{align*}
  \item For every $k \in \mathbb N \setminus \{ 0 \}$, we inductively define the
  \emph{$(k+1)^\text{st}$~border} of~$\mathcal O$ by the rule
  \begin{align*}
  \partial^{k+1} \mathcal O = \partial (\oline{\partial^k \mathcal O}) \subseteq
  \mathbb T^n,
  \end{align*}
  and the \emph{$(k+1)^\text{st}$~border closure} by the rule
  \begin{align*}
  \oline{\partial^{k+1} \mathcal O} = \oline{\partial^k \mathcal O} \cup
  \partial^{k+1} \mathcal O \subseteq \mathbb T^n.
  \end{align*}
  For convenience, we let
  \begin{align*}
  \partial^0 \mathcal O = \oline{\partial^0 \mathcal O} = \mathcal O.
  \end{align*}
\end{enumerate}
\end{defn}

\begin{prop}
\label{thm:orderIdealBorder}
Let $\mathcal O \subseteq \mathbb T^n$ be an order ideal.
\begin{enumerate}
  \item For every $k \in \mathbb N$, we have a disjoint union
  \begin{align*}
  \oline{\partial^k \mathcal O} = \bigcup_{i=0}^k \partial^i \mathcal O.
  \end{align*}
  \item We have a disjoint union
  \begin{align*}
  \mathbb T^n = \bigcup_{i=0}^\infty \partial^i \mathcal O.
  \end{align*}
  \item For every $k \in \mathbb N \setminus \{ 0 \}$, we have
  \begin{align*}
  \partial^k \mathcal O & = ((\mathbb T^n_k \cdot \mathcal O) \cup \mathbb
  T^n_{k-1}) \setminus (\mathbb T^n_{<k} \cdot \mathcal O)\\
  & = \begin{cases} (\mathbb T^n_k \cdot \mathcal O) \setminus (\mathbb T^n_{<k}
  \cdot \mathcal O) & \text{if $\mathcal O \neq \emptyset$},\\
  \mathbb T^n_{k-1} & \text{if $\mathcal O = \emptyset$}.
  \end{cases}
  \end{align*}
  \item Let $t \in \mathbb T^n$ be a term. Then there exists a factorization of
  the form $t = t' b$ with a term $t' \in \mathbb T^n$ and $b \in \partial
  \mathcal O$ if and only if $t \in \mathbb T^n \setminus \mathcal O$.
\end{enumerate}
\end{prop}

\begin{proof}
We start to prove claim~a) with induction over $k \in \mathbb N$. For the
induction start, Definition~\ref{defn:orderIdealBorder} yields
\begin{align*}
\oline{\partial^0 \mathcal O} = \partial^0 \mathcal O = \bigcup_{i=0}^0
\partial^i \mathcal O
\end{align*}
and
\begin{align*}
\oline{\partial^1 \mathcal O} = \oline{\partial^0 \mathcal O} \cup \partial^1
\mathcal O = \partial^0 \mathcal O \cup \partial^1 \mathcal O = \bigcup_{i=0}^1
\partial^i \mathcal O.
\end{align*}
For every $k \in \mathbb N \setminus \{ 0 \}$,
Definition~\ref{defn:orderIdealBorder} and the induction hypothesis yield
\begin{align*}
\oline{\partial^{k+1} \mathcal O} = \oline{\partial^k \mathcal O} \cup
\partial^{k+1} \mathcal O = \bigcup_{i=0}^k \partial^i \mathcal O \cup
\partial^{k+1} \mathcal O = \bigcup_{i=0}^{k+1} \partial^i \mathcal O.
\end{align*}
Moreover, for every $i, j \in \mathbb N$ with $i \neq j$, the borders
$\partial^i \mathcal O$ and $\partial^j \mathcal O$ are disjoint according to
Definition~\ref{defn:orderIdealBorder}. Thus the claim follows.

Since every term in~$\mathbb T^n$ is in~$\partial^i \mathcal O$ for some $i \in
\mathbb N$ by Definition~\ref{defn:orderIdealBorder}, claim~b) is a direct
consequence of claim~a).

We now prove claim~c) by induction over $k \in \mathbb N \setminus \{ 0 \}$. For
the induction start $k=1$, Definition~\ref{defn:orderIdealBorder} yields
\begin{align*}
\partial^1 \mathcal O & = ((\mathbb T^n_1 \cdot \mathcal O) \cup \{ 1 \})
\setminus \mathcal O\\
& = ((\mathbb T^n_1 \cdot \mathcal O) \cup \mathbb T^n_0) \setminus (\mathbb
T^n_{<1} \cdot \mathcal O).
\end{align*}
For the induction step, we let $k > 1$. Then the induction hypothesis and
Definition~\ref{defn:orderIdealBorder} yield
\begin{align*}
\partial^k \mathcal O & = \partial (\oline{\partial^{k-1} \mathcal O})\\
& = \partial(\oline{((\mathbb T^n_{k-1} \cdot \mathcal O) \cup \mathbb
T^n_{k-2}) \setminus (\mathbb T^n_{<k-1} \cdot \mathcal O)})\\
& = \partial((\mathbb T^n_{\leq k-1} \cdot \mathcal O) \cup \mathbb T^n_{\leq
k-2})\\
& = ((\mathbb T^n_1 \cdot \mathbb T^n_{\leq k-1} \cdot \mathcal O) \cup
(\mathbb T^n_1 \cdot \mathbb T^n_{\leq k-2})) \setminus ((\mathbb T^n_{\leq k-1}
\cdot \mathcal O) \cup \mathbb T^n_{\leq k-2})\\
& = ((\mathbb T^n_k \cdot \mathcal O) \cup \mathbb T^n_{k-1}) \setminus (\mathbb
T^n_{<k} \cdot \mathcal O).
\end{align*}
If $\mathcal O \neq \emptyset$, we have $1 \in \mathcal O$ by
Definition~\ref{defn:orderIdeal}, and thus $\mathbb T^n_{k-1} \setminus (\mathbb
T^n_{<k} \cdot \mathcal O) = \emptyset$, and the equation above shows
\begin{align*}
\partial^k \mathcal O & = ((\mathbb T^n_k \cdot \mathcal O) \setminus (\mathbb
T^n_{<k} \cdot \mathcal O)) \cup (\mathbb T^n_{k-1} \setminus (\mathbb
T^n_{<k} \cdot \mathcal O))\\
& = (\mathbb T^n_k \cdot \mathcal O) \setminus (\mathbb T^n_{<k} \cdot \mathcal
O).
\end{align*}
If $\mathcal O = \emptyset$, the equation above yields
\begin{align*}
\partial^k \mathcal O = (\emptyset \cup \mathbb T^n_{k-1}) \setminus \emptyset =
\mathbb T^n_{k-1},
\end{align*}
and thus claim~c) follows.

Finally, we prove claim~d). We distinguish two cases.\\
For the first case, suppose that $\mathcal O = \emptyset$. Then we have
$\partial \mathcal O = \{ 1 \}$ by Definition~\ref{defn:orderIdealBorder} and
for every term $t \in \mathbb T^n \setminus \mathcal O = \mathbb T^n$, there is
the factorization $t = t \cdot 1$.\\
For the second case, suppose that $\mathcal O \neq \emptyset$. Let $t \in
\mathbb T^n \setminus \mathcal O$. Then there exists a $k \in \mathbb N
\setminus \{ 0 \}$ such that $t \in \partial^k \mathcal O = (\mathbb T^n_k \cdot
\mathcal O) \setminus (\mathbb T^n_{<k} \cdot \mathcal O)$ according to the
claims~b) and~c). In particular, we can write $t = x_\ell t_1 t_2$ with $\ell
\in \{ 1, \ldots, n \}$, $t_1 \in \mathbb T^n_{k-1}$, and $t_2 \in \mathcal O$.
Assume that $x_\ell t_2 \in \mathcal O$. Then we get the contradiction $t = t_1
(x_\ell t_2) \in \mathbb T^n_{<k} \cdot \mathcal O$. Thus
Definition~\ref{defn:orderIdealBorder} yields $x_\ell t_2 \in \partial \mathcal
O$, and the the first implication follows because of~$t = t_1 (x_\ell t_2)$.
For the converse implication, let $t' \in \mathbb T^n$ and $b \in \partial
\mathcal O$. Assume that $t' b \in \mathcal O$. Then
Definition~\ref{defn:orderIdeal} yields the contradiction $b \in \mathcal O$.
Thus we have $t' b \in \mathbb T^n \setminus \mathcal O$ and the claim follows.
\end{proof}

Having defined order ideals of~$\mathbb T^n$, we generalize this notion to order
module setting. We start with the definition of order modules of~$\mathbb T^n
\langle e_1, \ldots, e_r \rangle$, the border of order modules, and some
propositions of them. Our definitions and propositions are analogous versions
of the corresponding ones in \cite[Defn.~6.4.3/4]{KR2} and
\cite[Prop.~6.4.6]{KR2}.

\begin{defn}
\label{defn:orderModule}
Let $\mathcal O_1, \ldots, \mathcal O_r \subseteq \mathbb T^n$ be order ideals.
Then we call the set
\begin{align*}
\mathcal M = \mathcal O_1 \cdot e_1 \cup \cdots \cup \mathcal O_r \cdot e_r
\subseteq \mathbb T^n \langle e_r, \ldots, e_r \rangle
\end{align*}
an \emph{order module}.
\end{defn}

\begin{defn}
\label{defn:border}
Let $\mathcal M = \mathcal O_1 e_1 \cup \cdots \cup \mathcal O_r e_r$ with order
ideals $\mathcal O_1, \ldots, \mathcal O_r \subseteq \mathbb T^n$ be an order
module.
\begin{enumerate}
  \item We call the set
  \begin{align*}
  \partial^1 \mathcal M = \partial \mathcal M & = ((\mathbb T^n_1 \cdot \mathcal
  M) \cup \{ e_1, \ldots, e_r \}) \setminus \mathcal M\\
  & = \partial \mathcal O_1 \cdot e_1 \cup \cdots \cup \partial \mathcal O_r
  \cdot e_r \subseteq \mathbb T^n \langle e_1, \ldots, e_r \rangle
  \end{align*}
  the \emph{(first) border} of~$\mathcal M$. The \emph{(first) border closure}
  of~$\mathcal M$ is the set
  \begin{align*}
  \oline{\partial^1 \mathcal M} = \oline{\partial \mathcal M} & = \mathcal M
  \cup \partial \mathcal M\\
  & = \oline{\partial \mathcal O_1} \cdot e_1 \cup \cdots \cup \oline{\partial
  \mathcal O_r} \cdot e_r \subseteq \mathbb T^n \langle e_1, \ldots, e_r
  \rangle.
  \end{align*}
  \item For every $k \in \mathbb N \setminus \{ 0 \}$, we inductively define the
  \emph{$(k+1)^{\text{st}}$~border} of~$\mathcal M$ by
  \begin{align*}
  \partial^{k+1} \mathcal M & = \partial (\oline{\partial^k \mathcal M})\\
  & = \partial^{k+1} \mathcal O_1 \cdot e_1 \cup \cdots \cup \partial^{k+1}
  \mathcal O_r \cdot e_r \subseteq \mathbb T^n \langle e_1, \ldots, e_r \rangle,
  \end{align*}
  and the \emph{$(k+1)^{\text{st}}$~border closure} of~$\mathcal M$ by the rule
  \begin{align*}
  \oline{\partial^{k+1} \mathcal M} & = \oline{\partial^k \mathcal M} \cup
  \partial^{k+1} \mathcal M\\
  & = \oline{\partial^{k+1} \mathcal O_1} \cdot e_1 \cup \cdots \cup
  \oline{\partial^{k+1} \mathcal O_r} \cdot e_r \subseteq \mathbb T^n \langle
  e_1, \ldots, e_r \rangle.
  \end{align*}
  For convenience, we let
  $$\partial^0 \mathcal M = \oline{\partial^0 \mathcal M} = \mathcal M.$$
\end{enumerate}
\end{defn}

\begin{exmp}
\label{exmp:border}
Let $K$ be a field, and let
\begin{align*}
\mathcal O_1 & = \{ x, y, 1 \} \subseteq \mathbb T^2\\
\intertext{and}
\mathcal O_2 & = \{ x^2, x, 1 \} \subseteq \mathbb T^2.
\end{align*}
Then $\mathcal O_1$ and $\mathcal O_2$ are both order ideals with first borders
\begin{align*}
\partial \mathcal O_1 & = \{ x^2, xy, y^2 \} \subseteq \mathbb T^2\\
\intertext{and}
\partial \mathcal O_2 & = \{ x^3, x^2y, xy, y \} \subseteq \mathbb T^2,
\end{align*}
and second borders
\begin{align*}
\partial^2 \mathcal O_1 & = \{ x^3, x^2y, xy^2, y^3 \} \subseteq \mathbb T^2\\
\intertext{and}
\partial^2 \mathcal O_2 & = \{ x^4, x^3y, x^2y^2, xy^2, y^2 \} \subseteq
\mathbb T^2.
\end{align*}
Let $e_1 = (1,0) \in (K[x,y])^2$ and $e_2 = (0,1) \in (K[x,y])^2$. Then
\begin{align*}
\mathcal M = \{ x e_1, y e_1, e_1, x^2 e_2, x e_2, e_2 \} \subseteq \mathbb T^2
\langle e_1, e_2 \rangle
\end{align*}
is an order module with first border
\begin{align*}
\partial \mathcal M = \{ x^2 e_1, xy e_1, y^2 e_1, x^3 e_2, x^2y e_2, xy e_2, y
e_2 \} \subseteq \mathbb T^2 \langle e_1, e_2 \rangle
\end{align*}
and second border
\begin{align*}
\partial^2 \mathcal M = \{ x^3 e_1, x^2y e_1, xy^2 e_1, y^3 e_1, x^4 e_2, x^3y
e_2, x^2y^2 e_2, xy^2 e_2, y^2 e_2 \} \subseteq \mathbb T^2 \langle e_1, e_2
\rangle.
\end{align*}
\end{exmp}

\begin{prop}
\label{thm:border}
Let $\mathcal M = \mathcal O_1 e_1 \cup \cdots \cup \mathcal O_r e_r$ with order
ideals $\mathcal O_1, \ldots, \mathcal O_r \subseteq \mathbb T^n$ be an order
module.
\begin{enumerate}
  \item For every $k \in \mathbb N$, we have a disjoint union
  \begin{align*}
  \oline{\partial^k \mathcal M} = \bigcup_{i=0}^k \partial^i \mathcal M.
  \end{align*}
  \item We have a disjoint union
  \begin{align*}
  \mathbb T^n \langle e_1, \ldots, e_r \rangle = \bigcup_{i=0}^\infty
  \partial^i \mathcal M.
  \end{align*}
  \item For every $k \in \mathbb N \setminus \{ 0 \}$, we have
  \begin{align*}
  \partial^k \mathcal M = ((\mathbb T^n_k \cdot \mathcal M) \cup \mathbb
  T^n_{k-1} \langle e_1, \ldots, e_r \rangle) \setminus (\mathbb T^n_{<k} \cdot
  \mathcal M).
  \end{align*}
  \item Let $t e_k \in \mathbb T^n \langle e_1, \ldots, e_r \rangle$ be a term.
  Then there exists a factorization of the form $t e_k = t' b e_k$ with a term
  $t' \in \mathbb T^n$ and $b e_k \in \partial \mathcal M$ if and only if $t e_k
  \in \mathbb T^n \langle e_1, \ldots, e_r \rangle \setminus \mathcal M$.
\end{enumerate}
\end{prop}

\begin{proof}
For every $s \in \{ 1, \ldots, r \}$, we have $\{ (p_1, \ldots, p_r) \in
\mathcal M \mid p_s \neq 0 \} = \mathcal O_s \cdot e_s$ by
Definition~\ref{defn:orderModule}. Thus the claim immediately follows from
Proposition~\ref{thm:orderIdealBorder}.
\end{proof}

Since Proposition~\ref{thm:orderIdealBorder} is a natural generalization of
\cite[Prop.~6.4.6]{KR2}, which additionally allows empty order ideals, we can
define an $\mathcal O$-index for an order ideal~$\mathcal O$ just the way it
has been defined in \cite[Defn.~6.4.7]{KR2}.

\begin{defn}
\label{defn:orderIdealIndex}
Let $\mathcal O \subseteq \mathbb T^n$ be an order ideal.
\begin{enumerate}
  \item For every term $t \in \mathbb T^n$, the number $i \in \mathbb N$ such
  that $t \in \partial^i \mathcal O$, which is unique according to
  Proposition~\ref{thm:orderIdealBorder}, is called the \emph{$\mathcal
  O$-index} of~$t$ and it is denoted by~$\ind_{\mathcal O}(t)$.
  \item For a polynomial $p \in P \setminus \{ 0 \}$, we define the
  \emph{$\mathcal O$-index} of~$p$ by
  \begin{align*}
  \ind_{\mathcal O}(p) = \max \{ \ind_{\mathcal O}(t) \mid t \in \Supp(p) \}.
  \end{align*}
\end{enumerate}
\end{defn}

Now we can immediately deduce the same properties for the $\mathcal O$-index
of an order ideal~$\mathcal O$, as it has been done in \cite[Prop.~6.4.8]{KR2}.
Though the proof stays essentially the same when we also regard the empty set as
an order ideal, we will give the proof of the following proposition for the
convenience of the reader.

\begin{prop}
\label{thm:orderIdealIndex}
Let $\mathcal O \subseteq \mathbb T^n$ be an order ideal.
\begin{enumerate}
  \item For a term $t \in \mathbb T^n \setminus \mathcal O$, the number $i =
  \ind_{\mathcal O}(t)$ is the smallest natural number such that $t = t' b$
  with $t' \in \mathbb T^n_{i-1}$ and $b \in \partial \mathcal O$.
  \item Given a term $t \in \mathbb T^n$ and a term $t' \in \mathbb T^n$, we
  have
  \begin{align*}
  \ind_{\mathcal O}(t t') \leq \deg(t) + \ind_{\mathcal O}(t').
  \end{align*}
  \item For two polynomials $p, q \in P \setminus \{ 0 \}$ such that $p + q \neq
  0$, we have the inequality
  \begin{align*}
  \ind_{\mathcal O}(p + q) \leq \max \{ \ind_{\mathcal O}(p), \ind_{\mathcal
  O}(q) \}.
  \end{align*}
  \item For two polynomials $p, q \in P \setminus \{ 0 \}$, we have the
  inequality
  \begin{align*}
  \ind_{\mathcal O}(pq) \leq \deg(p) + \ind_{\mathcal O}(q).
  \end{align*}
\end{enumerate}
\end{prop}

\begin{proof}
Claim~a) is a direct consequence of Proposition~\ref{thm:orderIdealBorder}, and
claim~b) follows immediately from claim~a). Since $\Supp(p+q) \subseteq \Supp(p)
\cup \Supp(q)$, claim~c) follows immediately with
Definition~\ref{defn:orderIdealIndex}. At last, claim~d) follows from claim~b)
since $\Supp(pq) \subseteq \{ t t' \mid t \in \Supp(p), t'' \in \Supp(q) \}$.
\end{proof}

Since Proposition~\ref{thm:border} is an analogous version of
\cite[Prop.~6.4.6]{KR2} for an order module $\mathcal M \subseteq \mathbb T^n
\langle e_1, \ldots, e_r \rangle$, we can now define the $\mathcal
M$-index similar to the index of order ideals in
Definition~\ref{defn:orderIdealIndex}, and deduce some properties of it as it
has been done in Proposition~\ref{thm:orderIdealIndex}.

\begin{defn}
\label{defn:index}
Let $\mathcal M = \mathcal O_1 e_1 \cup \cdots \cup \mathcal O_r e_r$ with order
ideals $\mathcal O_1, \ldots, \mathcal O_r \subseteq \mathbb T^n$ be an order
module.
\begin{enumerate}
  \item For every term $t e_k \in \mathbb T^n \langle e_1, \ldots, e_r \rangle$,
  the number $i \in \mathbb N$ such that $t e_k \in \partial^i \mathcal M$,
  which is unique according to Proposition~\ref{thm:border}, is called the
  \emph{$\mathcal M$-index} of~$t e_k$ and is denoted by~$\ind_{\mathcal M}(t
  e_k)$.
  \item For a vector $\mathcal V = (p_1, \ldots, p_r) \in P^r \setminus \{ 0
  \}$, we define the \emph{$\mathcal M$-index} of~$\mathcal V$ by
  \begin{align*}
  \ind_{\mathcal M}(\mathcal V) = \max \{ \ind_{\mathcal M}(t e_k) \mid t e_k
  \in \Supp(\mathcal V) \}.
  \end{align*}
\end{enumerate}
\end{defn}

\begin{exmp}
\label{exmp:index}
We consider the order module
\begin{align*}
\mathcal M = \{ x e_1, y e_1, e_1, x^2 e_2, x e_2, e_2 \} \subseteq \mathbb T^2
\langle e_1, e_2 \rangle
\end{align*}
of Example~\ref{exmp:border}, again. Then we have
\begin{align*}
\ind_{\mathcal M}(x e_1) = 0
\end{align*}
and
\begin{align*}
\ind_{\mathcal M}(x^2y^2 e_2) = 2.
\end{align*}
Furthermore, we see that
\begin{align*}
\ind_{\mathcal M}(x e_1 + x^2 y^2 e_2) = \max \{ 0, 2 \} = 2.
\end{align*}
\end{exmp}

\begin{prop}
\label{thm:index}
Let $\mathcal M = \mathcal O_1 e_1 \cup \cdots \cup \mathcal O_r e_r$ with order
ideals $\mathcal O_1, \ldots, \mathcal O_r \subseteq \mathbb T^n$ be an order
module.
\begin{enumerate}
  \item For a term $t e_k \in \mathbb T^n \langle e_1, \ldots, e_r \rangle
  \setminus \mathcal M$, the number $i = \ind_{\mathcal M}(t e_k)$ is the
  smallest natural number such that $t = t' b$ with $t' \in \mathbb T^n_{i-1}$
  and $b \in \partial \mathcal O_k$.
  \item Given a term $t \in \mathbb T^n$ and a term $t' e_k \in \mathbb T^n
  \langle e_1, \ldots, e_r \rangle$, we have
  \begin{align*}
  \ind_{\mathcal M}(t t' e_k) \leq \deg(t) + \ind_{\mathcal M}(t' e_k).
  \end{align*}
  \item For two vectors $\mathcal V, \mathcal W \in P^r \setminus \{ 0 \}$ such
  that $\mathcal V + \mathcal W \neq 0$, we have the inequality
  \begin{align*}
  \ind_{\mathcal M}(\mathcal V + \mathcal W) \leq \max \{ \ind_{\mathcal
  M}(\mathcal V), \ind_{\mathcal M}(\mathcal W) \}.
  \end{align*}
  \item For a vector $\mathcal V \in P^r \setminus \{ 0 \}$ and a polynomial $p
  \in P \setminus \{ 0 \}$, we have the inequality
  \begin{align*}
  \ind_{\mathcal M}(p \mathcal V) \leq \deg(p) + \ind_{\mathcal M}(\mathcal V).
  \end{align*}
\end{enumerate}
\end{prop}

\begin{proof}
Let $s \in \{ 1, \ldots, r \}$. Then we have $\{ (p_1, \ldots, p_r) \in \mathcal
M \mid p_s \neq 0 \} = \mathcal O_s \cdot e_s$ by
Definition~\ref{defn:orderModule}. Thus the claim follows immediately from
Definition~\ref{defn:index} and Proposition~\ref{thm:orderIdealIndex}.
\end{proof}

Now we have all the ingredients to define module border bases in an analogous
way as border bases have been introduced in \cite[Defn.~6.4.10/13]{KR2}. Recall
that, for every $P$-module~$M$ (respectively $K$-vector space) and for every
$P$-submodule $U \subseteq M$ (respectively $K$-vector subspace),
\begin{align*}
\varepsilon_U: M \twoheadrightarrow M / U, \quad m \mapsto m + U
\end{align*}
denotes the canonical $P$-module epimorphism (respectively $K$-vector space
epimorphism).

\begin{defn}
\label{defn:moduleBB}
Let $\mathcal M = \mathcal O_1 e_1 \cup \cdots \cup \mathcal O_r e_r$ be an
order module where we let $\mathcal O_1, \ldots, \mathcal O_r \subseteq \mathbb
T^n$ be finite order ideals. We write $\mathcal M = \{ t_1 e_{\alpha_1}, \ldots,
t_\mu e_{\alpha_\mu} \}$ and $\partial \mathcal M = \{ b_1 e_{\beta_1}, \ldots,
b_\nu e_{\beta_\nu} \}$ with $\mu, \nu \in \mathbb N$, $t_i, b_j \in \mathbb
T^n$, and with $\alpha_i, \beta_j \in \{ 1, \ldots, r \}$ for all $i \in \{ 1,
\ldots, \mu \}$ and $j \in \{ 1, \ldots, \nu \}$.
\begin{enumerate}
  \item A set of vectors $\mathcal G = \{ \mathcal G_1, \ldots, \mathcal G_\nu
  \} \subseteq P^r$ is called an \emph{$\mathcal M$-module border prebasis} if
  the vectors have the form
  \begin{align*}
  \mathcal G_j = b_j e_{\beta_j} - \sum_{i=1}^\mu c_{ij} t_i e_{\alpha_i}
  \end{align*}
   with $c_{ij} \in K$ for all $i \in \{ 1, \ldots, \mu \}$ and $j \in \{ 1,
   \ldots, \nu \}$.
  \item Let $\mathcal G = \{ \mathcal G_1, \ldots, \mathcal G_\nu \} \subseteq
  P^r$ be an $\mathcal M$-module border prebasis and let $U \subseteq P^r$ be a
  $P$-submodule. We call~$\mathcal G$ an \emph{$\mathcal M$-module border basis}
  of~$U$ if $\mathcal G \subseteq U$, if the set
  \begin{align*}
  \varepsilon_{U}(\mathcal M) = \{ t_1 e_{\alpha_1} + U, \ldots, t_\mu
  e_{\alpha_\mu} + U \}
  \end{align*}
  is a $K$-vector space basis of $P^r / U$ and if $\# \varepsilon_{U}(\mathcal
  M) = \mu$.
\end{enumerate}
\end{defn}

For the remainder of this section, we let $\mathcal O_1, \ldots,
\mathcal O_r \subseteq \mathbb T^n$ be finite order ideals and
\begin{align*}
\mathcal M = \mathcal O_1 e_1 \cup \cdots \cup \mathcal O_r e_r = \{ t_1
e_{\alpha_1}, \ldots, t_\mu e_{\alpha_\mu} \}
\end{align*}
be an order module with border
\begin{align*}
\partial \mathcal M = \partial \mathcal O_1 e_1 \cup \cdots \cup \partial
\mathcal O_r e_r = \{ b_1 e_{\beta_1}, \ldots, b_\nu e_{\beta_\nu} \},
\end{align*}
where $\mu, \nu \in \mathbb N$, $t_i, b_j \in \mathbb T^n$ and $\alpha_i,
\beta_j \in \{ 1, \ldots, r \}$ for all $i \in \{ 1, \ldots, \mu \}$ and $j \in
\{ 1, \ldots, \nu \}$. Moreover, we let $\mathcal G = \{ \mathcal G_1, \ldots,
\mathcal G_\nu \} \subseteq P^r$ with
\begin{align*}
\mathcal G_j = b_j e_{\beta_j} - \sum_{i=1}^\mu c_{ij} t_i e_{\alpha_i},
\end{align*}
where $c_{ij} \in K$ for all $i \in \{ 1, \ldots, \mu \}$ and $j \in \{ 1,
\ldots, \nu \}$, be an $\mathcal M$-module border prebasis.

Next we generalize the Border Division Algorithm~\cite[Prop.~6.4.11]{KR2} to the
Module Border Division Algorithm and deduce some consequences of it similar to
the ones in \cite[Section~6.4.A]{KR2}.

\begin{algorithm}[H]
\caption{${\tt divAlg}(\mathcal V, \mathcal G)$}
\begin{algorithmic}[1]
\label{algo:divAlg}
  \REQUIRE $\mathcal V \in P^r$\\
  $\mathcal G = \{ \mathcal G_1, \ldots, \mathcal G_\nu \} \subseteq P^r$ is an
  $\mathcal M$-module border prebasis where we denote $\mathcal M = \{ t_1
  e_{\alpha_1}, \ldots, t_\mu e_{\alpha_\mu} \}$ and $\partial \mathcal M = \{
  b_1 e_{\beta_1}, \ldots, b_\nu e_{\beta_\nu} \}$ with $\mu, \nu \in \mathbb
  N$, $t_i, b_j \in \mathbb T^n$ and $\alpha_i, \beta_j \in \{ 1, \ldots, r \}$
  for $i \in \{ 1, \ldots, \mu \}$ and $j \in \{ 1, \ldots, \nu \}$.
  \STATE $(p_1, \ldots, p_\nu) \assign 0 \in P^\nu$\label{algo:divAlg-initStart}
  \STATE $(q_1, \ldots, q_r) \assign \mathcal V$\label{algo:divAlg-initEnd}
  \WHILE{$(q_1, \ldots, q_r) \neq 0$ and $\ind_{\mathcal M}((q_1, \ldots, q_r))
  > 0$}\label{algo:divAlg-while}
    \STATE \CHOOSE $t e_k \in \Supp((q_1, \ldots, q_r))$ with $\ind_{\mathcal
    M}(t e_k) = \ind_{\mathcal M}((q_1, \ldots,
    q_r))$.\label{algo:divAlg-choice}
    \STATE Determine the smallest index $j \in \{ 1, \ldots, \nu \}$ such that
    there exists a term $t' \in \mathbb T^n$ with $\deg(t') = \ind_{\mathcal
    M}((q_1, \ldots, q_r)) - 1$ and $t e_k = t' b_j
    e_{\beta_j}$.\label{algo:divAlg-factorize}
    \STATE Let $a \in K$ be the coefficient of $t e_k = t' b_j e_{\beta_j}$ in
    $(q_1, \ldots, q_r)$.\label{algo:divAlg-coeff}
    \STATE $p_j \assign p_j + a t'$\label{algo:divAlg-compensate}
    \STATE $(q_1, \ldots, q_r) \assign (q_1, \ldots, q_r) - a t' \mathcal
    G_j$\label{algo:divAlg-subtract}
  \ENDWHILE
  \IF{$(q_1, \ldots, q_r) = 0$}
    \RETURN $((p_1, \ldots, p_\nu), 0) \in P^\nu \times
    K^\mu$\label{algo:divAlg-trivialCase}
  \ENDIF
  \STATE Determine $c_1, \ldots, c_\mu \in K$ such that $(q_1, \ldots, q_r) =
  c_1 t_1 e_{\alpha_1} + \cdots + c_\mu t_\mu
  e_{\alpha_\mu}$.\label{algo:divAlg-ind0}
  \RETURN $((p_1, \ldots, p_\nu), (c_1, \ldots, c_\mu))$
\end{algorithmic}
\end{algorithm}

\begin{thm}[The Module Border Division Algorithm]
\label{thm:divAlg}
Let $\mathcal V \in P^r$. Then Algorithm~\ref{algo:divAlg} is actually an
algorithm, and the result
\begin{align*}
((p_1, \ldots, p_\nu), (c_1, \ldots, c_\mu)) \assign {\tt divAlg}(\mathcal
V, \mathcal G)
\end{align*}
of Algorithm~\ref{algo:divAlg} applied to the input data~$\mathcal V$
and~$\mathcal G$ satisfies the following conditions.
\begin{enumerate}
\renewcommand{\labelenumi}{\roman{enumi})}
  \item The result $((p_1, \ldots, p_\nu), (c_1, \ldots, c_\mu))$ is a tuple in
  $P^\nu \times K^\mu$ and it does not depend on the choice of the term~$t e_k$
  in line~\ref{algo:divAlg-choice}.
  \item We have
  \begin{align*}
  \mathcal V = p_1 \mathcal G_1 + \cdots + p_\nu \mathcal G_\nu + c_1 t_1
  e_{\alpha_1} + \cdots + c_\mu t_\mu e_{\alpha_\mu}.
  \end{align*}
  \item We have
  \begin{align*}
  \deg(p_j) \leq \ind_{\mathcal M}(\mathcal V) - 1
  \end{align*}
  for all $j \in \{ 1, \ldots, \nu \}$ with $p_j \neq 0$.
\end{enumerate}
\end{thm}

\begin{proof}
Firstly, we show that every step of the procedure can be computed. Let~$m$
always denote the current value of $\ind_{\mathcal M}((q_1, \ldots, q_r))$
during the procedure. The existence of a term $t e_k \in \Supp((q_1, \ldots,
q_r))$ such that $\ind_{\mathcal M}(t e_k) = m$ in line~\ref{algo:divAlg-choice}
follows from Definition~\ref{defn:index}. We now take a closer look at
line~\ref{algo:divAlg-factorize}. Since the while-loop in
line~\ref{algo:divAlg-while} is executed, we have $\ind_{\mathcal M}(t e_k) >
0$, i.\,e.  $t e_k \in \mathbb T^n \langle e_1, \ldots, e_r \rangle \setminus
\mathcal M$ by Definition~\ref{defn:index}. By Proposition~\ref{thm:index} there
is a factorization $t e_k = t' b_j e_{\beta_j}$ with a term $t' \in \mathbb T^n$
where $\deg(t') = \ind_{\mathcal M}(t e_k) - 1 = m - 1$ and an index $j \in \{
1, \ldots, \nu \}$. At last, since the while-loop is already finished in
line~\ref{algo:divAlg-ind0}, we have $m = 0$ and thus $\Supp((q_1, \ldots, q_r))
\subseteq \mathcal M$ according to Definition~\ref{defn:index}.
Altogether, we see that every step of the procedure can actually be computed.

Secondly, we prove termination. We show that the while-loop starting in
line~\ref{algo:divAlg-while} is executed only finitely many times. Taking a
closer look at the subtraction in line~\ref{algo:divAlg-subtract}, we see that
we subtract the vector
\begin{align*}
a t' \mathcal G_j = a t' b_j e_{\beta_j} - a t' \sum_{i=1}^\mu c_{ij} t_i
e_{\alpha_i}
\end{align*}
from $(q_1, \ldots, q_r)$. By the choice of $j \in \{ 1, \ldots, \nu \}$ and $t'
\in \mathbb T^n_{m-1}$ in line~\ref{algo:divAlg-factorize}, and the choice of $a
\in K$ in line~\ref{algo:divAlg-coeff}, it follows that the term $t e_k = t' b_j
e_{\beta_j}$ with $\mathcal M$-index~$m$ is replaced by terms of the form $t'
t_i e_{\alpha_i} \in \oline{\partial^{m-1} \mathcal M}$ with $i \in \{ 1,
\ldots, \mu \}$, which have strictly smaller $\mathcal M$-index according to
Proposition~\ref{thm:index}. The procedure hence terminates after finitely many
steps because there are only finitely many terms of $\mathcal M$-index smaller
than or equal to a given term. Altogether, we see that the procedure is actually
an algorithm.

We go on with the proof of the correctness. To this end, we show that the
equation
\begin{align*}
\mathcal V = p_1 \mathcal G_1 + \cdots + p_\nu \mathcal G_\nu + (q_1, \ldots,
q_r)
\end{align*}
is an invariant of the while-loop in line~\ref{algo:divAlg-while}. Before the
first iteration of the while-loop, we have $p_1 = \cdots = p_\nu = 0$ and $(q_1,
\ldots, q_r) = \mathcal V$, i.\,e.\ the invariant is obviously fulfilled. We now
regard the changes of $(p_1, \ldots, p_\nu) \in P^\nu$ and $(q_1, \ldots, q_r)
\in P^r$ during one iteration of the while-loop. Let $(p_1, \ldots, p_\nu) \in
P^\nu$ and $(q_1, \ldots, q_r) \in P^r$ be such that the invariant holds,
and let $(p'_1, \ldots, p'_\nu) \in P^\nu$ and $(q'_1, \ldots, q'_r) \in P^r$ be
the values of $(p_1, \ldots, p_\nu)$ and $(q_1, \ldots, q_r)$ after one
iteration of the while-loop. The values of the vectors $(p_1, \ldots, p_\nu) \in
P^\nu$ and $(q_1, \ldots, q_r) \in P^r$ are only changed in
line~\ref{algo:divAlg-compensate} and line~\ref{algo:divAlg-subtract}. Thus we
have $p_j' = p_j + a t'$, $p'_i = p_i$ for all $i \in \{ 1, \ldots, \nu \}
\setminus \{ j \}$, and $(q'_1, \ldots, q'_r) = (q_1, \ldots, q_r) - a t'
\mathcal G_j$. This yields
\begin{align*}
\mathcal V & = p_1 \mathcal G_1 + \cdots + p_\nu \mathcal G_\nu + (q_1, \ldots,
q_r)\\
& = p'_1 \mathcal G_1 + \cdots + p'_{j-1} \mathcal G_{j-1} + (p'_j - at')
\mathcal G_j + p'_{j+1} \mathcal G_{j+1} + \cdots + p'_\nu \mathcal G_\nu\\
& \alignLongFormula + ((q'_1, \ldots, q'_r) + a t' \mathcal G_j)\\
& = p'_1 \mathcal G_1 + \cdots p'_\nu \mathcal G_\nu + (q'_1, \ldots, q'_r),
\end{align*}
i.\,e.\ the invariant is also satisfied after one iteration of the while-loop.
By induction over the number of iterations of the while-loop, we see that the
invariant is always satisfied.\\
As we have already seen in the proof of the termination, the term $t'$ in
line~\ref{algo:divAlg-factorize}---and thus also the polynomials $p_1, \ldots,
p_\nu$ at the end of the algorithm---always has at most the degree
$\ind_{\mathcal M}(\mathcal V) - 1$. If the algorithm terminates in
line~\ref{algo:divAlg-trivialCase}, we have
\begin{align*}
\mathcal V = p_1 \mathcal G_1 + \cdots + p_\nu \mathcal G_\nu + (q_1, \ldots,
q_r) = p_1 \mathcal G_1 + \cdots + p_\nu \mathcal G_\nu.
\end{align*}
If the algorithm terminates in line~\ref{algo:divAlg-ind0}, we have
\begin{align*}
\mathcal V & = p_1 \mathcal G_1 + \cdots + p_\nu \mathcal G_\nu + (q_1,
\ldots, q_r)\\
& = p_1 \mathcal G_1 + \cdots + p_\nu \mathcal G_\nu + c_1 t_1 e_{\alpha_1} +
\cdots + c_\mu t_\mu e_{\alpha_\mu}
\end{align*}
with $c_1, \ldots, c_\mu \in K$. In both cases, the algorithm computes a
representation of~$\mathcal V$ with the claimed properties.

Finally, we prove that the result does not depend on the choice of the term~$t
e_k$ in line~\ref{algo:divAlg-choice}. This fact follows from the observation
that~$t e_k$ is replaced by terms of strictly smaller $\mathcal M$-index during
every iteration of the while-loop in line~\ref{algo:divAlg-while}. Thus 
different choices of the term~$t e_k$ in line~\ref{algo:divAlg-choice} do not
interfere with one another. Altogether, we see that the final result, after all
those terms have been rewritten, is independent of the ordering in which they
are handled.
\end{proof}

Using the Module Border Division Algorithm~\ref{thm:divAlg}, we can define the
normal remainder of a vector in~$P^r$ with respect to a module border prebasis
in the usual way like it has been done for border prebases in
\cite[p.~427]{KR2}.

\begin{defn}
\label{defn:NR}
Let $\mathcal V \in P^r$. We apply the Module Border Division
Algorithm~\ref{thm:divAlg} to~$\mathcal V$ and~$\mathcal G$ to obtain a
representation
\begin{align*}
\mathcal V = p_1 \mathcal G_1 + \cdots + p_\nu \mathcal G_\nu + c_1 t_1
e_{\alpha_1} + \cdots + c_\mu t_\mu e_{\alpha_\mu}
\end{align*}
with $p_1, \ldots, p_\nu \in P$, $c_1, \ldots, c_\mu \in K$, and
\begin{align*}
\deg(p_j) \leq \ind_{\mathcal M}(\mathcal V) - 1
\end{align*}
for all $j \in \{ 1, \ldots, \nu \}$ where $p_j \neq 0$. Then the vector
\begin{align*}
\NR_{\mathcal G}(\mathcal V) = c_1 t_1 e_{\alpha_1} + \cdots + c_\mu t_\mu
e_{\alpha_\mu} \in \langle \mathcal M \rangle_K \subseteq P^r
\end{align*}
is called the \emph{normal remainder} of~$\mathcal V$ with respect
to~$\mathcal G$.
\end{defn}

\begin{exmp}
\label{exmp:divAlg}
We consider Example~\ref{exmp:border}, again, and write the order module
\begin{align*}
\mathcal M = \{ x e_1, y e_1, e_1, x^2 e_2, x e_2, e_2 \} = \{ t_1 e_{\alpha_1},
\ldots, t_6 e_{\alpha_6} \}
\end{align*}
and its border
\begin{align*}
\partial \mathcal M = \{ x^2 e_1, xy e_1, y^2e_1, x^3 e_2, x^2y e_2, xy e_2, y
e_2 \} = \{ b_1 e_{\beta_1}, \ldots, b_7 e_{\beta_7} \}.
\end{align*}
Then the set
$\mathcal G = \{ \mathcal G_1, \ldots, \mathcal G_7 \} \subseteq P^2$ with
\begin{align*}
\mathcal G_1 & = x^2 e_1 - y e_1 + e_2,\\
\mathcal G_2 & = xy e_1 - e_2,\\
\mathcal G_3 & = y^2 e_1 - x e_2,\\
\mathcal G_4 & = x^3 e_2 - e_1,\\
\mathcal G_5 & = x^2y e_2 - e_1 - e_2,\\
\mathcal G_6 & = xy e_2 + 3 e_1,\\
\mathcal G_7 & = y e_2 - x e_1 - y e_1 - e_1 - e_2
\end{align*}
is an $\mathcal M$-module border prebasis by Definition~\ref{defn:moduleBB}. We
consider the steps of the Module Border Division Algorithm~\ref{algo:divAlg}
applied to
\begin{align*}
\mathcal V = x^3 e_1 + xy e_1 + x^3y e_2 \in P^2
\end{align*}
and~$\mathcal G$ in detail.

The initialization process of the algorithm in the
lines~\ref{algo:divAlg-initStart}--\ref{algo:divAlg-initEnd} yields
\begin{align*}
(p_1, \ldots, p_7) = (0, 0, 0, 0, 0, 0, 0)
\end{align*}
and
\begin{align*}
(q_1, q_2) = (x^3 + xy, x^3y).
\end{align*}
Since
\begin{align*}
\ind_{\mathcal M}((q_1, q_2)) = \ind_{\mathcal M}(x^3 e_1) = \ind_{\mathcal
M}(x^3 y e_2) = 2 > 0,
\end{align*}
the while-loop in line~\ref{algo:divAlg-while} is executed. We choose~$x^3y e_2$
in line~\ref{algo:divAlg-choice}. Then we have $j = 4$ and the factorization
\begin{align*}
x^3y e_2 = y \cdot x^3 e_2 = y \cdot b_4 e_{\beta_4}
\end{align*}
in line~\ref{algo:divAlg-factorize}. After line~\ref{algo:divAlg-compensate}, we
have
\begin{align*}
(p_1, \ldots, p_7) = (0, 0, 0, y, 0, 0, 0),
\end{align*}
and after line~\ref{algo:divAlg-subtract}, we have
\begin{align*}
(q_1, q_2) = (x^3 + xy, x^3y) - y \cdot (-1, x^3) = (x^3 + xy + y, 0).
\end{align*}
Now the $\mathcal M$-index is
\begin{align*}
\ind_{\mathcal M}((q_1, q_2)) = \ind_{\mathcal M}(x^3 e_1) = 2 > 0,
\end{align*}
and we must choose~$x^3 e_1$ in line~\ref{algo:divAlg-choice}. Then we have $j
= 1$ and the factorization
\begin{align*}
x^3 e_1 = x \cdot x^2 e_1 = x \cdot b_1 e_{\beta_1}
\end{align*}
in line~\ref{algo:divAlg-factorize}. This yields
\begin{align*}
(p_1, \ldots, p_7) = (x, 0, 0, y, 0, 0, 0)
\end{align*}
after line~\ref{algo:divAlg-compensate}, and
\begin{align*}
(q_1, q_2) = (x^3 + xy + y, 0) - x \cdot (x^2 - y, 1) = (2xy + y, -x)
\end{align*}
after line~\ref{algo:divAlg-subtract}. Now the $\mathcal M$-index has decreased
to
\begin{align*}
\ind_{\mathcal M}((q_1, q_2)) = \ind_{\mathcal M}(xy e_1) = 1 > 0,
\end{align*}
and we must choose~$xy e_1$ in line~\ref{algo:divAlg-choice}. Then we have $j
= 2$ and the factorization
\begin{align*}
xy e_1 = 1 \cdot xy e_1 = 1 \cdot b_2 e_{\beta_2}
\end{align*}
in line~\ref{algo:divAlg-factorize}. This yields
\begin{align*}
(p_1, \ldots, p_7) = (x, 2, 0, y, 0, 0, 0)
\end{align*}
after line~\ref{algo:divAlg-compensate}, and
\begin{align*}
(q_1, q_2) = (2xy + y, -x) - 2 \cdot (xy, -1) = (y, -x + 2).
\end{align*}
after line~\ref{algo:divAlg-subtract}. After these iterations, we see that
\begin{align*}
(q_1, q_2) \neq 0
\end{align*}
and
\begin{align*}
\ind_{\mathcal M}((q_1, q_2)) = 0.
\end{align*}
Since
\begin{align*}
(q_1, q_2) = (y, -x + 2) = t_2 e_{\alpha_2} - t_5 e_{\alpha_5} + 2 t_6
e_{\alpha_6},
\end{align*}
the algorithm returns
\begin{align*}
((x, 2, 0, y, 0, 0, 0), (0, 1, 0, 0, -1, 2)) \in P^7 \times K^6
\end{align*}
in line~\ref{algo:divAlg-ind0}.

Moreover, this yields
\begin{align*}
\mathcal V = x^3 e_1 + xy e_1 + x^3y e_2 = x \mathcal G_1 + 2 \mathcal G_2 + y
\mathcal G_4 + \NR_{\mathcal G}(\mathcal V)
\end{align*}
where
\begin{align*}
\NR_{\mathcal G}(\mathcal V) = y e_1 - x e_2 + 2 e_2 \in \langle \mathcal M
\rangle_K.
\end{align*}
according to Theorem~\ref{thm:divAlg} and Definition~\ref{defn:NR}.
\end{exmp}

Now we use the Module Border Division Algorithm~\ref{thm:divAlg} to deduce some
propositions about module border bases. We start to prove that an $\mathcal
M$-module border basis $\mathcal G$ of a $P$-submodule $U \subseteq P^r$ is
indeed a basis, i.\,e.\ that $\langle \mathcal G \rangle = U$ holds. The proof
of this result follows exactly like the corresponding result for border bases in
\cite[Prop.~6.4.15]{KR2}.

\begin{cor}
\label{thm:genSetModBB}
Let $\mathcal G$ be an $\mathcal M$-module border basis of a $P$-submodule $U
\subseteq P^r$. Then $\langle \mathcal G \rangle = U$.
\end{cor}

\begin{proof}
According to Definition~\ref{defn:moduleBB}, we have $\langle \mathcal G \rangle
\subseteq U$. For the converse implication, we let $\mathcal V \in U$. We apply
the Module Border Basis Algorithm~\ref{thm:divAlg} to~$\mathcal V$ and~$\mathcal
G$ to obtain a representation
\begin{align*}
\mathcal V = \mathcal W + c_1 t_1 e_{\alpha_1} + \cdots + c_\mu t_\mu
e_{\alpha_\mu}
\end{align*}
with $\mathcal W \in \langle \mathcal G \rangle \subseteq U$, and $c_1, \ldots,
c_\mu \in K$.
It follows that
\begin{align*}
U = \mathcal V + U = c_1 t_1 e_{\alpha_1} + \cdots + c_\mu t_\mu e_{\alpha_\mu}
+ U \in P^r / U.
\end{align*}
Since $\mathcal G$ is an $\mathcal M$-module border basis of $U$,
Definition~\ref{defn:moduleBB} yields
\begin{align*}
c_1 = \cdots = c_\mu = 0
\end{align*}
and thus $\mathcal V = \mathcal W \in \langle \mathcal G \rangle$.
\end{proof}

The proof of the next corollary follows the proof of \cite[Coro.~3.8]{SubBB}.

\begin{cor}
\label{thm:genSet}
We have $\langle \varepsilon_{\langle \mathcal G \rangle}(\mathcal M) \rangle_K
= P^r / \langle \mathcal G \rangle$. In particular, for every vector $\mathcal V
\in P^r$, the normal remainder $\NR_{\mathcal G}(\mathcal V)$ of~$\mathcal V$
with respect to~$\mathcal G$ is a representative of the residue class~$\mathcal
V + \langle \mathcal G \rangle$.
\end{cor}

\begin{proof}
Let $\mathcal V \in P^r$. We apply the Module Border Division
Algorithm~\ref{thm:divAlg} to~$\mathcal V$ and~$\mathcal G$ to obtain a
representation
\begin{align*}
\mathcal V = p_1 \mathcal G_1 + \cdots + p_\nu \mathcal G_\nu + \NR_{\mathcal
G}(\mathcal V)
\end{align*}
with $p_1, \ldots, p_\nu \in P$ and $\NR_{\mathcal G}(\mathcal V) \in \langle
\mathcal M \rangle_K$. Thus we have
\begin{align*}
\mathcal V + \langle \mathcal G \rangle = \NR_{\mathcal G}(\mathcal V) + \langle
\mathcal G \rangle \in \langle \varepsilon_{\langle \mathcal G \rangle}(\mathcal
M) \rangle_K.
\end{align*}
The other inclusion follows trivially because~$\mathcal M \subseteq P^r$.
\end{proof}

We are now able to give a first characterization of module border bases similar
to the characterization of border bases in \cite[Defn.~6.4.13]{KR2}.

\begin{cor}
\label{thm:char}
Let $U \subseteq P^r$ be a $P$-submodule with $\mathcal G \subseteq U$. Then the
following conditions are equivalent.
\begin{enumerate}
\renewcommand{\labelenumi}{\roman{enumi})}
  \item The $\mathcal M$-module border prebasis~$\mathcal G$ is an $\mathcal
  M$-module border basis of~$U$.
  \item We have $U \cap \langle \mathcal M \rangle_K = \{ 0 \}$.
  \item We have $P^r = U \oplus \langle \mathcal M \rangle_K$.
\end{enumerate}
\end{cor}

\begin{proof}
We start to prove that~i) implies~ii). Let $\mathcal V \in U \cap \langle
\mathcal M \rangle_K$. Then there exist $c_1, \ldots, c_\mu \in K$ such that
\begin{align*}
\mathcal V = c_1 t_1 e_{\alpha_1} + \cdots + c_\mu t_\mu e_{\alpha_\mu}.
\end{align*}
Modulo $U$, this yields
\begin{align*}
U = \mathcal V + U = c_1 t_1 e_{\alpha_1} + \cdots + c_\mu t_\mu e_{\alpha_\mu}
+ U \in P^r / U.
\end{align*}
As~$\mathcal G$ is an $\mathcal M$-module border basis of~$U$, it follows $c_1 =
\cdots = c_\mu = 0$ and thus $\mathcal V = 0$ by Definition~\ref{defn:moduleBB}.

Next we prove that~ii) implies~iii). As we have $U \cap \langle \mathcal M
\rangle_K = \{ 0 \}$, it suffices to prove that $P^r = U + \langle \mathcal M
\rangle_K$. Obviously we have $P^r \supseteq U + \langle \mathcal M \rangle_K$.
In order to prove the converse inclusion, we let $\mathcal V \in P^r$. Applying
the Module Border Division Algorithm~\ref{thm:divAlg} to~$\mathcal V$
and~$\mathcal G$, we obtain a representation
\begin{align*}
\mathcal V = p_1 \mathcal G_1 + \cdots + p_\nu \mathcal G_\nu + \NR_{\mathcal
G}(\mathcal V) \in \langle \mathcal G \rangle + \langle \mathcal M \rangle_K
\end{align*}
with $p_1, \ldots, p_\nu \in P$ and $\NR_{\mathcal G}(\mathcal V) \in \langle
\mathcal M \rangle_K$. The hypothesis $\mathcal G \subseteq U$ hence yields the
claim.

Finally, we prove that~iii) implies~i). Let $c_1, \ldots, c_\mu \in K$ be
coefficients such that
\begin{align*}
U = c_1 t_1 e_{\alpha_1} + \cdots + c_\mu t_\mu e_{\alpha_\mu} + U \in P^r / U.
\end{align*}
Then we have
\begin{align*}
c_1 t_1 e_{\alpha_1} + \cdots + c_\mu t_\mu e_{\alpha_\mu} \in U \cap \langle
\mathcal M \rangle_K
\end{align*}
which yields
\begin{align*}
c_1 t_1 e_{\alpha_1} + \cdots + c_\mu t_\mu e_{\alpha_\mu} = 0
\end{align*}
because of $P^r = U \oplus \langle \mathcal M \rangle_K$. As $\mathcal M$ is
$K$-linearly independent, it follows that $c_1 = \cdots = c_\mu = 0$, and we see
that $\varepsilon_U(\mathcal M) \subseteq P^r / U$ is $K$-linearly independent
and that $\# \varepsilon_U(\mathcal M) = \mu$. Moreover, every vector $\mathcal
V \in P^r$ can be written in the form $\mathcal V = \mathcal W + \NR_{\mathcal
G}(\mathcal V)$ where $\mathcal W \in \langle \mathcal G \rangle$ according to
Corollary~\ref{thm:genSet}. As $\mathcal G \subseteq U$, the claim now follows
by Definition~\ref{defn:moduleBB}.
\end{proof}

If~$\mathcal G$ is an $\mathcal M$-module border bases of a $P$-submodule $U
\subseteq P^r$, we can introduce a normal form with respect to a $P$-submodule
$U \subseteq P^r$ for every vector in $P^r$ similarly to the way it has been
done in \cite[Defn.~6.4.20]{KR2} and \cite[Prop.~6.4.19/21]{KR2} for border
bases.

\begin{lem}
\label{thm:NFunique}
Let $U \subseteq P^r$ be a $P$-submodule, and let $\mathcal G$ and $\mathcal G'$
be two $\mathcal M$-module border bases of~$U$. Then we have
\begin{align*}
\NR_{\mathcal G}(\mathcal V) = \NR_{\mathcal G'}(\mathcal V)
\end{align*}
for every vector~$\mathcal V \in P^r$.
\end{lem}

\begin{proof}
Let $\mathcal V \in P^r$. We apply the Module Border Division
Algorithm~\ref{thm:divAlg} to~$\mathcal V$ and~$\mathcal G$, and to~$\mathcal V$
and~$\mathcal G'$ in order to obtain two representations
\begin{align*}
\mathcal V = \mathcal W + \NR_{\mathcal G}(\mathcal V) = \mathcal W' +
\NR_{\mathcal G'}(\mathcal V)
\end{align*}
with $\mathcal W, \mathcal W' \in \langle \mathcal G \rangle$ and $\NR_{\mathcal
G}(\mathcal V), \NR_{\mathcal G'}(\mathcal V) \in \langle \mathcal M \rangle_K$.
As~$\mathcal G$ and~$\mathcal G'$ are $\mathcal M$-module border bases of~$U$,
the claim follows since we have
\begin{align*}
\NR_{\mathcal G}(\mathcal V) - \NR_{\mathcal G'}(\mathcal V) = \mathcal W' -
\mathcal W \in \langle \mathcal G \rangle \cap \langle \mathcal M \rangle_K
\subseteq U \cap \langle \mathcal M \rangle_K = \{ 0 \}
\end{align*}
according to Definition~\ref{defn:moduleBB} and Corollary~\ref{thm:char}.
\end{proof}

\begin{rem}
\label{rem:NF}
Let $\mathcal V \in P^r$ be a vector. Similar to the situation of Gröbner bases
and border bases, the normal remainder of~$\mathcal V$ with respect to the
$\mathcal M$-module border prebasis~$\mathcal G$ is a representative of the
residue class $\mathcal V + \langle \mathcal G \rangle \in P^r / \langle
\mathcal G \rangle$ by Corollary~\ref{thm:genSet}. But the normal remainder
of~$\mathcal V$ with respect to~$\mathcal G$ depends on the particularly chosen
$\mathcal M$-module border prebasis~$\mathcal G$ and on the ordering of the
elements in~$\mathcal G$ by Definition~\ref{defn:NR}. But if~$\mathcal G$ is even
an $\mathcal M$-module border basis of~$\langle \mathcal G \rangle$,
Lemma~\ref{thm:NFunique} shows that the result is independent of the
particularly chosen $\mathcal M$-module border basis~$\mathcal G$, and of the
ordering of the elements in~$\mathcal G$. Thus the normal remainder defines a
normal form with respect to~$\langle \mathcal G \rangle$ in this situation. In
particular, we can also compute this normal form with the Module Border Division
Algorithm~\ref{thm:divAlg}.
\end{rem}

\begin{defn}
\label{defn:NF}
Let~$\mathcal G$ be an $\mathcal M$-module border basis of a $P$-submodule $U
\subseteq P^r$, and let $\mathcal V \in P^r$. Then we call the vector
\begin{align*}
\NF_{\mathcal M, U}(\mathcal V) = \NR_{\mathcal G}(\mathcal V) \in \langle
\mathcal M \rangle_K \subseteq P^r,
\end{align*}
which is unique according to Remark~\ref{rem:NF}, the \emph{normal form}
of~$\mathcal V$ with respect to~$\mathcal M$ and~$U$.
\end{defn}

\begin{prop}
Let~$\mathcal G$ be an $\mathcal M$-module border basis of a $P$-submodule $U
\subseteq P^r$.
\begin{enumerate}
  \item For all $\mathcal V \in P^r$, we have
  \begin{align*}
  \NF_{\mathcal M, U}(\mathcal V) = \NF_{\sigma, U}(\mathcal V)
  \end{align*}
  if there exists a term ordering~$\sigma$ on~$\mathbb T^n \langle e_1, \ldots,
  e_r \rangle$ such that~$\mathcal M = \mathcal O_\sigma(U)$.
  \item We have
  \begin{align*}
  \NF_{\mathcal M, U}(c \mathcal V + c' \mathcal V') = c \NF_{\mathcal M,
  U}(\mathcal V) + c' \NF_{\mathcal M, U}(\mathcal V')
  \end{align*}
  for all $c, c' \in K$ and all~$\mathcal V, \mathcal V' \in P^r$.
  \item We have
  \begin{align*}
  \NF_{\mathcal M, U}(\NF_{\mathcal M, U}(\mathcal V)) = \NF_{\mathcal M,
  U}(\mathcal V)
  \end{align*}
  for all~$\mathcal V \in P^r$.
  \item We have
  \begin{align*}
  \NF_{\mathcal M, U}(p \mathcal V) = \NF_{\mathcal M, U}(p \NF_{\mathcal M,
  U}(\mathcal V))
  \end{align*}
  for all $p \in P$ and all~$\mathcal V \in P^r$.
\end{enumerate}
\end{prop}

\begin{proof}
Claim~a) follows because for all $\mathcal V \in P^r$, both $\NF_{\mathcal M,
U}(\mathcal V)$ and $\NF_{\sigma, U}(\mathcal V)$ are equal to the unique vector
in $\mathcal V + U \in P^r / U$ whose support is contained in $\mathcal M =
\mathcal O_\sigma(U)$ according to Definition~\ref{defn:NF} and
\cite[Defn.~2.4.8]{KR1}. The other claims follow from the same uniqueness.
\end{proof}

%
%

\section{Existence and Uniqueness of Module Border Bases}
\label{sect:existUnique}

In this section, we show the existence and uniqueness of a module border bases
of a given submodule $U \subseteq P^r$ like it has been done for border bases in
\cite[Prop.~6.4.17]{KR2}. Moreover, we proof that there is a correspondence of
the reduced $\sigma$-Gröbner bases of a $P$-submodule $U \subseteq P^r$ and the
$\mathcal O_\sigma(U)$-module border bases of~$U$ similar to
\cite[Prop.~6.4.18]{KR2}. Finally, we give a first \enquote{naive} algorithm for
the computation of module border bases that uses Gröbner bases techniques.

Like in the previous section, we let $\mathcal O_1, \ldots, \mathcal O_r
\subseteq \mathbb T^n$ be finite order ideals, and we let $\mathcal M = \mathcal
O_1 e_1 \cup \cdots \cup \mathcal O_r e_r = \{ t_1 e_{\alpha_1}, \ldots, t_\mu
e_{\alpha_\mu} \}$ be an order module with border $\partial \mathcal M = \{ b_1
e_{\beta_1}, \ldots, b_\nu e_{\beta_\nu} \}$, where $\mu, \nu \in \mathbb N$,
$t_i, b_j \in \mathbb T^n$ and $\alpha_i, \beta_j \in \{ 1, \ldots, r \}$ for
all $i \in \{ 1, \ldots, \mu \}$ and $j \in \{ 1, \ldots, \nu \}$. Furthermore,
we let $\mathcal G = \{ \mathcal G_1, \ldots, \mathcal G_\nu \} \subseteq P^r$
be an $\mathcal M$-module border prebasis, where we write $\mathcal G_j = b_j
e_{\beta_j} - \sum_{i=1}^\mu c_{ij} t_i e_{\alpha_i}$ with $c_{1j}, \ldots,
c_{\mu j} \in K$ for every $j \in \{ 1, \ldots, \nu \}$.

We start with the proof of the existence and uniqueness of module border bases.
To this end, we imitate the proof of~\cite[Prop.~6.4.17]{KR2}.

\begin{prop}[Existence and Uniqueness of Module Border Bases]
\label{thm:existUnique}
Let $U \subseteq P^r$ be a $P$-submodule. Moreover, we let
$\varepsilon_U(\mathcal M) \subseteq P^r / U$ be a $K$-vector space basis
of~$P^r / U$ with $\# \varepsilon_U(\mathcal M) = \mu$.
\begin{enumerate}
  \item There exists a unique $\mathcal M$-module border basis of~$U$.
  \item Let $\mathcal G$ be an $\mathcal M$-module border prebasis with
  $\mathcal G \subseteq U$. Then $\mathcal G$ is the $\mathcal M$-module border
  basis of~$U$.
  \item Let $K'$ be the field of definition of~$U$. Then the $\mathcal M$-module
  border basis of~$U$ is contained in $K'[x_1, \ldots, x_n]$.
\end{enumerate}
\end{prop}

\begin{proof}
We start with the proof of claim~a). Let $j \in \{ 1, \ldots, \nu \}$. By
Proposition~\ref{thm:border} we see that $b_j e_{\beta_j} \notin \mathcal M$.
Thus there exist $c_{1j}, \ldots, c_{\mu j} \in K$ such that
\begin{align*}
b_j e_{\beta_j} + U = c_{1j} t_1 e_{\alpha_1} + \cdots + c_{\mu j} t_\mu
e_{\alpha_\mu} + U
\end{align*}
and this yields
\begin{align*}
\mathcal G_j = b_j e_{\beta_j} - \sum_{i=1}^\mu c_{ij} t_i e_{\alpha_i} \in U.
\end{align*}
Now the set $\mathcal G = \{ \mathcal G_1, \ldots, \mathcal G_\nu \} \subseteq
U$ is an $\mathcal M$-module border prebasis, and the assumptions yield that
$\mathcal G$ is an $\mathcal M$-module border basis of~$U$ by
Definition~\ref{defn:moduleBB}. It remains to prove the uniqueness. Let
$\mathcal G' = \{ \mathcal G'_1, \ldots, \mathcal G'_\nu \} \subseteq U$ be
another $\mathcal M$-module border basis of~$U$ where
\begin{align*}
\mathcal G'_j = b_j e_{\beta_j} - \sum_{i=1}^\mu c'_{ij} t_i e_{\alpha_i}
\end{align*}
with $c'_{ij} \in K$ for all $i \in \{ 1, \ldots, \mu \}$ and $j \in \{ 1,
\ldots, \nu \}$. Assume there exists an $i \in \{ 1, \ldots, \mu \}$ and a $j
\in \{ 1, \ldots, \nu \}$ such that $c_{ij} \neq c'_{ij}$. Then
Corollary~\ref{thm:char} yields the contradiction
\begin{align*}
0 \neq \mathcal G_j - \mathcal G'_j \in U \cap \langle \mathcal M \rangle_K = \{
0 \}
\end{align*}
and claim~a) follows.

We go on with the proof of~b). As $\varepsilon_U(\mathcal M)$ is a $K$-vector
space basis of $P^r / U$ such that $\# \varepsilon_U(\mathcal M) = \mu$, we see
that $\mathcal G$ is an $\mathcal M$-module border basis of $U$ by
Definition~\ref{defn:moduleBB}. The claim now follows with~a).

Finally, we prove claim~c). Let $P' = K'[x_1, \ldots, x_n]$ and $U' = U \cap
(P')^r$. Given a term ordering~$\sigma$ on~$\mathbb T^n \langle e_1, \ldots, e_r
\rangle$, the $P$-submodules $U \subseteq P^r$ and $U' \subseteq (P')^r$ have
the same reduced $\sigma$-Gröbner basis and
\begin{align*}
\LT_\sigma \{ U \} = \LT_\sigma \{ U' \}
\end{align*}
by \cite[Lemma~2.4.16]{KR1}. Hence we see that
\begin{align*}
\dim_K((P')^r / U') = \dim_K(P^r / U) = \# \mathcal M
\end{align*}
according to Macaulay's Basis Theorem~\cite[Thm.~1.5.7]{KR1} and
Definition~\ref{defn:moduleBB}. Moreover, the elements of $\mathcal M$ are
contained in $(P')^r$ and they are $K$-linearly independent modulo $U'
\subseteq U$. Thus it follows that $\varepsilon_{U'}(\mathcal M)$ is a
$K$-vector space basis of~$(P')^r / U'$ and $\# \varepsilon_{U'}(\mathcal M) =
\mu$. According to~a), there exists a unique $\mathcal M$-module border basis
$\mathcal G' \subseteq P^r$ of~$U'$. Since~$\mathcal G'$ is an $\mathcal
M$-module border prebasis with $\mathcal G' \subseteq U$, the claim follows
from~b).
\end{proof}

Next we show that there is a connection between certain module border bases and
reduced Gröbner bases like it has been done in \cite[Prop.~6.4.18]{KR2} for the
border bases case.

The set $\mathbb T^n \langle e_1, \ldots, e_r \rangle \setminus \mathcal M$ is
obviously a monomial submodule of $P^r$, i.\,e.\ it has a system of generators
consisting of terms by \cite[Defn.~1.3.7]{KR1}. Thus there exists a uniquely
determined minimal set of generators of this monomial submodule according to
\cite[Prop.~1.3.11]{KR1}. The elements between this minimal generating set play
an essential role in the connection of certain module border bases and reduced
Gröbner bases, and thus get a name.

\begin{defn}
\label{defn:corners}
Let $\mathcal M$ be an order module. We call the elements of the minimal
generating set of the monomial submodule $\mathbb T^n \langle e_1, \ldots, e_r
\rangle \setminus \mathcal M$ the \emph{corners} of~$\mathcal M$.
\end{defn}

\begin{prop}
\label{thm:corners}
Let $U \subseteq P^r$ be a $P$-submodule and let $\sigma$ be a term ordering
on~$\mathbb T^n \langle e_1, \ldots, e_r \rangle$. Then there exists a unique
$\mathcal O_{\sigma}(U)$-module border basis~$\mathcal G$ of~$U$, and the
reduced $\sigma$-Gröbner basis of~$U$ is the subset of~$\mathcal G$
corresponding to the corners of~$\mathcal O_{\sigma}(U)$.
\end{prop}

\begin{proof}
By Macaulay's Basis Theorem~\cite[Thm.~1.5.7]{KR1}, the set
$\varepsilon_U(\mathcal O_\sigma(U))$ is a $K$-vector space basis of~$P^r / U$
and we have $\# \varepsilon_U(\mathcal O_\sigma(U)) = \# \mathcal O_\sigma(U)$.
Thus Proposition~\ref{thm:existUnique} implies the existence of a unique
$\mathcal O_\sigma(U)$-module border basis~$\mathcal G$ of~$U$.

In order to prove the second claim, we let $b e_k \in \LT_\sigma \{ U \}$ with
$k \in \{ 1, \ldots, r \}$ be a corner of~$\mathcal O_\sigma(U)$. The element of
the reduced $\sigma$-Gröbner basis of~$U$ with leading term~$b e_k$ has the form
$b e_k - \NF_{\sigma, U}(b e_k)$ where $\NF_{\sigma, U}(b e_k) \subseteq \langle
\mathcal O_\sigma(U) \rangle_K$ by \cite[Defn.~2.4.8]{KR1}. Since the $\mathcal
O_\sigma(U)$-module border basis~$\mathcal G$ of~$U$ is unique, this element of
the reduced $\sigma$-Gröbner basis coincides with the element in~$\mathcal G$
corresponding to~$b e_k$. According to the definition of the reduced
$\sigma$-Gröbner basis of~$U$, cf.\ \cite[Defn.~2.4.12]{KR1}, the corners of
$\mathcal O_\sigma(U)$ are exactly the leading terms of the elements of the
reduced $\sigma$-Gröbner basis of~$U$ and hence the claim follows.
\end{proof}

\begin{rem}
\label{rem:naiveAlgo}
The proof of Proposition~\ref{thm:corners} gives rise to an algorithm for the
computation of a module border basis of a given $P$-submodule $U \subseteq P^r$
with $\codim_K(U, P^r) < \infty$. Let $\sigma$ be any term ordering on $\mathbb
T^n \langle e_1, \ldots, e_r \rangle$. We first have to compute a
$\sigma$-Gröbner basis $\mathcal H$ of~$P^r / U$ to determine the order module
$\mathcal O_{\sigma}(U)$ with Macaulay's Basis Theorem~\cite[Thm.~1.5.7]{KR1}.
Then we have to compute its border
\begin{align*}
\partial \mathcal O_{\sigma}(U) = \{ b_1 e_{\beta_1}, \ldots, b_\nu
e_{\beta_\nu} \}.
\end{align*}
Using $\mathcal H$ and the Division Algorithm for Gröbner
Bases~\cite[Thm.~1.6.4]{KR1}, we then compute
\begin{align*}
\mathcal G_j = b_j e_{\beta_j} - \NF_{\sigma,U}(b_j e_{\beta_j}) = b_j
e_{\beta_j} - \NR_{\sigma, \mathcal H}(b_j e_{\beta_j}) \in U \subseteq P^r
\end{align*}
for all $j \in \{ 1, \ldots, \nu \}$ and see that $\mathcal G = \{ \mathcal
G_1, \ldots, \mathcal G_\nu \} \subseteq P^r$ is the $\mathcal
O_\sigma(U)$-module border basis of~$U$ according to
Proposition~\ref{thm:existUnique}.
\end{rem}

In section~\ref{sect:moduleBBAlg}, we give a more efficient algorithm for the
computation of module border bases which uses linear algebra, and which is an
analogous version of the Border Basis Algorithm in~\cite[Thm.~6.4.36]{KR2}.

%
%

\section{Characterizations of Module Border Bases}
\label{sect:chars}

In this section, we want to characterize module border bases in an analogous way
as border bases have been characterized in \cite[Section~6.4.B]{KR2} and in
\cite{CharBB}. In particular, we characterize module border bases via the
special generation property in Theorem~\ref{thm:specGen}, via border form
modules in Theorem~\ref{thm:BFMod}, via rewrite rules in
Theorem~\ref{thm:rewrite}, via commuting matrices in Theorem~\ref{thm:commMat},
and via liftings of border syzygies in Theorem~\ref{thm:liftings}. At last, we
proof Buchberger's Criterion for Module Border Bases~\ref{thm:buchbCrit}.
This is the main result of this section as it allows us to check easily,
whether a given module border prebasis is a module border bases, or not.

Like in the previous section, we let $\mathcal O_1, \ldots, \mathcal O_r
\subseteq \mathbb T^n$ be finite order ideals and we let $\mathcal M = \mathcal
O_1 e_1 \cup \cdots \cup \mathcal O_r e_r = \{ t_1 e_{\alpha_1}, \ldots, t_\mu
e_{\alpha_\mu} \}$ be an order module with border $\partial \mathcal M = \{ b_1
e_{\beta_1}, \ldots, b_\nu e_{\beta_\nu} \}$ where $\mu, \nu \in \mathbb N$,
$t_i, b_j \in \mathbb T^n$ and $\alpha_i, \beta_j \in \{ 1, \ldots, r \}$ for
all $i \in \{ 1, \ldots, \mu \}$ and $j \in \{ 1, \ldots, \nu \}$. Furthermore,
we let $\mathcal G = \{ \mathcal G_1, \ldots, \mathcal G_\nu \} \subseteq P^r$
be an $\mathcal M$-module border basis, where we write $\mathcal G_j = b_j
e_{\beta_j} - \sum_{i=1}^\mu c_{ij} t_i e_{\alpha_i}$ with $c_{1j}, \ldots,
c_{\mu j} \in K$ for every $j \in \{ 1, \ldots, \nu \}$.

%
%

We start to show the connection of module border bases and the special
generation property which has originally been proven for border bases in
\cite[Prop.~9]{CharBB}. Our proof follows the corresponding version for border
bases in~\cite[Prop.~6.4.23]{KR2}.

\begin{thm}[Module Border Bases and Special Generation]
\label{thm:specGen}
The $\mathcal M$-module border prebasis $\mathcal G$ is the $\mathcal M$-module
border basis of $\langle \mathcal G \rangle$ if and only if the following
equivalent conditions are satisfied.
\begin{enumerate}
\renewcommand{\labelenumi}{$A_{\arabic{enumi}})$}
  \item For every vector $\mathcal V \in \langle \mathcal G \rangle \setminus \{
  0 \}$, there exist polynomials $p_1, \ldots, p_\nu \in P$ such that
  \begin{align*}
  \mathcal V = p_1 \mathcal G_1 + \cdots + p_\nu \mathcal G_\nu
  \end{align*}
  and
  \begin{align*}
  \deg(p_j) \leq \ind_{\mathcal M}(\mathcal V) - 1
  \end{align*}
  for all $j \in \{ 1, \ldots, \nu \}$ with~$p_j \neq 0$.
  \item For every vector $\mathcal V \in \langle \mathcal G \rangle \setminus \{
  0 \}$, there exist polynomials $p_1, \ldots, p_\nu \in P$ such that
  \begin{align*}
  \mathcal V = p_1 \mathcal G_1 + \cdots + p_\nu \mathcal G_\nu
  \end{align*}
  and
  \begin{align*}
  \max \{ \deg(p_j) \mid j \in \{ 1, \ldots, \nu \}, p_j \neq 0 \} =
  \ind_{\mathcal M}(\mathcal V) - 1.
  \end{align*}
\end{enumerate}
\end{thm}

\begin{proof}
In order to show that~$A_1)$ holds if~$\mathcal G$ is the $\mathcal M$-module
border basis of~$\langle \mathcal G \rangle$, we let $\mathcal V \in \langle
\mathcal G \rangle \setminus \{ 0 \}$. We apply the Module Border Division
Algorithm~\ref{thm:divAlg} to~$\mathcal V$ and~$\mathcal G$ to obtain a
representation
\begin{align*}
\mathcal V = p_1 \mathcal G_1 + \cdots + p_\nu \mathcal G_\nu + c_1 t_1
e_{\alpha_1} + \cdots + c_\mu t_\mu e_{\alpha_\mu}
\end{align*}
with $p_1, \ldots, p_\nu \in P$, $c_1, \ldots, c_\mu \in K$, and
\begin{align*}
\deg(p_j) \leq \ind_{\mathcal M}(\mathcal V) - 1
\end{align*}
for all $j \in \{ 1, \ldots, \nu \}$ with $p_j \neq 0$. As $\mathcal V \in
\langle \mathcal G \rangle$, this yields
\begin{align*}
\langle \mathcal G \rangle = \mathcal V + \langle \mathcal G \rangle = c_1 t_1
e_{\alpha_1} + \cdots + c_\mu t_\mu e_{\alpha_\mu} + \langle \mathcal G \rangle
\in P^r / \langle \mathcal G \rangle.
\end{align*}
Hence Definition~\ref{defn:moduleBB} yields $c_1 = \cdots = c_\mu = 0$
as~$\mathcal G$ is the $\mathcal M$-module border basis of~$\langle \mathcal G
\rangle$ and the claim follows.

Next we prove that~$A_1)$ implies~$A_2)$. Let $\mathcal V \in \langle \mathcal G
\rangle \setminus \{ 0 \}$ and let
\begin{align*}
\mathcal V = p_1 \mathcal G_1 + \cdots + p_\nu \mathcal G_\nu
\end{align*}
with $p_1, \ldots, p_\nu \in P$ and
\begin{align*}
\deg(p_j) \leq \ind_{\mathcal M}(\mathcal V) - 1
\end{align*}
for all $j \in \{ 1, \ldots, \nu \}$ where $p_j \neq 0$ like in~$A_1)$. If
\begin{align*}
\deg(p_j) < \ind_{\mathcal M}(\mathcal V) - 1
\end{align*}
for some $j \in \{ 1, \ldots, \nu \}$, Proposition~\ref{thm:index} yields
\begin{align*}
\ind_{\mathcal M}(p_j \mathcal G_j) \leq \deg(p_j) + \ind_{\mathcal M}(\mathcal
G_j) = \deg(p_j) + 1 < \ind_{\mathcal M}(\mathcal V).
\end{align*}
Moreover, Proposition~\ref{thm:index} also yields
\begin{align*}
\ind_{\mathcal M}(\mathcal V) \leq \max \{ \ind_{\mathcal M}(p_j \mathcal G_j)
\mid j \in \{ 1, \ldots, \nu \}, p_j \neq 0 \} \leq \ind_{\mathcal M}(\mathcal
V).
\end{align*}
Altogether, we see that there has to be at least one index $j \in \{ 1, \ldots,
\nu \}$ such that $p_j \neq 0$ and $\deg(p_j) = \ind_{\mathcal M}(\mathcal V) -
1$.

At last, we show that~$\mathcal G$ is the $\mathcal M$-module border basis
of~$\langle \mathcal G \rangle$ if~$A_2)$ holds. Let $\mathcal V \in \langle
\mathcal G \rangle \cap \langle \mathcal M \rangle_K$. Assume that $\mathcal V
\neq 0$. Then~$A_2)$ yields the existence of $p_1, \ldots, p_\nu \in P$ such
that
\begin{align*}
\mathcal V = p_1 \mathcal G_1 + \cdots + p_\nu \mathcal G_\nu
\end{align*}
and
\begin{align*}
\deg(p_j) \leq \ind_{\mathcal M}(\mathcal V) - 1 = -1
\end{align*}
for all $j \in \{ 1, \ldots, \nu \}$ where $p_j \neq 0$. Thus we have $p_1 =
\cdots = p_\nu = 0$ and this yields the contradiction $\mathcal V = 0$.
Altogether, we have $\langle \mathcal G \rangle \cap \langle \mathcal M
\rangle_K = \{ 0 \}$ and the claim follows with Corollary~\ref{thm:char}.
\end{proof}

%
%

Next we show the connection of module border bases and their border form
modules. The characterization of module border bases again is a straightforward
generalization of the concept of the border form of border bases like it has
originally been proven in \cite[Prop.~11]{CharBB}. We follow the corresponding
definition of the border form of a vector in \cite[Defn.~6.4.24]{KR2} and the
corresponding characterization in \cite[Prop.~6.4.25]{KR2}.

\begin{defn}
\phantomsection\label{defn:BF}
\begin{enumerate}
  \item Let $\mathcal V \in P^r \setminus \{ 0 \}$. We write
  \begin{align*}
  \mathcal V = a_1 u_1 e_{i_1} + \cdots + a_s u_s e_{i_s}
  \end{align*}
  with $a_1, \ldots, a_s \in K \setminus \{ 0 \}$, $u_1 e_{i_1}, \ldots, u_s
  e_{i_s} \in \mathbb T^n \langle e_1, \ldots, e_r \rangle$ and
  \begin{align*}
  \ind_{\mathcal M}(u_1 e_{i_1}) \geq \ind_{\mathcal M}(u_2 e_{i_2}) \geq \cdots
  \geq \ind_{\mathcal M}(u_s e_{i_s}).
  \end{align*}
  Then we call the vector
  \begin{align*}
  \BF_{\mathcal M}(\mathcal V) = \sum_{\substack{j \in \{ 1, \ldots, s \}\\
  \ind_{\mathcal M}(u_j e_{i_j}) = \ind_{\mathcal M}(\mathcal V)}} a_j u_j
  e_{i_j} \in P^r
  \end{align*}
  the \emph{border form} of~$\mathcal V$ with respect to~$\mathcal M$.
  \item Let $U \subseteq P^r$ be a $P$-submodule. Then we call the $P$-submodule
  \begin{align*}
  \BF_{\mathcal M}(U) = \langle \BF_{\mathcal M}(\mathcal V) \mid \mathcal V \in
  U \setminus \{ 0 \} \rangle \subseteq P^r
  \end{align*}
  the \emph{border form module} of~$U$ with respect to~$\mathcal M$.
\end{enumerate}
\end{defn}

\begin{thm}[Module Border Bases and Border Form Modules]
\label{thm:BFMod}
The $\mathcal M$-module border prebasis~$\mathcal G$ is the $\mathcal M$-module
border basis of~$\langle \mathcal G \rangle$ if and only if the following
equivalent conditions are satisfied.
\begin{enumerate}
\renewcommand{\labelenumi}{$B_{\arabic{enumi}})$}
  \item For every $\mathcal V \in \langle \mathcal G \rangle \setminus \{ 0 \}$,
  we have
  \begin{align*}
  \Supp(\BF_{\mathcal M}(\mathcal V)) \subseteq \mathbb T^n \langle e_1, \ldots,
  e_r \rangle \setminus \mathcal M.
  \end{align*}
  \item We have
  \begin{align*}
  \BF_{\mathcal M}(\langle \mathcal G \rangle) = \langle \BF_{\mathcal
  M}(\mathcal G_1), \ldots, \BF_{\mathcal M}(\mathcal G_\nu) \rangle = \langle
  b_1 e_{\beta_1}, \ldots, b_\nu e_{\beta_\nu} \rangle.
  \end{align*}
\end{enumerate}
\end{thm}

\begin{proof}
We start with the proof that condition~$B_1)$ is satisfied if~$\mathcal G$ is
the $\mathcal M$-module border basis of~$\langle \mathcal G \rangle$. Let
$\mathcal V \in \langle \mathcal G \rangle \setminus \{ 0 \}$. Assume that
$\BF_{\mathcal M}(\mathcal V)$ contains a term of~$\mathcal M$ in its support.
Then we have $\ind_{\mathcal M}(\mathcal V) = 0$ and thus
\begin{align*}
\mathcal V = \BF_{\mathcal M}(\mathcal V) \subseteq \langle \mathcal M \rangle_K
\end{align*}
by Definition~\ref{defn:BF}. Now Corollary~\ref{thm:char} yields the
contradiction
\begin{align*}
\mathcal V \in \langle \mathcal G \rangle \cap \langle \mathcal M \rangle_K = \{
0 \}.
\end{align*}
Thus $\BF_{\mathcal M}(\mathcal V)$ does not contain a term of~$\mathcal M$ in
its support, and the claim follows.

Next we show that~$B_1)$ implies~$B_2)$. Since $\mathcal G_j \in \langle
\mathcal G \rangle \setminus \{ 0 \}$, we see that we have
\begin{align*}
\BF_{\mathcal M}(\mathcal G_j) = b_j e_{\beta_j} \in \BF_{\mathcal M}(\langle
\mathcal G \rangle)
\end{align*}
for every $j \in \{ 1, \ldots, \nu \}$. To prove the converse inclusion, we let
$\mathcal V \in \langle \mathcal G \rangle \setminus \{ 0 \}$. Then we have
\begin{align*}
\Supp(\BF_{\mathcal M}(\mathcal V)) \subseteq \mathbb T^n \langle
e_1, \ldots, e_r \rangle \setminus \mathcal M
\end{align*}
by~$B_1)$. Thus Definition~\ref{defn:BF} and Proposition~\ref{thm:border} yield
that every term in the support of~$\BF_{\mathcal M}(\mathcal V)$ is divisible by
a term in $\partial \mathcal M = \{ b_1 e_{\beta_1}, \ldots, b_\nu e_{\beta_\nu}
\}$. Thus we see that
\begin{align*}
\BF_{\mathcal M}(\mathcal V) \in \langle b_1 e_{\beta_1}, \ldots, b_\nu
e_{\beta_\nu} \rangle.
\end{align*}

Finally, we prove that~$\mathcal G$ is the $\mathcal M$-module border basis
of~$\langle \mathcal G \rangle$ if~$B_2)$ is satisfied. Assume that there exists
a vector $\mathcal V \in \langle \mathcal G \rangle \cap \langle \mathcal M
\rangle_K \setminus \{ 0 \}$. Then we have $\ind_{\mathcal M}(\mathcal V) = 0$
and as a consequence we see that
\begin{align*}
\BF_{\mathcal M}(\mathcal V) = \mathcal V \subseteq \langle \mathcal M
\rangle_K.
\end{align*}
Condition~$B_2)$ and Proposition~\ref{thm:border} now yield the contradiction
$\mathcal V = 0$. Altogether, the claim follows with Corollary~\ref{thm:char}.
\end{proof}

%
%

We now define a rewrite relation corresponding to~$\mathcal G$ and characterize
module border bases with it. For border bases, the analogous version has
originally been proven in \cite[Prop.~14]{CharBB}. Like for the previous
characterizations, our definition of rewrite rules and the first proposition
about it are straightforward generalizations of the corresponding ones for
border bases in \cite[p.~432]{KR2}, \cite[Remark~6.4.27]{KR2}, and
respectively in \cite[Prop.~2.2.2]{KR1}.

\begin{defn}
\phantomsection\label{defn:rewrite}
\begin{enumerate}
  \item Let $\mathcal V \in P^r$, and let $t e_k \in \Supp(\mathcal V)$ such
  that there exists a factorization
  \begin{align*}
  t e_k = t' b_j e_{\beta_j}
  \end{align*}
  with $t' \in \mathbb T^n$ and $j \in \{ 1, \ldots, \nu \}$. Let $c \in K$ be
  the coefficient of~$t e_k$ in~$\mathcal V$. Then the vector $\mathcal V - c t'
  \mathcal G_j \in P^r$ does not contain the term~$t e_k$ in its support
  anymore. We say that~\emph{$\mathcal V$ reduces to~$\mathcal V - c t' \mathcal
  G_j$ in one step} using the \emph{rewrite rule}~$\RedR{\mathcal G_j}$ defined
  by~$\mathcal G_j$ and write
  \begin{align*}
  \mathcal V \RedR{\mathcal G_j} \mathcal V - c t' \mathcal G_j.
  \end{align*}
  The passage from~$\mathcal V$ to~$\mathcal V - c t' \mathcal G_j$ is also
  called a \emph{reduction step} using~$\mathcal G_j$.
  \item The reflexive, transitive closure of the rewrite rules~$\RedR{\mathcal
  G_j}$ defined by~$\mathcal G_j$ for all $j \in \{ 1, \ldots, \nu \}$ is called
  the \emph{rewrite relation} associated to~$\mathcal G$ and is denoted
  by~$\RedR{\mathcal G}$.
  \item The equivalence relation generated by~$\RedR{\mathcal G}$ is denoted
  by~$\RedLR{\mathcal G}$.
\end{enumerate}
\end{defn}

\begin{prop}
\phantomsection\label{thm:rewriteProps}
\begin{enumerate}
  \item If $\mathcal V_1, \mathcal V_2 \in P^r$ satisfy $\mathcal V_1
  \RedR{\mathcal G} \mathcal V_2$, and if $c \in K$ and $t \in \mathbb T^n$,
  then we have $c t \mathcal V_1 \RedR{\mathcal G} c t \mathcal V_2$.
  \item If $\mathcal V_1, \mathcal V_2 \in P^r$ satisfy $\mathcal V_1
  \RedR{\mathcal G_j} \mathcal V_2$ for $j \in \{ 1, \ldots, \nu \}$, and if
  $\mathcal V_3 \in P^r$, then there exists a vector $\mathcal V_4 \in P^r$ such
  that $\mathcal V_1 + \mathcal V_3 \RedR{\mathcal G} \mathcal V_4$ and
  $\mathcal V_2 + \mathcal V_3 \RedR{\mathcal G} \mathcal V_4$.
  \item If $\mathcal V_1, \mathcal V_2, \mathcal V_3, \mathcal V_4 \in P^r$
  satisfy $\mathcal V_1 \RedLR{\mathcal G} \mathcal V_2$ and $\mathcal V_3
  \RedLR{\mathcal G} \mathcal V_4$, then we have $\mathcal V_1 + \mathcal V_3
  \RedLR{\mathcal G} \mathcal V_2 + \mathcal V_4$.
  \item If $\mathcal V_1, \mathcal V_2 \in P^r$ satisfy $\mathcal V_1
  \RedLR{\mathcal G} \mathcal V_2$, and if $p \in P$, then we have $p \mathcal
  V_1 \RedLR{\mathcal G} p \mathcal V_2$.
  \item For a vector $\mathcal V \in P^r$, we have $\mathcal V \RedLR{\mathcal
  G} 0$ if and only if~$\mathcal V \in \langle \mathcal G \rangle$.
  \item For vectors $\mathcal V_1, \mathcal V_2 \in P^r$, we have $\mathcal V_1
  \RedLR{\mathcal G} \mathcal V_2$ if and only if~$\mathcal V_1 - \mathcal V_2
  \in \langle \mathcal G \rangle$.
\end{enumerate}
\end{prop}

\begin{proof}
Firstly, we prove~a). Let $c \in K$ and $t \in \mathbb T^n$. If $c = 0$, nothing
has to be shown. Thus we suppose $c \neq 0$. By induction it suffices to prove
the claim for a single reduction step using~$\mathcal G_j$ where $j \in \{ 1,
\ldots, \nu \}$. Let $\mathcal V_1, \mathcal V_2 \in P^r$ and $j \in \{ 1,
\ldots, \nu \}$ be such that $\mathcal V_1 \RedR{\mathcal G_j} \mathcal V_2$.
Then Definition~\ref{defn:rewrite} yields the existence of a term $\hat{t}
e_{\beta_j} \in \Supp(\mathcal V_1)$, a term $t' \in \mathbb T^n$, and a
coefficient $c' \in K$ such that $\mathcal V_2 = \mathcal V_1 - c' t' \mathcal
G_j$ does not contain the term $\hat{t} e_{\beta_j} = t' b_j e_{\beta_j}$ in its
support anymore. Therefore $c t \mathcal V_2 = c t \mathcal V_1 - c t c' t'
\mathcal G_j$ does also not contain the term $t \hat{t} e_{\beta_j} = t t' b_j
e_{\beta_j}$ in its support anymore, i.\,e.\ we have $c t \mathcal V_1
\RedR{\mathcal G_j} c t \mathcal V_2$ by Definition~\ref{defn:rewrite}.

Secondly, we show~b). We let $\mathcal V_1, \mathcal V_2, \mathcal V_3 \in P^r$
be such that $\mathcal V_1 \RedR{\mathcal G_j} \mathcal V_2$ for some $j \in \{
1, \ldots, \nu \}$. According to Definition~\ref{defn:rewrite}, there exist a
term $t b_j e_{\beta_j} \in \Supp(\mathcal V_1)$ with $t \in \mathbb T^n$, and a
coefficient $c \in K \setminus \{ 0 \}$ such that we have $\mathcal V_2 =
\mathcal V_1 - c t \mathcal G_j$ and such that~$\mathcal V_2$ does not contain
the term~$t b_j e_{\beta_j}$ in its support anymore. Let $c' \in K$ be the
coefficient of~$t b_j e_{\beta_j}$ in~$\mathcal V_3$. We distinguish two cases.
If $c' = -c$, we have
\begin{align*}
\mathcal V_1 + \mathcal V_3 = \mathcal V_2 + c t \mathcal G_j + \mathcal V_3 =
\mathcal V_2 + \mathcal V_3 - c' t \mathcal G_j
\end{align*}
and $\mathcal V_2 + \mathcal V_3 - c' t \mathcal G_j$ does not contain the
term~$t b_j e_{\beta_j}$ in its support anymore, i.\,e.\ we have $\mathcal V_1 +
\mathcal V_3 \RedR{\mathcal G_j} \mathcal V_2 + \mathcal V_3$. The claim now
follows with $\mathcal V_4 = \mathcal V_2 + \mathcal V_3$. If $c' \neq -c$, we
define
\begin{align*}
\mathcal V_4 = \mathcal V_1 + \mathcal V_3 - (c + c') t \mathcal G_j =
\mathcal V_2 + c t \mathcal G_j + \mathcal V_3 -(c + c') t \mathcal G_j =
\mathcal V_2 + \mathcal V_3 - c' t \mathcal G_j.
\end{align*}
Then we see that~$\mathcal V_4$ does not contain the term~$t b_j e_{\beta_j}$ in
its support anymore and thus the claim follows.

For the proof of~c), we let $\mathcal V_1, \ldots, \mathcal V_4 \in P^r$ be such
that $\mathcal V_1 \RedLR{\mathcal G} V_2$ and such that $\mathcal V_3
\RedLR{\mathcal G} \mathcal V_4$. Then Definition~\ref{defn:rewrite} yields the
existence of vectors $\mathcal V'_0, \ldots, \mathcal V'_k \in P^r$ satisfying
$\mathcal V'_0 = \mathcal V_1$, $\mathcal V'_k = \mathcal V_2$, and $\mathcal
V'_{\ell - 1} \RedR{\mathcal G_{i_\ell}} \mathcal V'_\ell$ or $\mathcal V'_{\ell
- 1} \RedL{\mathcal G_{i_\ell}} \mathcal V'_\ell$ where $i_\ell \in \{ 1,
\ldots, \nu \}$ for all $\ell \in \{ 1, \ldots, k \}$. According to~b), for all
$\ell \in \{ 1, \ldots, k \}$, there exist vectors $\tilde{\mathcal V}'_\ell \in
P^r$ satisfying
\begin{align*}
\mathcal V'_{\ell - 1} + \mathcal V_3 \RedR{\mathcal G} \tilde{\mathcal V}'_\ell
\RedL{\mathcal G} \mathcal V'_\ell + \mathcal V_3.
\end{align*}
Therefore, for all $\ell \in \{ 1, \ldots, k \}$, we have
\begin{align*}
\mathcal V'_{\ell - 1} + \mathcal V_3 \RedLR{\mathcal G} \mathcal V'_\ell +
\mathcal V_3
\end{align*}
and induction over $\ell \in \{ 1, \ldots, k \}$ yields
\begin{align*}
\mathcal V_1 + \mathcal V_3 \RedLR{\mathcal G} \mathcal V_2 + \mathcal V_3.
\end{align*}
by Definition~\ref{defn:rewrite}. An analogous construction yields the claim
\begin{align*}
\mathcal V_2 + \mathcal V_4 \RedLR{\mathcal G} \mathcal V_2 + \mathcal V_3
\RedLR{\mathcal G} \mathcal V_1 + \mathcal V_3.
\end{align*}

In order to show~d), let $\mathcal V_1, \mathcal V_2 \in P^r$ be such that
$\mathcal V_1 \RedLR{\mathcal G} \mathcal V_2$ and let
\begin{align*}
p = c_1 u_1 + \cdots + c_s u_s \in P
\end{align*}
be with coefficients $c_1, \ldots, c_s \in K$ and terms $u_1, \ldots, u_s \in
\mathbb T^n$. Then we have
\begin{align*}
c_i u_i \mathcal V_1 \RedLR{\mathcal G} c_i u_i \mathcal V_2
\end{align*}
for all $i \in \{ 1, \ldots, s \}$ by~a). Induction over $i \in \{ 1, \ldots, s
\}$ and~c) now yield the claim
\begin{align*}
p \mathcal V_1 = c_1 u_1 \mathcal V_1 + \cdots + c_s u_s \mathcal V_1
\RedLR{\mathcal G} c_1 u_1 \mathcal V_2 + \cdots + c_s u_s \mathcal V_2 = p
\mathcal V_2.
\end{align*}

We go on with the proof of the equivalence~e). Let $\mathcal V \in P^r$. If
$\mathcal V \RedLR{\mathcal G} 0$, we collect the monomials used in the various
reduction steps and get polynomials $p_1, \ldots, p_\nu \in P$ such that
\begin{align*}
\mathcal V = p_1 \mathcal G_1 + \cdots + p_\nu \mathcal G_\nu \in \langle
\mathcal G \rangle.
\end{align*}
For the converse implication, suppose that $\mathcal V \in \langle \mathcal G
\rangle$. Then there exist polynomials $p_1, \ldots, p_\nu \in P$ such that
\begin{align*}
\mathcal V = p_1 \mathcal G_1 + \cdots + p_\nu \mathcal G_\nu \in \langle
\mathcal G \rangle.
\end{align*}
Obviously, we have $\mathcal G_j \RedLR{\mathcal G} 0$ by
Definition~\ref{defn:rewrite} and thus~d) yields $p_j \mathcal G_j \RedLR{\mathcal
G} 0$ for all $j \in \{ 1, \ldots, \nu \}$.
Therefore, induction over $j \in \{ 1, \ldots, \nu \}$ together with claim~c)
shows
\begin{align*}
\mathcal V = p_1 \mathcal G_1 + \cdots + p_\nu \mathcal G_\nu \RedLR{\mathcal G}
0.
\end{align*}

Finally, we show the equivalence in~f). Let $\mathcal V_1, \mathcal V_2 \in
P^r$. We have $\mathcal V \RedLR{\mathcal G} \mathcal V$ for all $\mathcal V \in
P^r$ by Definition~\ref{defn:rewrite}. In particular, we see that $-\mathcal V_2
\RedLR{\mathcal G} -\mathcal V_2$. Thus the condition $\mathcal V_1
\RedLR{\mathcal G} \mathcal V_2$ is equivalent to the condition $\mathcal V_1 -
\mathcal V_2 \RedLR{\mathcal G} \mathcal V_2 - \mathcal V_2 = 0$ by~c). We now
see that the claim is a direct consequence of~e).
\end{proof}

Next we define irreducible vectors and confluent and Noetherian rewrite rules
associated to a module border prebasis in the usual way, cf.\
\cite[pp.~432--433]{KR2}. Using Proposition~\ref{thm:rewriteProps}, we can then
characterize module border bases via confluent rewrite rules as it has been done
for border bases in \cite[Prop.~6.4.28]{KR2}.

\begin{defn}
\phantomsection\label{defn:confluent}
\begin{enumerate}
  \item A vector $\mathcal V_1 \in P^r$ is called \emph{irreducible} with
  respect to the rewrite relation~$\RedR{\mathcal G}$ if there exist no $j \in
  \{ 1, \ldots, \nu \}$ and no $\mathcal V_2 \in P^r \setminus \{   \mathcal V_1
  \}$ such that $\mathcal V_1 \RedR{\mathcal G_j} \mathcal V_2$.
  \item The rewrite rule~$\RedR{\mathcal G}$ is called \emph{Noetherian} if
  there exists no infinitely descending chain
  \begin{align*}
  \mathcal V_0 \RedR{\mathcal G_{i_1}} \mathcal V_1 \RedR{\mathcal G_{i_2}}
  \mathcal V_2 \RedR{\mathcal G_{i_3}} \cdots
  \end{align*}
  with $i_j \in \{ 1, \ldots, \nu \}$, $\mathcal V_j \in P^r \setminus \{
  \mathcal V_{j-1} \}$ for all $j \in \mathbb N \setminus \{ 0 \}$, and
  $\mathcal V_0 \in P^r$.
  \item The rewrite rule~$\RedR{\mathcal G}$ is called \emph{confluent} if for
  all vectors $\mathcal V_1, \mathcal V_2, \mathcal V_3 \in P^r$ satisfying
  $\mathcal V_1 \RedR{\mathcal G} \mathcal V_2$ and $\mathcal V_1 \RedR{\mathcal
  G} \mathcal V_3$, there exists a vector $\mathcal V_4 \in P^r$ such that
  $\mathcal V_2 \RedR{\mathcal G} \mathcal V_4$ and $\mathcal V_3 \RedR{\mathcal
  G} \mathcal V_4$.
\end{enumerate}
\end{defn}

\begin{rem}
\phantomsection\label{rem:rewriteNoeth}
\begin{enumerate}
  \item For $r = 1$, module border bases coincide with the usual border bases.
  Thus \cite[Exmp.~6.4.26]{KR2} shows that the rewrite relation~$\RedR{\mathcal
  G}$ is not Noetherian, in general.
  \item A vector $\mathcal V \in P^r$ is irreducible with respect
  to~$\RedR{\mathcal G}$ if and only if $\mathcal V \in \langle \mathcal M
  \rangle_K$ according to Definition~\ref{defn:orderModule} and
  Definition~\ref{defn:confluent}.
  \item Considering the steps of the Module Border Division
  Algorithm~\ref{thm:divAlg} in detail, we see that it performs reduction steps
  using~$\mathcal G_j$ where $j \in \{ 1, \ldots, \nu \}$ to compute the normal
  remainder of a given vector. Thus for every $\mathcal V \in P^r$, we have
  $\mathcal V \RedR{\mathcal G} \NR_{\mathcal G}(\mathcal V)$. In particular,
  $\NR_{\mathcal G}(\mathcal V) \in \langle \mathcal M \rangle_K$ is irreducible
  with respect to~$\RedR{\mathcal G}$.
\end{enumerate}
\end{rem}

\begin{thm}[Module Border Bases and Rewrite Rules]
\label{thm:rewrite}
The $\mathcal M$-module border prebasis~$\mathcal G$ is the $\mathcal M$-module
border basis of~$\langle \mathcal G \rangle$ if and only if the following
equivalent conditions are satisfied.
\begin{enumerate}
\renewcommand{\labelenumi}{$C_{\arabic{enumi}})$}
  \item For all $\mathcal V \in P^r$, we have $\mathcal V \RedR{\mathcal G} 0$
  if and only if~$\mathcal V \in \langle \mathcal G \rangle$.
  \item If $\mathcal V \in \langle \mathcal G \rangle$ is irreducible with
  respect to~$\RedR{\mathcal G}$, then we have~$\mathcal V = 0$.
  \item For all $\mathcal V \in P^r$, there is a unique vector $\mathcal W \in
  P^r$ such that $\mathcal V \RedR{\mathcal G} \mathcal W$ and such
  that~$\mathcal W$ is irreducible with respect to~$\RedR{\mathcal G}$.
  \item The rewrite relation~$\RedR{\mathcal G}$ is confluent.
\end{enumerate}
\end{thm}

\begin{proof}
We start to show that~$C_1)$ implies~$C_2)$. Let $\mathcal V \in \langle
\mathcal G \rangle$ be irreducible with respect to~$\RedR{\mathcal G}$. As we
have $\mathcal V \RedR{\mathcal G} 0$ by~$C_1)$, $\mathcal V$ must be zero.

Next we show that~$C_2)$ implies~$C_3)$. Let $\mathcal V \in P^r$. According to
Remark~\ref{rem:rewriteNoeth}, $\NR_{\mathcal G}(\mathcal V)$ is a vector with
the claimed properties. In order to show the uniqueness, we let $\mathcal W \in
P^r$ be irreducible with respect to~$\RedR{\mathcal G}$ and satisfying $\mathcal
V \RedR{\mathcal G} \mathcal W$. Then we see that $\NR_{\mathcal G}(\mathcal V)
\RedLR{\mathcal G} \mathcal W$ by Definition~\ref{defn:rewrite} and it follows
$\NR_{\mathcal G}(\mathcal V) - \mathcal W \in \langle \mathcal G \rangle$
according to Proposition~\ref{thm:rewriteProps}. Additionally,
Remark~\ref{rem:rewriteNoeth} yields that $\NR_{\mathcal G}(\mathcal V) -
\mathcal W \in \langle \mathcal M \rangle_K$ is irreducible with respect
to~$\RedR{\mathcal G}$. Thus the claim follows from~$C_2)$.

In order to prove that~$C_3)$ implies~$C_4)$, we let $\mathcal V_1, \mathcal
V_2, \mathcal V_3 \in P^r$ be satisfying $\mathcal V_1 \RedR{\mathcal G}
\mathcal V_2$ and $\mathcal V_1 \RedR{\mathcal G} \mathcal V_3$. According to
Remark~\ref{rem:rewriteNoeth}, we see that
\begin{align*}
\mathcal V_1 \RedR{\mathcal G} \mathcal V_2 \RedR{\mathcal G} \NR_{\mathcal
G}(\mathcal V_2) \in \langle \mathcal M \rangle_K
\end{align*}
and
\begin{align*}
\mathcal V_1 \RedR{\mathcal G} \mathcal V_3 \RedR{\mathcal G} \NR_{\mathcal
G}(\mathcal V_3) \in \langle \mathcal M \rangle_K,
\end{align*}
and that both~$\NR_{\mathcal G}(\mathcal V_2)$ and~$\NR_{\mathcal G}(\mathcal
V_3)$ are irreducible with respect to the rewrite relation~$\RedR{\mathcal G}$.
Thus~$C_3)$ implies the equality $\NR_{\mathcal G}(\mathcal V_2) = \NR_{\mathcal
G}(\mathcal V_3)$ and the claim follows by Definition~\ref{defn:confluent}.

We go on with the proof that~$C_4)$ implies~$C_1)$. Let $\mathcal V \in P^r$ be
satisfying $\mathcal V \RedR{\mathcal G} 0$. Then
Proposition~\ref{thm:rewriteProps} yields $\mathcal V \in \langle \mathcal G
\rangle$. For the converse implication, we let $\mathcal V \in \langle \mathcal
G \rangle$. Then Proposition~\ref{thm:rewriteProps} yields $\mathcal V
\RedLR{\mathcal G} 0$. Let $\mathcal V_0, \ldots, \mathcal V_k \in P^r$ be
such that $\mathcal V_0 = \mathcal V$, $\mathcal V_k = 0$ and $\mathcal V_{\ell
- 1} \RedR{\mathcal G} \mathcal V_\ell$ or $\mathcal V_{\ell - 1} \RedL{\mathcal
G} \mathcal V_\ell$ for all $\ell \in \{ 1, \ldots, k \}$. If there exists no
index $\ell \in \{ 1, \ldots, k \}$ such that $\mathcal V_{\ell - 1}
\RedL{\mathcal G} \mathcal V_\ell$, the claim follows immediately with
Definition~\ref{defn:rewrite}. Thus suppose that $\mathcal V_{\ell - 1}
\RedL{\mathcal G} \mathcal V_\ell$ for some $\ell \in \{ 1, \ldots, k \}$. By
Definition~\ref{defn:rewrite}, we see that $\mathcal V_{k-1} \RedR{\mathcal G}
\mathcal V_k = 0$. Let $s \in \{ 1, \ldots, k-1 \}$ be maximal such that
$\mathcal V_{s-1} \RedL{\mathcal G} \mathcal V_s$. Then we have $\mathcal V_s
\RedR{\mathcal G} 0$, $\mathcal V_s \RedR{\mathcal G} \mathcal V_{s-1}$,
and~$C_4)$ yields $\mathcal V_{s-1} \RedR{\mathcal G} 0$. If we replace the
sequence $\mathcal V_0, \ldots, \mathcal V_{k-1}, 0$ with the shorter sequence
$\mathcal V_0, \ldots, \mathcal V_{s-1}, 0$, we see that the claim follows by
induction over the number of reduction steps $\mathcal V_{\ell - 1}
\RedL{\mathcal G} \mathcal V_\ell$ where~$\ell \in \{ 1, \ldots, k \}$.

Next we show that condition~$C_1)$ is satisfied if~$\mathcal G$ is the $\mathcal
M$-module border basis of~$\langle \mathcal G \rangle$. If a vector $\mathcal V
\in P^r$ satisfies $\mathcal V \RedR{\mathcal G} 0$, we have $\mathcal V \in
\langle \mathcal G \rangle$ by Proposition~\ref{thm:rewriteProps}. Conversely,
let $\mathcal V \in \langle \mathcal G \rangle$. Then it follows that $\mathcal
V \RedR{\mathcal G} \NR_{\mathcal G}(\mathcal V) \in \langle \mathcal M
\rangle_K$ by Remark~\ref{rem:rewriteNoeth}. Since $\mathcal V \in \langle
\mathcal G \rangle$, we also have $\NR_{\mathcal G}(\mathcal V) \in \langle
\mathcal G \rangle$ according to Definition~\ref{defn:rewrite}. Hence the claim
follows as Corollary~\ref{thm:char} yields~$\NR_{\mathcal G}(\mathcal V) \in
\langle \mathcal G \rangle \cap \langle \mathcal M \rangle_K = \{ 0 \}$.

Finally, we prove that~$\mathcal G$ is the $\mathcal M$-module border basis
of~$\langle \mathcal G \rangle$ if~$C_2)$ holds. Assume there exists a
$\mathcal V \in \langle \mathcal G \rangle \cap \langle \mathcal M \rangle_K
\setminus \{ 0 \}$. Then $\ind_{\mathcal M}(\mathcal V) = 0$ and
therefore~$\mathcal V$ is irreducible with respect to~$\RedR{\mathcal G}$
according to Remark~\ref{rem:rewriteNoeth}. Thus~$C_2)$ yields the contradiction
$\mathcal V = 0$ and the claim follows with Corollary~\ref{thm:char}.
\end{proof}

%
%

We go on with the characterization of module border bases via commuting
matrices. This characterization imitates the corresponding characterization for
border bases originally published in \cite[Prop.~16]{CharBB} and
\cite[pp.~13--15]{CharBB}. Our generalization to the module setting follows the
corresponding versions of border bases in \cite[Defn.~6.4.29]{KR2},
\cite[Thm.~6.4.30]{KR2}, and \cite[Prop.~6.4.32]{KR2}. To ease the notation, we
let
\begin{align*}
\delta_{ij} = \begin{cases}
1 & \text{if $i = j$}\\
0 & \text{if $i \neq j$}
\end{cases}
\end{align*}
denote the Kronecker delta and $\mathfrak I_\mu \in \Mat_\mu(K)$ denote the
identity matrix for the remainder of this section.

\begin{defn}
\label{defn:multMat}
Given $s \in \{ 1, \ldots, n \}$, we define the \emph{$s^{\text{th}}$~formal
multiplication matrix}
\begin{align*}
\mathfrak X_s = \big( \xi_{k \ell}^{(s)} \big)_{1 \leq k, \ell \leq \mu} \in
\Mat_\mu(K)
\end{align*}
of~$\mathcal G$ by
\begin{align*}
\xi_{k \ell}^{(s)} = \begin{cases}
\delta_{ki} & \text{if $x_s t_\ell e_{\alpha_\ell} = t_i e_{\alpha_i} \in
\mathcal M$,}\\
c_{kj} & \text{if $x_s t_\ell e_{\alpha_\ell} = b_j e_{\beta_j} \in \partial
\mathcal M$.}
\end{cases}
\end{align*}
\end{defn}

\begin{exmp}
\label{exmp:multMat}
We consider the $\mathcal M$-module border prebasis $\mathcal G \subseteq P^2$
from Example~\ref{exmp:divAlg}, again. Recall, that
\begin{align*}
\mathcal M = \{ x e_1, y e_1, e_1, x^2 e_2, x e_2, e_2 \} = \{ t_1
e_{\alpha_1}, \ldots, t_6 e_{\alpha_6} \},
\end{align*}
and
\begin{align*}
\mathcal G_1 & = x^2 e_1 - y e_1 + e_2,\\
\mathcal G_2 & = xy e_1 - e_2,\\
\mathcal G_3 & = y^2 e_1 - x e_2,\\
\mathcal G_4 & = x^3 e_2 - e_1,\\
\mathcal G_5 & = x^2y e_2 - e_1 - e_2,\\
\mathcal G_6 & = xy e_2 + 3 e_1,\\
\mathcal G_7 & = y e_2 - x e_1 - y e_1 - e_1 - e_2.
\end{align*}
Then the formal multiplication matrices $\mathfrak X, \mathfrak Y \in
\Mat_7(\mathbb Q)$ of~$\mathcal G$ are
\begin{align*}
\mathfrak X & = \begin{pmatrix}
0 & 0 & 1 & 0 & 0 & 0\\
1 & 0 & 0 & 0 & 0 & 0\\
0 & 0 & 0 & 1 & 0 & 0\\
0 & 0 & 0 & 0 & 1 & 0\\
0 & 0 & 0 & 0 & 0 & 1\\
-1 & 1 & 0 & 0 & 0 & 0
\end{pmatrix},
& \mathfrak Y & = \begin{pmatrix}
0 & 0 & 0 & 0 & 0 & 1\\
0 & 0 & 1 & 0 & 0 & 1\\
0 & 0 & 0 & 1 & -3 & 1\\
0 & 0 & 0 & 0 & 0 & 0\\
0 & 1 & 0 & 0 & 0 & 0\\
1 & 0 & 0 & 1 & 0 & 1
\end{pmatrix}.
\end{align*}
\end{exmp}

\begin{rem}
\label{rem:interprMultMat}
Similar to the interpretation of the formal multiplication matrices of a border
prebasis in \cite[Page~434]{KR2}, we can interpret the multiplication matrices
of the $\mathcal M$-module border prebases~$\mathcal G$ the following way: Let
$s \in \{ 1, \ldots, n \}$ and let $\mathfrak X_s \in \Mat_\mu(K)$ be the
$s^\text{th}$~formal multiplication matrix of~$\mathcal G$. We can identify
every vector
\begin{align*}
\mathcal V = c_1 t_1 e_{\alpha_1} + \cdots + c_\mu t_\mu e_{\alpha_\mu} \in
\langle \mathcal M \rangle_K \subseteq P^r
\end{align*}
with the corresponding column vector
\begin{align*}
(c_1, \ldots, c_\mu)^{\tr} \in K^\mu.
\end{align*}
Then the column vector
\begin{align*}
(c_1', \ldots, c_\mu') = \mathfrak X_s (c_1, \ldots, c_\mu)^{\tr} \in K^\mu
\end{align*}
corresponds to the vector
\begin{align*}
c_1' t_1 e_{\alpha_1} + \cdots + c_\mu' t_\mu e_{\alpha_\mu} = \NR_{\mathcal
G}(x_s \mathcal  V) \in \langle \mathcal M \rangle_K \subseteq P^r.
\end{align*}
In particular, we have
\begin{align*}
c_1' t_1 e_{\alpha_1} + \cdots + c_\mu' t_\mu e_{\alpha_\mu} + \langle \mathcal
G \rangle = x_s \mathcal V + \langle \mathcal G \rangle.
\end{align*}
\end{rem}

\begin{lem}
\label{thm:orderModuleModule}
Assume that $\mu \neq 0$, i.\,e.\ $\mathcal M \neq \emptyset$. Let $\mathfrak
X_1, \ldots, \mathfrak X_n \in \Mat_\mu(K)$ be the formal multiplication
matrices of~$\mathcal G$ and assume that $\mathfrak X_1, \ldots, \mathfrak X_n$
are pairwise commuting. Then the $K$-vector subspace~$\langle \mathcal M
\rangle_K \subseteq P^r$ is a $P$-module with scalar multiplication $\circ: P
\times \langle \mathcal M \rangle_K \to \langle \mathcal M \rangle_K$ defined by
\begin{align*}
p \circ (c_1 t_1 e_{\alpha_1} + \cdots + c_\mu t_\mu e_{\alpha_\mu}) = (t_1
e_{\alpha_1}, \ldots, t_\mu e_{\alpha_\mu}) p(\mathfrak X_1, \ldots, \mathfrak
X_n) (c_1, \ldots, c_\mu)^{\tr}
\end{align*}
for all $p \in P$ and all $c_1, \ldots, c_\mu \in K$. Moreover, the set $\{
e_1, \ldots, e_r \} \cap \mathcal M \subseteq \mathcal M$ generates $\langle
\mathcal M \rangle_K$ as a $P$-module.
\end{lem}

\begin{proof}
As~$\mathcal M$ is $K$-linearly independent, the map~$\circ$ is
well-defined. Moreover, the $K$-vector subspace $\langle \mathcal M \rangle_K
\subseteq P^r$ is obviously an additive group. Since the formal multiplication
matrices $\mathfrak X_1, \ldots, \mathfrak X_n \in \Mat_\mu(K)$ are pairwise
commuting, for all polynomials $p, q \in P$ and all coefficients $c_1, \ldots,
c_\mu, d_1, \ldots, d_\mu \in K$, we have
\begin{align*}
& 1 \circ (c_1 t_1 e_{\alpha_1} + \cdots + c_\mu t_\mu e_{\alpha_\mu})\\
& = (t_1 e_{\alpha_1}, \ldots, t_\mu e_{\alpha_\mu}) \mathfrak I_\mu (c_1,
\ldots, c_\mu)^{\tr}\\
& = c_1 t_1 e_{\alpha_1} + \cdots + c_\mu t_\mu e_{\alpha_\mu},
\end{align*}
and
\begin{align*}
& (pq) \circ (c_1 t_1 e_{\alpha_1} + \cdots + c_\mu t_\mu e_{\alpha_\mu})\\
& = (t_1 e_{\alpha_1}, \ldots, t_\mu e_{\alpha_\mu}) (pq)(\mathfrak X_1, \ldots,
\mathfrak X_n) (c_1, \ldots, c_\mu)^{\tr}\\
& = (t_1 e_{\alpha_1}, \ldots, t_\mu e_{\alpha_\mu}) p(\mathfrak X_1, \ldots,
\mathfrak X_n) q(\mathfrak X_1, \ldots, \mathfrak X_n) (c_1, \ldots,
c_\mu)^{\tr}\\
& = p \circ ((t_1 e_{\alpha_1}, \ldots, t_\mu e_{\alpha_\mu}) q(\mathfrak X_1,
\ldots, \mathfrak X_n) (c_1, \ldots, c_\mu)^{\tr})\\
& = p \circ (q \circ (c_1 t_1 e_{\alpha_1} + \cdots + c_\mu t_\mu
e_{\alpha_\mu})),
\end{align*}
and
\begin{align*}
& (p+q) \circ (c_1 t_1 e_{\alpha_1} + \cdots + c_\mu t_\mu e_{\alpha_\mu})\\
& = (t_1 e_{\alpha_1}, \ldots, t_\mu e_{\alpha_\mu}) (p+q)(\mathfrak X_1,
\ldots, \mathfrak X_n) (c_1, \ldots, c_\mu)^{\tr}\\
& = (t_1 e_{\alpha_1}, \ldots, t_\mu e_{\alpha_\mu}) (p(\mathfrak X_1,
\ldots, \mathfrak X_n) + q(\mathfrak X_1, \ldots, \mathfrak X_n)) (c_1, \ldots,
c_\mu)^{\tr}\\
& = (t_1 e_{\alpha_1}, \ldots, t_\mu e_{\alpha_\mu}) p(\mathfrak X_1, \ldots,
\mathfrak X_n) (c_1, \ldots, c_\mu)^{\tr}\\
& \alignLongFormula + (t_1 e_{\alpha_1}, \ldots, t_\mu e_{\alpha_\mu})
q(\mathfrak X_1, \ldots, \mathfrak X_n) (c_1, \ldots, c_\mu)^{\tr}\\
& = (p \circ (c_1 t_1 e_{\alpha_1} + \cdots + c_\mu t_\mu e_{\alpha_\mu})) + (q
\circ (c_1 t_1 e_{\alpha_1} + \cdots + c_\mu t_\mu e_{\alpha_\mu})),
\end{align*}
and
\begin{align*}
& p \circ ((c_1 t_1 e_{\alpha_1} + \cdots + c_\mu t_\mu e_{\alpha_\mu}) + (d_1
t_1 e_{\alpha_1} + \cdots + d_\mu t_\mu e_{\alpha_\mu}))\\
& = p \circ ((c_1 + d_1) t_1 e_{\alpha_1} + \cdots + (c_\mu + d_\mu) t_\mu
e_{\alpha_\mu})\\
& = (t_1 e_{\alpha_1}, \ldots, t_\mu e_{\alpha_\mu}) p(\mathfrak X_1, \ldots,
\mathfrak X_n) (c_1 + d_1, \ldots, c_\mu + d_\mu)^{\tr}\\
& = (t_1 e_{\alpha_1}, \ldots, t_\mu e_{\alpha_\mu}) p(\mathfrak X_1, \ldots,
\mathfrak X_n) ((c_1, \ldots, c_\mu)^{\tr} + (d_1, \ldots, d_\mu)^{\tr})\\
& = (t_1 e_{\alpha_1}, \ldots, t_\mu e_{\alpha_\mu}) p(\mathfrak X_1, \ldots,
\mathfrak X_n) (c_1, \ldots, c_\mu)^{\tr}\\
& \alignLongFormula + (t_1 e_{\alpha_1}, \ldots, t_\mu e_{\alpha_\mu})
p(\mathfrak X_1, \ldots, \mathfrak X_n) (d_1, \ldots, d_\mu)^{\tr}\\
& = (p \circ (c_1 t_1 e_{\alpha_1} + \cdots + c_\mu t_\mu e_{\alpha_\mu})) + (p
\circ (d_1 t_1 e_{\alpha_1} + \cdots + d_\mu t_\mu e_{\alpha_\mu})).
\end{align*}
Altogether, we see that that $(\langle \mathcal M \rangle_K, +, \circ)$ is
indeed a $P$-module.

It remains to prove that the $P$-module~$\langle \mathcal M \rangle_K$ is
generated by $\{ e_1, \ldots, e_r \} \cap \mathcal M$. W.\,l.\,o.\,g.\ we
suppose that we have $e_1, \ldots, e_\ell \in \mathcal M$ and $e_{\ell + 1},
\ldots, e_r \notin \mathcal M$ for some $\ell \in \{ 1, \ldots, r \}$, and that
$t_k e_{\alpha_k} = e_k$ for all $k \in \{ 1, \ldots, \ell \}$. Let $\{
\mathfrak e_1, \ldots, \mathfrak e_\mu \} \subseteq K^\mu$ be the canonical
$K$-vector space basis of~$K^\mu$, and let $k \in \{ 1, \ldots, \ell \}$. We
prove by induction on degree that  $t \circ e_k = t e_k$ for all $t \in \mathcal
O_k$. The induction starts with
\begin{align*}
1 \circ e_k & = (t_1 e_{\alpha_1}, \ldots, t_\mu e_{\alpha_\mu}) 1(\mathfrak
X_1, \ldots, \mathfrak X_n) \mathfrak e_k^{\tr}\\
& = (t_1 e_{\alpha_1}, \ldots, t_\mu e_{\alpha_\mu}) \mathfrak I_\mu \mathfrak
e_k^{\tr}\\
& = (t_1 e_{\alpha_1}, \ldots, t_\mu e_{\alpha_\mu}) \mathfrak e_k^{\tr}\\
& = t_k e_{\alpha_k}\\
& = e_k.
\end{align*}
For the induction step, suppose that $t \in \mathcal O_k$ with $\deg(t) > 0$.
Then there is a factorization
\begin{align*}
t e_k = t_i e_{\alpha_i} = x_s t_j e_{\alpha_j}
\end{align*}
with $i, j \in \{ 1, \ldots, \mu \}$ and $s \in \{ 1, \ldots, n \}$ by
Definition~\ref{defn:orderModule}. The induction hypothesis yields
\begin{align*}
t \circ e_k = (x_s t_j) \circ e_{\alpha_j} = x_s \circ (t_j \circ e_{\alpha_j})
= x_s \circ (t_j e_{\alpha_k}).
\end{align*}
Thus we see that
\begin{align*}
t \circ e_k & = x_s \circ (t_j e_{\alpha_j})\\
& = (t_1 e_{\alpha_1}, \ldots, t_\mu e_{\alpha_\mu}) x_s(\mathfrak X_1,
\ldots, \mathfrak X_n) \mathfrak e_j^{\tr}\\
& = (t_1 e_{\alpha_1}, \ldots, t_\mu e_{\alpha_\mu}) \mathfrak X_s \mathfrak
e_j^{\tr}\\
& = (t_1 e_{\alpha_1}, \ldots, t_\mu e_{\alpha_\mu}) \big( \xi_{1 j}^{(s)},
\ldots, \xi_{\mu j}^{(s)} \big)^{\tr}\\
& = \delta_{1 i} t_1 e_{\alpha_1} + \cdots + \delta_{\mu i} t_\mu
e_{\alpha_\mu}\\
& = t_i e_{\alpha_i}\\
& = t e_k,
\end{align*}
i.\,e.\ the above claim has been proven by induction. For every $c \in K$ and
every $i \in \{ 1, \ldots, \mu \}$, we also have
\begin{align*}
c \circ t_i e_{\alpha_i} & = (t_1 e_{\alpha_1}, \ldots, t_\mu e_{\alpha_\mu})
c(\mathfrak X_1, \ldots, \mathfrak X_n) \mathfrak e_i^{\tr}\\
& = (t_1 e_{\alpha_1}, \ldots, t_\mu e_{\alpha_\mu}) c \mathfrak I_\mu
\mathfrak e_i^{\tr}\\
& = (t_1 e_{\alpha_1}, \ldots, t_\mu e_{\alpha_\mu}) c \mathfrak e_i^{\tr}\\
& = c t_i e_{\alpha_i}.
\end{align*}
Altogether, since $\{ e_1, \ldots, e_r \} \cap \mathcal M = \{ e_1, \ldots,
e_\ell \}$, we see that
\begin{align*}
(c_1 t_1) \circ e_{\alpha_1} + \cdots + (c_\mu t_\mu) \circ e_{\alpha_\mu} =
c_1 t_1 e_{\alpha_1} + \cdots + c_\mu t_\mu e_{\alpha_\mu}
\end{align*}
for all $c_1, \ldots, c_\mu \in K$, i.\,e.\ the $P$-module $\langle \mathcal
M \rangle_K$ is generated by $\{ e_1, \ldots, e_r \} \cap \mathcal M$.
\end{proof}

\begin{thm}[Module Border Bases and Commuting Matrices]
\label{thm:commMat}
For all $s \in \{ 1, \ldots, n \}$, we define the map
\begin{align*}
\varrho_s: \{ 1, \ldots, \mu \} \to \mathbb N, \quad i \mapsto \begin{cases}
j & \text{if $x_s t_i e_{\alpha_i} = t_j e_{\alpha_j} \in \mathcal M$,}\\
k & \text{if $x_s t_i e_{\alpha_i} = b_k e_{\beta_k} \in \partial \mathcal M$.}
\end{cases}
\end{align*}
Then the $\mathcal M$-module border prebasis~$\mathcal G$ is the $\mathcal
M$-module border basis of~$\langle \mathcal G \rangle$ if and only if the
following equivalent conditions are satisfied.
\begin{enumerate}
\renewcommand{\labelenumi}{$D_{\arabic{enumi}})$}
  \item The formal multiplication matrices $\mathfrak X_1, \ldots, \mathfrak X_n
  \in \Mat_\mu(K)$ of~$\mathcal G$ are pairwise commuting.
  \item The following equations are satisfied for every $p \in \{ 1, \ldots, \mu
  \}$ and for every $s, u \in \{ 1, \ldots, n \}$ with $s \neq u$:
  \begin{enumerate}
  \renewcommand{\labelenumi}{(\arabic{enumi})}
    \item If $x_s t_i e_{\alpha_i} = t_j e_{\alpha_j}$, $x_u t_i e_{\alpha_i} =
    b_k e_{\beta_k}$, and $x_s b_k e_{\beta_k} = b_\ell e_{\beta_\ell}$ with
    indices $i, j \in \{ 1, \ldots, \mu \}$ and $k, \ell \in \{ 1, \ldots, \nu
    \}$, we have
    \begin{align*}
    \sum_{\substack{m \in \{ 1, \ldots, \mu \}\\ x_s t_m e_{\alpha_m} \in
    \mathcal M}} \delta_{p \varrho_s(m)} c_{mk} + \sum_{\substack{m \in \{ 1,
    \ldots, \mu \}\\ x_s t_m e_{\alpha_m} \in \partial \mathcal M}} c_{p
    \varrho_s(m)} c_{mk} = c_{p \ell}.
    \end{align*}
    \item If $x_s t_i e_{\alpha_i} = b_j e_{\beta_j}$ and $x_u t_i e_{\alpha_i}
    = b_k e_{\beta_k}$ with indices $i \in \{ 1, \ldots, \mu \}$ and $j, k \in
    \{ 1, \ldots, \nu \}$, we have
    \begin{align*}
    & \sum_{\substack{m \in \{ 1, \ldots, \mu \}\\ x_s t_m e_{\alpha_m} \in
    \mathcal M}} \delta_{p \varrho_s(m)} c_{mk} + \sum_{\substack{m \in \{ 1,
    \ldots, \mu \}\\ x_s t_m e_{\alpha_m} \in \partial \mathcal M}} c_{p
    \varrho_s(m)} c_{mk}\\
    & = \sum_{\substack{m \in \{ 1, \ldots, \mu \}\\ x_u t_m e_{\alpha_m} \in
    \mathcal M}} \delta_{p \varrho_u(m)} c_{mj} + \sum_{\substack{m \in \{ 1,
    \ldots, \mu \}\\ x_u t_m e_{\alpha_m} \in \partial \mathcal M}} c_{p
    \varrho_u(m)} c_{mj}.
    \end{align*}
  \end{enumerate}
\end{enumerate}
If the equivalent conditions are satisfied, for all $s \in \{ 1, \ldots, n \}$,
the formal multiplication matrix $\mathfrak X_s$ represents the multiplication
endomorphism of the $K$-vector space~$P^r / \langle \mathcal G \rangle$ defined
by $\mathcal V + \langle \mathcal G \rangle \mapsto x_s \mathcal V + \langle
\mathcal G \rangle$, where $\mathcal V \in P^r$, with respect to the $K$-vector
space basis~$\varepsilon_{\langle \mathcal G \rangle}(\mathcal M)$.
\end{thm}

\begin{proof}
If $\mu = 0$, i.\,e.\ if $\mathcal M = \emptyset$ and $\partial \mathcal M = \{
e_1, \ldots, e_r \} = \mathcal G$, the claim is obviously true. Thus suppose
that $\mu \neq 0$.

We start with the proof that condition~$D_1)$ is satisfied if~$\mathcal G$ is
the $\mathcal M$-module border basis of $\langle \mathcal G \rangle$, i.\,e.\
that $\varepsilon_{\langle \mathcal G \rangle}(\mathcal M)$ is a $K$-vector
space basis of $P^r / \langle \mathcal G \rangle$ and $\# \varepsilon_{\langle
\mathcal G \rangle}(\mathcal M) = \mu$ by Definition~\ref{defn:moduleBB}. Let
$s \in \{ 1, \ldots, n \}$. The formal multiplication matrix $\mathfrak X_s \in
\Mat_\mu(K)$ defines a $K$-vector space endomorphism~$\phi_s$ of $P^r / \langle
\mathcal G \rangle$ with respect to the $K$-vector space
basis~$\varepsilon_{\langle \mathcal G \rangle}(\mathcal M)$. We show that
$\phi_s(\mathcal V + \langle \mathcal G \rangle) = x_s \mathcal V + \langle
\mathcal G \rangle$ for all $\mathcal V \in P^r / \langle \mathcal G \rangle$,
i.\,e.\ $\phi_s$ is the $K$-vector space endomorphism corresponding to the
multiplication by~$x_s$. Consider the expansions
\begin{align*}
\phi_s(t_1 e_{\alpha_1} + \langle \mathcal G \rangle) & = \xi_{11}^{(s)} (t_1
e_{\alpha_1} + \langle \mathcal G \rangle) + \cdots + \xi_{\mu 1}^{(s)} (t_\mu
e_{\alpha_\mu} + \langle \mathcal G \rangle),\\
& \alignDotsSpace\vdots \\
\phi_s(t_\mu e_{\alpha_\mu} + \langle \mathcal G \rangle) & = \xi_{1 \mu}^{(s)}
(t_1 e_{\alpha_1} + \langle \mathcal G \rangle) + \cdots + \xi_{\mu \mu}^{(s)}
(t_\mu e_{\alpha_\mu} + \langle \mathcal G \rangle).
\end{align*}
Let $u \in \{ 1, \ldots, \mu \}$. Then we see that $x_s t_u e_{\alpha_u} \in
\oline{\partial \mathcal M}$ according to Definition~\ref{defn:border}. We now
distinguish two cases. If $x_s t_u e_{\alpha_u} = t_i e_{\alpha_i} \in \mathcal
M$ with $i \in \{ 1, \ldots, \mu \}$, Definition~\ref{defn:multMat} yields
\begin{align*}
\phi_s(t_u e_{\alpha_u} + \langle \mathcal G \rangle) = \delta_{1i} t_1
e_{\alpha_1} + \cdots + \delta_{\mu i} t_\mu e_{\alpha_\mu} + \langle \mathcal G
\rangle = t_i e_{\alpha_i} + \langle \mathcal G \rangle = x_s t_u e_{\alpha_u} +
\langle \mathcal G \rangle.
\end{align*}
If $x_s t_u e_{\alpha_u} = b_j e_{\beta_j} \in \partial \mathcal M$ with $j \in
\{ 1, \ldots, \nu \}$, Definition~\ref{defn:multMat} yields
\begin{align*}
\phi_s(t_u e_{\alpha_u} + \langle \mathcal G \rangle) = c_{1j} t_1 e_{\alpha_1}
+ \cdots + c_{\mu j} t_\mu e_{\alpha_\mu} + \langle \mathcal G \rangle = b_j
e_{\beta_j} + \langle \mathcal G \rangle = x_s t_u e_{\alpha_u} + \langle
\mathcal G \rangle.
\end{align*}
Therefore, we see that~$\phi_s$ is the multiplication by~$x_s$ for all $s \in \{
1, \ldots, n \}$. Since the multiplication in $P^r / \langle \mathcal G \rangle$
is commutative, and since the formal multiplication matrices $\mathfrak X_1,
\ldots, \mathfrak X_n$ represent the endomorphisms $\phi_1, \ldots, \phi_n$, it
follows that $\mathfrak X_k \mathfrak X_\ell = \mathfrak X_\ell \mathfrak X_k$
for all $k, \ell \in \{ 1, \ldots, n \}$, i.\,e.\ the formal multiplication
matrices $\mathfrak X_1, \ldots, \mathfrak X_n$ are pairwise commuting.

Next we show that~$\mathcal G$ is the $\mathcal M$-module border basis of
$\langle \mathcal G \rangle$ if~$D_1)$ holds. The maps
\begin{align*}
\{ e_1, \ldots, e_r \} \to P^r, \quad e_k \mapsto e_k
\end{align*}
and
\begin{align*}
\{ e_1, \ldots, e_r \} \to \langle \mathcal M \rangle_K, \quad e_k \mapsto
\begin{cases}
e_k & \text{if $e_k \in \mathcal M$}\\
\sum_{i=1}^\mu c_{ij} t_i e_{\alpha_i} & \text{if $e_k = e_{\beta_j} \in
\partial \mathcal M$}
\end{cases}
\end{align*}
induce the $P$-module epimorphism
\begin{align*}
\varphi: P^r \twoheadrightarrow \langle \mathcal M \rangle_K, \quad e_k \mapsto
\begin{cases}
e_k & \text{if $e_k \in \mathcal M$},\\
\sum_{i=1}^\mu c_{ij} t_i e_{\alpha_i} & \text{if $e_k = e_{\beta_j} \in
\partial \mathcal M$}
\end{cases}
\end{align*}
by the Universal Property of the Free Module as $\{ e_1, \ldots, e_r \} \cap
\mathcal M$ generates $\langle \mathcal M \rangle_K$ as a $P$-module according
to Lemma~\ref{thm:orderModuleModule}. Thus the Isomorphism Theorem induces the
$P$-module isomorphism
\begin{align*}
\oline{\varphi}: P^r / \ker(\varphi) \xrightarrow{\sim} \langle \mathcal M
\rangle_K, \quad e_k + \ker(\varphi) \mapsto
\begin{cases}
e_k & \text{if $e_k \in \mathcal M$},\\
\sum_{i=1}^\mu c_{ij} t_i e_{\alpha_i} & \text{if $e_k = e_{\beta_j} \in
\partial \mathcal M$.}
\end{cases}
\end{align*}
In particular, it follows that $\varepsilon_{\ker(\varphi)}(\mathcal M) =
\oline{\varphi}^{-1}(\mathcal M)$ is a $K$-vector space basis of $P^r /
\ker(\varphi)$ as~$\mathcal M$ is a $K$-vector space basis of~$\langle \mathcal
M \rangle_K$.\\
We now show that $\langle \mathcal G \rangle \subseteq \ker(\varphi)$. Let $\{
\mathfrak e_1, \ldots, \mathfrak e_\mu \} \subseteq K^\mu$ be the canonical
$K$-vector space basis of~$K^\mu$. W.\,l.\,o.\,g.\ we suppose that $e_1, \ldots,
e_\ell \in \mathcal M$ and $e_{\ell + 1}, \ldots, e_r \notin \mathcal M$ for
some $\ell \in \{ 1, \ldots, r \}$, and that $t_k e_{\alpha_k} = e_k$ for all $k
\in \{ 1, \ldots, \ell \}$. Furthermore, we let $j \in \{ 1, \ldots, \nu \}$.
We have to distinguish two cases.\\
For the first case, assume that $\mathcal O_{\beta_j} = \emptyset$. Then we have
$b_j = 1$ and $e_{\beta_j} \in \partial \mathcal M$ by
Definition~\ref{defn:border}. Hence
\begin{align*}
\varphi(\mathcal G_j) & = \varphi \left( b_j e_{\beta_j} - \sum_{i=1}^\mu c_{ij}
t_i e_{\alpha_i} \right)\\
& = b_j \circ \varphi(e_{\beta_i}) - \sum_{i=1}^\mu (c_{ij} t_i) \circ
\varphi(e_{\alpha_i})\\
& = 1 \circ \sum_{i=1}^\mu c_{ij} t_i e_{\alpha_i} - \sum_{i=1}^\mu c_{ij} \circ
(t_i \circ e_{\alpha_i})\\
& = \sum_{i=1}^\mu c_{ij} t_i e_{\alpha_i} - \sum_{i=1}^\mu (t_1 e_{\alpha_1},
\ldots, t_\mu e_{\alpha_\mu}) c_{ij} t_i(\mathfrak X_1, \ldots, \mathfrak
X_n) \mathfrak e_{\alpha_i}^{\tr}\\
& = \sum_{i=1}^\mu c_{ij} t_i e_{\alpha_i} - \sum_{i=1}^\mu (t_1 e_{\alpha_1},
\ldots, t_\mu e_{\alpha_\mu}) c_{ij} \mathfrak e_i^{\tr}\\
& = \sum_{i=1}^\mu c_{ij} t_i e_{\alpha_i} - \sum_{i=1}^\mu c_{ij} t_i
e_{\alpha_i}\\
& = 0.
\end{align*}
For the second case, assume that $\mathcal O_{\beta_j} \neq \emptyset$. Then
$\deg(b_j) \geq 1$ and $e_{\beta_j} \in \mathcal M$ by
Definition~\ref{defn:border}. Then there exist an $s \in \{ 1, \ldots, n \}$
and a $k \in \{ 1, \ldots, \mu \}$ such that $b_j e_{\beta_j} = x_s t_k
e_{\alpha_k}$ and we have $\beta_j = \alpha_k$ by Definition~\ref{defn:border}.
Thus we see that
\begin{align*}
\varphi(\mathcal G_j) & = \varphi \left( b_j e_{\beta_j} - \sum_{i=1}^\mu c_{ij}
t_i e_{\alpha_i} \right)\\
& = b_j \circ \varphi(e_{\beta_j}) - \sum_{i=1}^\mu (c_{ij} t_i) \circ
\varphi(e_{\alpha_i})\\
& = b_j \circ e_{\beta_j} - \sum_{i=1}^\mu c_{ij} \circ (t_i \circ
e_{\alpha_i})\\
& = (t_1 e_{\alpha_1}, \ldots, t_\mu e_{\alpha_\mu}) \left( b_j(\mathfrak X_1,
\ldots, \mathfrak X_n) \mathfrak e_{\beta_j}^{\tr} - \sum_{i=1}^\mu c_{ij}
t_i(\mathfrak X_1, \ldots, \mathfrak X_n) \mathfrak e_{\alpha_i}^{\tr} \right)\\
& = (t_1 e_{\alpha_1}, \ldots, t_\mu e_{\alpha_\mu}) \left( \mathfrak X_s
t_k(\mathfrak X_1, \ldots, \mathfrak X_n) \mathfrak e_{\alpha_k}^{\tr} -
\sum_{i=1}^\mu c_{ij} \mathfrak e_i^{\tr} \right)\\
& = (t_1 e_{\alpha_1}, \ldots, t_\mu e_{\alpha_\mu}) \left( \mathfrak X_s
\mathfrak e_k^{\tr} - \sum_{i=1}^\mu c_{ij} \mathfrak e_i^{\tr} \right)\\
& = (t_1 e_{\alpha_1}, \ldots, t_\mu e_{\alpha_\mu}) \left( \sum_{i=1}^\mu
\xi_{ij}^{(s)} \mathfrak e_i^{\tr} - \sum_{i=1}^\mu c_{ij} \mathfrak e_i^{\tr}
\right)\\
& = (t_1 e_{\alpha_1}, \ldots, t_\mu e_{\alpha_\mu}) \left( \sum_{i=1}^\mu
c_{ij} \mathfrak e_i^{\tr} - \sum_{i=1}^\mu c_{ij} \mathfrak e_i^{\tr} \right)\\
& = 0.
\end{align*}
Altogether, it follows that $\langle \mathcal G \rangle \subseteq
\ker(\varphi)$.\\
The Universal Property of the Residue Class Module now induces the $P$-module
epimorphism
\begin{align*}
\psi & : P^r / \langle \mathcal G \rangle \twoheadrightarrow P^r /
\ker(\varphi)\\
& \alignLongFormula e_k + \langle \mathcal G \rangle \mapsto \begin{cases}
e_k + \ker(\varphi) & \text{if $e_k \in \mathcal M$},\\
\sum_{i=1}^\mu c_{ij} t_i e_{\alpha_i} + \ker(\varphi) & \text{if $e_k =
e_{\beta_j} \in \partial \mathcal M$}.
\end{cases}
\end{align*}
Moreover, we have $\psi(\varepsilon_{\langle \mathcal G \rangle}(\mathcal M)) =
\varepsilon_{\ker(\varphi)}(\mathcal M)$. Since~$\varepsilon_{\langle \mathcal G
\rangle}(\mathcal M)$ also generates the $K$-vector space $P^r / \langle
\mathcal G \rangle$ by Corollary~\ref{thm:genSet} and
since~$\varepsilon_{\ker(\varphi)}(\mathcal M)$ is a $K$-vector space basis
of~$P^r / \ker(\varphi)$, we see that~$\varepsilon_{\langle \mathcal G
\rangle}(\mathcal M)$ is also a $K$-vector space basis of~$P^r / \langle
\mathcal G \rangle$.
In particular, we have
\begin{align*}
\mu \geq \# \varepsilon_{\langle \mathcal G \rangle}(\mathcal M) \geq \#
\varepsilon_{\ker(\varphi)}(\mathcal M) = \# \oline{\varphi}^{-1}(\mathcal M) =
\# \mathcal M = \mu.
\end{align*}
Altogether, it follows that $\# \varepsilon_{\mathcal G}(\mathcal M) = \mu$ and
Definition~\ref{defn:moduleBB} yields that~$\mathcal G$ is the $\mathcal
M$-module border basis of~$\langle \mathcal G \rangle$.

Finally, we show that~$D_1)$ is equivalent to~$D_2)$. Let $p, i \in \{ 1,
\ldots, \mu \}$, and let $s, u \in \{ 1, \ldots, n \}$ such that $s \neq u$. In
order to show this equivalence, we translate the commutativity condition
$\mathfrak e_p \mathfrak X_s \mathfrak X_u \mathfrak e_i^{\tr} = \mathfrak e_p
\mathfrak X_u \mathfrak X_s \mathfrak e_i^{\tr}$ back into the language of
$\langle \mathcal M \rangle_K$. As the resulting condition depends on the
position of $t_i e_{\beta_i}$ relative to the border of~$\mathcal M$, we
distinguish four cases.\\[1ex]
$\begin{array}{|c|c|}\hline t_k e_{\alpha_k} & t_\ell e_{\alpha_\ell}\\\hline
t_i e_{\alpha_i} & t_j e_{\alpha_j}\\\hline
\end{array} \quad \text{\textbf{First case:}} \quad x_s x_u t_i e_{\alpha_i}
\in \mathcal M$\\[1ex]
Since $\mathcal M$ is an order module, we also see that we have $x_s t_i
e_{\alpha_i}, x_u t_i e_{\alpha_i} \in \mathcal M$. Say, $x_s t_i e_{\alpha_i} =
t_j e_{\alpha_j}$, $x_u t_i e_{\alpha_i} = t_k e_{\alpha_k}$, and $x_s x_u t_i
e_{\alpha_i} = t_\ell e_{\alpha_\ell}$ where $j, k, \ell \in \{ 1, \ldots, \mu
\}$. Then we have
\begin{align*}
\mathfrak e_p \mathfrak X_s \mathfrak X_u \mathfrak e_i^{\tr} = \mathfrak e_p
\mathfrak X_s \mathfrak e_k^{\tr} = \xi_{pk}^{(s)} = \delta_{p \ell} =
\xi_{pj}^{(u)} = \mathfrak e_p \mathfrak X_u \mathfrak e_j^{\tr} = \mathfrak
e_p \mathfrak X_u \mathfrak X_s \mathfrak e_i^{\tr},
\end{align*}
i.\,e.\ the commutativity condition holds in this case by
Definition~\ref{defn:multMat}.\\[1ex]
$\begin{array}{|c|c|}\hline
t_k e_{\alpha_k} & b_\ell e_{\beta_\ell}\\\hline
t_i e_{\alpha_i} & t_j e_{\alpha_j}\\\hline
\end{array} \quad \text{\textbf{Second case:}} \quad x_s x_u t_i e_{\alpha_i}
\in \partial \mathcal M \quad \text{and} \quad x_s t_i e_{\alpha_i}, x_u t_i
e_{\alpha_i} \in \mathcal M$\\[1ex]
Say, $x_s t_i e_{\alpha_i} = t_j e_{\alpha_j}$, $x_u t_i e_{\alpha_i} = t_k
e_{\alpha_k}$, and $x_s x_u t_i e_{\alpha_i} = b_\ell e_{\beta_\ell}$ where $j,
k \in \{ 1, \ldots, \mu \}$ and $\ell \in \{ 1, \ldots, \nu \}$. Then we have
\begin{align*}
\mathfrak e_p \mathfrak X_s \mathfrak X_u \mathfrak e_i^{\tr} = \mathfrak e_p
\mathfrak X_s \mathfrak e_k^{\tr} = \xi_{pk}^{(s)} = c_{p \ell} = \xi_{pj}^{(u)}
= \mathfrak e_p \mathfrak X_u \mathfrak e_j^{\tr} = \mathfrak e_p \mathfrak X_u
\mathfrak X_s \mathfrak e_i^{\tr},
\end{align*}
i.\,e.\ the commutativity condition holds in this case by
Definition~\ref{defn:multMat}, again.\\[1ex]
$\begin{array}{|c|c|}\hline b_k e_{\beta_k} & b_\ell e_{\beta_\ell}\\\hline
t_i e_{\alpha_i} & t_j e_{\alpha_j}\\\hline
\end{array} \quad \text{\textbf{Third case:}} \quad x_s t_i e_{\alpha_i} \in
\mathcal M \quad \text{and} \quad x_u t_i e_{\alpha_i} \in \partial \mathcal
M$\\[1ex]
Since $\oline{\partial \mathcal M}$ and $\mathcal M$ are order modules, we see
that $x_s x_u t_i e_{\alpha_i} \in \partial \mathcal M$ by
Definition~\ref{defn:orderModule} and Definition~\ref{defn:border}. Say, $x_s t_i
e_{\alpha_i} = t_j e_{\alpha_j}$, $x_u t_i e_{\alpha_i} = b_k e_{\beta_k}$, and
$x_s x_u t_i e_{\alpha_i} = b_\ell e_{\beta_\ell}$ where $j \in \{ 1, \ldots,
\mu \}$ and $k, \ell \in \{ 1, \ldots, \nu \}$. Then we have
\begin{align*}
\mathfrak e_p \mathfrak X_s \mathfrak X_u \mathfrak e_i^{\tr} & = \mathfrak
e_p \mathfrak X_s (c_{1k}, \ldots, c_{\mu k})^{\tr}\\
& = \sum_{m=1}^\mu \xi_{pm}^{(s)} c_{mk}\\
& = \sum_{\substack{m \in \{ 1, \ldots, \mu \}\\ x_s t_m e_{\alpha_m} \in
\mathcal M}} \delta_{p \varrho_s(m)} c_{mk} + \sum_{\substack{m \in \{ 1,
\ldots, \mu \}\\ x_s t_m e_{\alpha_m} \in \partial \mathcal M}} c_{p
\varrho_s(m)} c_{mk},
\end{align*}
and
\begin{align*}
\mathfrak e_p \mathfrak X_u \mathfrak X_s \mathfrak e_i^{\tr} = \mathfrak e_p
\mathfrak X_u \mathfrak e_j^{\tr} = \xi_{pj}^{(u)} = c_{p \ell}
\end{align*}
by Definition~\ref{defn:multMat}. Thus the commutativity condition holds in this
case if and only if equation~$(1)$ is satisfied for~$s$, $u$ and~$p$.\\[1ex]
$\begin{array}{|c|c|}\hline
b_k e_{\beta_k} & *\\\hline
t_i e_{\alpha_i} & b_j e_{\beta_j}\\\hline
\end{array} \quad \text{\textbf{Fourth case:}} \quad x_s t_i e_{\alpha_i} \in
\partial \mathcal M \quad \text{and} \quad x_u t_i e_{\alpha_i} \in \partial
\mathcal M$\\[1ex]
Say, $x_s t_i e_{\alpha_i} = b_j e_{\beta_j}$ and $x_u t_i e_{\alpha_i} = b_k
e_{\beta_k}$ where $j, k \in \{ 1, \ldots, \nu \}$. Then we have
\begin{align*}
\mathfrak e_p \mathfrak X_s \mathfrak X_u \mathfrak e_i^{\tr} & = \mathfrak
e_p \mathfrak X_s (c_{1k}, \ldots, c_{\mu k})^{\tr}\\
& = \sum_{m=1}^\mu \xi_{pm}^{(s)} c_{mk}\\
& = \sum_{\substack{m \in \{ 1, \ldots, \mu \}\\ x_s t_m e_{\alpha_m} \in
\mathcal M}} \delta_{p \varrho_s(m)} c_{mk} + \sum_{\substack{m \in \{ 1,
\ldots, \mu \}\\ x_s t_m e_{\alpha_m} \in \partial \mathcal M}} c_{p
\varrho_s(m)} c_{mk},
\end{align*}
and
\begin{align*}
\mathfrak e_p \mathfrak X_u \mathfrak X_s \mathfrak e_i^{\tr} & = \mathfrak
e_p \mathfrak X_u (c_{1j}, \ldots, c_{\mu j})^{\tr}\\
& = \sum_{m=1}^\mu \xi_{pm}^{(u)} c_{mj}\\
& = \sum_{\substack{m \in \{ 1, \ldots, \mu \}\\ x_u t_m e_{\alpha_m} \in
\mathcal M}} \delta_{p \varrho_u(m)} c_{mj} + \sum_{\substack{m \in \{ 1,
\ldots, \mu \}\\ x_u t_m e_{\alpha_m} \in \partial \mathcal M}} c_{p
\varrho_u(m)} c_{mj}.
\end{align*}
by Definition~\ref{defn:multMat}. Thus the commutativity condition holds in this
case if and only if equation~$(2)$ is satisfied for~$s$, $u$ and~$p$.\\
Altogether, we have regarded all possible cases and have seen that
condition~$D_1)$ holds if and only if, the equations~$(1)$ and~$(2)$ are
satisfied for all $s, u \in \{ 1, \ldots, n \}$ such that $s \neq u$ and all $p
\in \{ 1, \ldots, \mu \}$, i.\,e.\ if and only if~$D_2)$ is satsfied.
\end{proof}

\begin{exmp}
\label{exmp:commMat}
We consider Example~\ref{exmp:multMat}, again. We have seen that the formal
multiplication matrices $\mathfrak X, \mathfrak Y \in \Mat_7(\mathbb Q)$
of~$\mathcal G$ are
\begin{align*}
\mathfrak X & = \begin{pmatrix}
0 & 0 & 1 & 0 & 0 & 0\\
1 & 0 & 0 & 0 & 0 & 0\\
0 & 0 & 0 & 1 & 0 & 0\\
0 & 0 & 0 & 0 & 1 & 0\\
0 & 0 & 0 & 0 & 0 & 1\\
-1 & 1 & 0 & 0 & 0 & 0
\end{pmatrix},
& \mathfrak Y & = \begin{pmatrix}
0 & 0 & 0 & 0 & 0 & 1\\
0 & 0 & 1 & 0 & 0 & 1\\
0 & 0 & 0 & 1 & -3 & 1\\
0 & 0 & 0 & 0 & 0 & 0\\
0 & 1 & 0 & 0 & 0 & 0\\
1 & 0 & 0 & 1 & 0 & 1
\end{pmatrix}.
\end{align*}
Since
\begin{align*}
\mathfrak X \cdot \mathfrak Y = \begin{pmatrix}
0 &  0 &  0 &  1 &  -3 &  1\\
0 &  0 &  0 &  0 &  0 &  1\\
0 &  0 &  0 &  0 &  0 &  0\\
0 &  1 &  0 &  0 &  0 &  0\\
1 &  0 &  0 &  1 &  0 &  1\\
0 &  0 &  1 &  0 &  0 &  0
\end{pmatrix} \neq \begin{pmatrix}
-1 &  1 &  0 &  0 &  0 &  0\\
-1 &  1 &  0 &  1 &  0 &  0\\
-1 &  1 &  0 &  0 &  1 &  -3\\
0 &  0 &  0 &  0 &  0 &  0\\
1 &  0 &  0 &  0 &  0 &  0\\
-1 &  1 &  1 &  0 &  1 &  0
\end{pmatrix} = \mathfrak Y \cdot \mathfrak X,
\end{align*}
condition~$D_1)$ of Theorem~\ref{thm:commMat} yields that~$\mathcal G$ is not
the $\mathcal M$-module border basis of~$\langle \mathcal G \rangle$.
\end{exmp}

%
%

The next characterization shows the connection of module border bases and the
existence of liftings of border syzygies. The definitions and proofs of this
characterization follow the corresponding ones for border bases
in~\cite[Section~5]{CharBB}.

We now investigate the connection of $\mathcal M$-module border bases and
special syzygy modules. Recall, that we call $(p_1, \ldots, p_s) \in P^s$ a
syzygy of $(\mathcal V_1, \ldots, \mathcal V_s) \in (P^r)^s$ if $p_1 \mathcal
V_1 + \cdots + p_s \mathcal V_s = 0$ holds. For a vector $(\mathcal V_1, \ldots,
\mathcal V_s) \in (P^r)^s$, the set of all syzygies of~$(\mathcal V_1, \ldots,
\mathcal V_s)$ is a $P$-submodule of~$P^s$ and is denoted by~$\Syz_P(\mathcal
V_1, \ldots, \mathcal V_s)$.

\begin{defn}
\label{defn:borderSyz}
A syzygy $(q_1, \ldots, q_\nu) \in \Syz_P(b_1 e_{\beta_1}, \ldots, b_\nu
e_{\beta_\nu}) \subseteq P^\nu$ is called a \emph{border syzygy} with respect
to~$\mathcal M$.
\end{defn}

We have defined module border bases over the free $P$-module $P^r$ with its
canonical $P$-module basis $\{ e_1, \ldots, e_r \} \subseteq P^r$. Since border
syzygies with respect to~$\mathcal M$ are vectors in the free
$P$-module~$P^\nu$, we must distinguish these two free modules. In order to make
this distinction easier, we let $\{ \varepsilon_1, \ldots, \varepsilon_\nu \}
\subseteq P^\nu$ denote the canonical $P$-module basis of~$P^\nu$. To shorten
notation, we let
\begin{align*}
\sigma_{ij} = \tfrac{\lcm(b_i, b_j)}{b_i} \varepsilon_i - \tfrac{\lcm(b_i,
b_j)}{b_j} \varepsilon_j
\end{align*}
for all $i, j \in \{ 1, \ldots, \nu \}$. Then $\sigma_{ij}$ is a fundamental
syzygies of $(b_1 e_{\beta_1}, \ldots, b_\nu e_{\beta_\nu})$ for all $i, j \in
\{ 1, \ldots, \nu \}$ such that $\beta_i = \beta_j$ and \cite[Thm.~2.3.7]{KR1}
yields that the set
\begin{align*}
\{ \sigma_{ij} \mid i, j \in \{ 1, \ldots, \nu \}, i < j, \beta_i =
\beta_j \}.
\end{align*}
is a generating system of the $P$-submodule $\Syz_P(b_1 e_{\beta_1}, \ldots,
b_\nu e_{\beta_\nu}) \subseteq P^r$. We now determine an even smaller generating
system. But before that, we have to generalize the concepts of neighbors and
neighbor syzygies from~\cite[Defn.~17/20]{CharBB}

\begin{defn}
\label{defn:neighbors}
Let $i, j \in \{ 1, \ldots, \nu \}$ with $i \neq j$.
\begin{enumerate}
  \item The border terms $b_i e_{\beta_i}, b_j e_{\beta_j} \in \partial \mathcal
  M$ are called \emph{next-door neighbors} with respect to~$\mathcal M$ if
  $\beta_i = \beta_j$ and if $b_i, b_j \in \partial \mathcal O_{\beta_i}$ are
  next-door neighbors with respect to~$\mathcal O_{\beta_i}$, i.\,e.\ if we
  have
  \begin{align*}
  x_k b_i e_{\beta_i} = b_j e_{\beta_j}
  \end{align*}
  for some $k \in \{ 1, \ldots, n \}$. In this case, the fundamental syzygy
  \begin{align*}
  \sigma_{ij} = x_k \varepsilon_i - \varepsilon_j \in \Syz_P(b_1 e_{\beta_1},
  \ldots, b_\nu e_{\beta_\nu})
  \end{align*}
  is called a \emph{next-door neighbor syzygy} with respect to~$\mathcal M$.
  \item The border terms $b_i e_{\beta_i}, b_j e_{\beta_j} \in \mathcal M$ are
  called \emph{across-the-street neighbors} with respect to~$\mathcal M$ if
  $\beta_i = \beta_j$ and if $b_i, b_j \in \partial \mathcal O_{\beta_i}$ are
  across-the-street neighbors with respect to~$\mathcal O_{\beta_i}$, i.\,e.\
  if we have
  \begin{align*}
  x_k b_i e_{\beta_i} = x_\ell b_j e_{\beta_j}
  \end{align*}
  for some $k, \ell \in \{ 1, \ldots, n \}$. In this case, the fundamental
  syzygy
  \begin{align*}
  \sigma_{ij} = x_k \varepsilon_i - x_\ell \varepsilon_j \in \Syz_P(b_1
  e_{\beta_1}, \ldots, b_\nu e_{\beta_\nu})
  \end{align*}
  is called an \emph{across-the-street neighbor syzygy} with respect
  to~$\mathcal M$.
  \item The border terms $b_i e_{\beta_i}, b_j e_{\beta_j} \in \mathcal M$ are
  called \emph{neighbors} with respect to~$\mathcal M$ if $\beta_i = \beta_j$
  and if $b_i, b_j \in \partial \mathcal O_{\beta_i}$ are neighbors with respect
  to~$\mathcal O_{\beta_i}$, i.\,e.\ if they are next-door or across-the-street
  neighbors with respect to~$\mathcal M$. In this case, the fundamental syzygy
  $\sigma_{ij} \in \Syz_P(b_1 e_{\beta_1}, \ldots, b_\nu e_{\beta_\nu})$ is
  called a \emph{neighbor syzygy} with respect to~$\mathcal M$.
\end{enumerate}
\end{defn}

\begin{exmp}
\label{exmp:neighbors}
We proceed with Example~\ref{exmp:commMat} and determine the neighbors with
respect to~$\mathcal M$. Recall, that the border of~$\mathcal M$ was
\begin{align*}
\partial \mathcal M = \{ x^2 e_1, xy e_1, y^2 e_1, x^3 e_2, x^2y e_2, xy e_2, y
e_2 \}.
\end{align*}
We see that $x \cdot xy e_2 = x^2y e_2$, $x \cdot y e_2 = xy e_2$, i.\,e.\ the
border terms $xy e_2$ and $x^2y e_2$ and the border terms $y e_2$ and $xy e_2$
are next-door neighbors with respect to~$\mathcal M$. Moreover, we have $y
\cdot x^2 e_1 = x \cdot xy e_1$, $y \cdot xy e_1 = x \cdot y^2 e_1$ and $y \cdot
x^3 e_2 = x \cdot x^2y e_2$, i.\,e.\ the border terms $x^2 e_1$ and $xy e_1$,
the border terms $xy e_1$ and $y^2 e_1$, and the border terms $x^3 e_2$ and
$x^2y e_2$ are across-the-street neighbors with respect to~$\mathcal M$.
Obviously, there are no further neighbors with respect to~$\mathcal M$.
\end{exmp}

Now we prove an analogous version of \cite[Prop.~21]{CharBB} for the border
syzyigies with respect to~$\mathcal M$.

\begin{prop}
\label{thm:borderSyzGen}
The set of all neighbor syzygies with respect to $\mathcal M$ generates the
$P$-submodule $\Syz_P(b_1 e_{\beta_1}, \ldots, b_\nu e_{\beta_\nu}) \subseteq
P^\nu$.
\end{prop}

\begin{proof}
As previously mentioned, \cite[Thm.~2.3.7]{KR1} yields that the $P$-submodule
$\Syz_P(b_1 e_{\beta_1}, \ldots, b_\nu e_{\beta_\nu}) \subseteq P^\nu$ is
generated by the set
\begin{align*}
\{ \sigma_{ij} \mid i, j \in \{ 1, \ldots, \nu \}, i < j, \beta_i = \beta_j \}.
\end{align*}
Let $i, j \in \{ 1, \ldots, \nu \}$ be with $i < j$ and $\beta_i = \beta_j$. We
prove that the fundamental syzygy $\sigma_{ij} \in \Syz_P(b_1 e_{\beta_1},
\ldots, b_\nu e_{\beta_\nu})$ is a $P$-linear combination of neighbor syzygies
with respect to~$\mathcal M$. Let $b_{ij} = \tfrac{\lcm(b_i, b_j)}{b_i}$ and
$b_{ji} = \tfrac{\lcm(b_i, b_j)}{b_j}$. Since $\beta_i = \beta_j$, we see that
$b_i, b_j \in \partial \mathcal O_{\beta_i}$ and that $\sigma_{ij} = b_{ij}
\varepsilon_i - b_{ji} \varepsilon_j$ is a syzygy of $(b_1 e_{\beta_1}, \ldots,
b_\nu e_{\beta_\nu})$ if and only if~$\sigma_{ij}$ is a syzygy of~$(b_1, \ldots,
b_\nu)$. Moreover, \cite[Prop.~21]{CharBB} yields that~$\sigma_{ij}$ is a
$P$-linear combination of fundamental syzygies~$\sigma_{k \ell}$ such that $k,
\ell \in \{ 1, \ldots, \nu \}$, $b_k, b_\ell \in \partial \mathcal O_{\beta_i}$,
i.\,e.\ $\beta_k = \beta_\ell = \beta_i$, and such that $b_k, b_\ell \in
\partial \mathcal O_{\beta_i}$ are neighbors with respect to~$\mathcal
O_{\beta_i}$. Furthermore, we see that two border terms $b_k, b_\ell \in
\partial \mathcal O_{\beta_i}$ with $k, \ell \in \{ 1, \ldots, \nu \}$ are
neighbors with respect to~$\mathcal O_{\beta_i}$ if and only if $b_k
e_{\beta_i}$ and $b_\ell e_{\beta_i}$ are neighbors with respect to~$\mathcal
M$ by Definition~\ref{defn:neighbors}. Altogether, it follows that the
fundamental syzygy~$\sigma_{ij}$ is also a $P$-linear combination of neighbor
syzygies~$\sigma_{k \ell}$ such that $k, \ell \in \{ 1, \ldots, \nu \}$ and
$b_k e_{\beta_k}$ is a neighbor of $b_\ell e_{\beta_\ell}$ with respect
to~$\mathcal M$.
\end{proof}

We are now able to define liftings of border syzygies with respect to~$\mathcal
M$  similar to \cite[Defn.~22]{CharBB}. Later, we use liftings to characterize
module border basis similar to the corresponding characterization for border
bases in~\cite[Prop.~25]{CharBB}.

\begin{defn}
\label{defn:lifting}
Let $(p_1, \ldots, p_\nu) \in \Syz_P(b_1 e_{\beta_1}, \ldots, b_\nu
e_{\beta_\nu})$ be a border syzygy with respect to~$\mathcal M$. Then we call a
syzygy
\begin{align*}
(P_1, \ldots, P_\nu) \in \Syz_P(\mathcal G_1, \ldots, \mathcal G_\nu)
\end{align*}
a \emph{lifting} of~$(p_1, \ldots, p_\nu)$, if one of the following conditions
holds for
\begin{align*}
\mathcal V = p_1 \mathcal G_1 + \cdots + p_\nu \mathcal G_\nu.
\end{align*}
\begin{enumerate}
\renewcommand{\labelenumi}{\arabic{enumi}.}
  \item We have $\mathcal V = 0$ and
  \begin{align*}
  (P_1, \ldots, P_\nu) = (p_1, \ldots, p_\nu).
  \end{align*}
  \item We have $\mathcal V \neq 0$ and
  \begin{align*}
  \deg(P_j - p_j) \leq \ind_{\mathcal M}(\mathcal V) - 1
  \end{align*}
  for all $j \in \{ 1, \ldots, \nu \}$ such that $P_j - p_j \neq 0$.
\end{enumerate}
In this situation, we also say that the border syzygy~$(p_1, \ldots, p_\nu)$
with respect to~$\mathcal M$ \emph{lifts} to the syzygy~$(P_1, \ldots, P_\nu)$
of~$(\mathcal G_1, \ldots, \mathcal G_\nu)$.
\end{defn}

\begin{thm}[Module Border Bases and Liftings of Border Syzygies]
\label{thm:liftings}
The $\mathcal M$-module border prebasis~$\mathcal G$ is the $\mathcal M$-module
border basis of~$\langle \mathcal G \rangle$ if and only if the following
equivalent conditions are satisfied.
\begin{enumerate}
\renewcommand{\labelenumi}{$E_{\arabic{enumi}})$}
  \item Every border syzygy with respect to~$\mathcal M$ lifts to a syzygy
  of~$(\mathcal G_1, \ldots, \mathcal G_\nu)$.
  \item Every neighbor syzygy with respect to~$\mathcal M$ lifts to a syzygy
  of~$(\mathcal G_1, \ldots, \mathcal G_\nu)$.
\end{enumerate}
\end{thm}

\begin{proof}
We start to show that condition~$E_1)$ is satisfied if~$\mathcal G$ is the
$\mathcal M$-module border basis of~$\langle \mathcal G \rangle$. Let
\begin{align*}
(p_1, \ldots, p_\nu) \in \Syz_P(b_1 e_{\beta_1}, \ldots, b_\nu e_{\beta_\nu})
\end{align*}
be a border syzygy with respect to~$\mathcal M$ and let
\begin{align*}
\mathcal V = p_1 \mathcal G_1 + \cdots + p_\nu \mathcal G_\nu.
\end{align*}
If $\mathcal V = 0$, we see that $(p_1, \ldots, p_\nu)$ is a lifting of $(p_1,
\ldots, p_\nu)$ by Definition~\ref{defn:lifting}. Thus we suppose that $\mathcal
V \neq 0$. Since $\mathcal V \in \langle \mathcal G \rangle \setminus \{ 0 \}$,
condition~$A_1)$ of Theorem~\ref{thm:specGen} yields a representation
\begin{align*}
\mathcal V = p_1 \mathcal G_1 + \cdots + p_\nu \mathcal G_\nu = q_1 \mathcal
G_1 + \cdots + q_\nu \mathcal G_\nu
\end{align*}
with $q_1, \ldots, q_\nu \in P$ and
\begin{align*}
\deg(q_j) \leq \ind_{\mathcal M}(\mathcal V) - 1
\end{align*}
for all $j \in \{ 1, \ldots, \nu \}$ such that $q_j \neq 0$. Let
\begin{align*}
(P_1, \ldots, P_\nu) = (p_1, \ldots, p_\nu) - (q_1, \ldots, q_\nu).
\end{align*}
Then~$(P_1, \ldots, P_\nu)$ is a syzygy of~$(\mathcal G_1, \ldots, \mathcal
G_\nu)$ by construction. Moreover, we have
\begin{align*}
\deg(P_j - p_j) = \deg(-q_j) \leq \ind_{\mathcal M}(\mathcal V) - 1
\end{align*}
for all $j \in \{ 1, \ldots, \nu \}$ with $P_j - p_j \neq 0$, i.\,e.\ $(p_1,
\ldots, p_\nu)$ lifts to~$(P_1, \ldots, P_\nu)$ by
Definition~\ref{defn:lifting}.

Since~$E_1)$ logically implies~$E_2)$ according to
Definitions~\ref{defn:borderSyz} and Definition~\ref{defn:neighbors}, it remains
to prove that~$\mathcal G$ is the $\mathcal M$-module border basis of~$\langle
\mathcal G \rangle$ if~$E_2)$ holds. For all $s \in \{ 1, \ldots, n \}$, we let
\begin{align*}
\varrho_s: \{ 1, \ldots, \mu \} \to \mathbb N, \quad i \mapsto \begin{cases}
j & \text{if $x_s t_i e_{\alpha_i} = t_j e_{\alpha_j} \in \mathcal M$,}\\
k & \text{if $x_s t_i e_{\alpha_i} = b_k e_{\beta_k} \in \partial \mathcal M$.}
\end{cases}
\end{align*}
be like in Theorem~\ref{thm:commMat}. We have to distinguish two cases.\\
Given next-door neighbors $b_k, b_\ell \in \partial \mathcal M$ with respect
to~$\mathcal M$, say
\begin{align*}
x_s b_k e_{\beta_k} = b_\ell e_{\beta_\ell}
\end{align*}
where $s \in \{ 1, \ldots, n \}$, $k, \ell \in \{ 1, \ldots, \nu \}$, and $k
\neq \ell$, the corresponding next-door neighbor syzygy with respect
to~$\mathcal M$ is of the form
\begin{align*}
\sigma_{k \ell} = x_s \varepsilon_k - \varepsilon_\ell \in \Syz_P(b_1
e_{\beta_1}, \ldots, b_\nu e_{\beta_\nu})
\end{align*}
by Definition~\ref{defn:neighbors}. If $x_s \mathcal G_k - \mathcal G_\ell \neq
0$, Definition~\ref{defn:lifting} and Proposition~\ref{thm:index} yield
\begin{align*}
\ind_{\mathcal M}(x_s \mathcal G_k - \mathcal G_\ell) = 1.
\end{align*}
Let~$\lambda_{k \ell}$ be a lifting of~$\sigma_{k \ell}$. Then there exist
$d_1, \ldots, d_\nu \in K$ such that
\begin{align*}
\lambda_{k \ell} = x_s \varepsilon_k - \varepsilon_\ell + \sum_{w=1}^\nu d_w
\varepsilon_w
\end{align*}
and~$\lambda_{k \ell}$ is also a syzygy of $(\mathcal G_1, \ldots, \mathcal
G_\nu)$ by Definition~\ref{defn:lifting}. Thus we have
\begin{align*}
0 & = x_s \mathcal G_k - \mathcal G_\ell + \sum_{w=1}^\nu d_w \mathcal G_w\\
& = x_s \left( b_k e_{\beta_k} - \sum_{m=1}^\mu c_{mk} t_m e_{\alpha_m}
\right) - \left( b_\ell e_{\beta_\ell} - \sum_{m=1}^\mu c_{m \ell} t_m
e_{\alpha_m} \right)\\
& \alignLongFormula +\sum_{w=1}^\nu d_w \left( b_w e_{\beta_w} - \sum_{m=1}^\mu
c_{mw} t_m e_{\alpha_m} \right)\\
& = -\sum_{m=1}^\mu c_{mk} (x_s t_m e_{\alpha_m}) + \sum_{m=1}^\mu c_{m \ell}
t_m e_{\alpha_m}\\
& \alignLongFormula +\sum_{w=1}^\nu d_w b_w e_{\beta_w} - \sum_{w=1}^\nu d_w
\sum_{m=1}^\mu c_{mw} t_m e_{\alpha_m}\\
& = -\sum_{\substack{m \in \{ 1, \ldots, \mu \}\\ x_s t_m e_{\alpha_m} \in
\mathcal M}} c_{mk} t_{\varrho_s(m)} e_{\alpha_{\varrho_s(m)}} -
\sum_{\substack{m \in \{ 1, \ldots, \mu \}\\ x_s t_m e_{\alpha_m} \in \partial
\mathcal M}} c_{mk} b_{\varrho_s(m)} e_{\beta_{\varrho_s(m)}}\\
& \alignLongFormula +\sum_{m=1}^\mu c_{m \ell} t_m e_{\alpha_m} + \sum_{w=1}^\nu
d_w b_w e_{\beta_w} - \sum_{w=1}^\nu d_w \sum_{m=1}^\mu c_{mw} t_m e_{\alpha_m}
\end{align*}
As~$\partial \mathcal M$ is $K$-linearly independent, comparison of the
coefficients of~$b_w e_{\beta_w}$ for all $w \in \{ 1, \ldots, \nu \}$ yields
\begin{align*}
d_w = \begin{cases}
c_{mk} & \text{if $b_w e_{\beta_w} \in x_s \mathcal M$,}\\
0 & \text{if $b_w e_{\beta_w} \notin x_s \mathcal M$.}
\end{cases}
\end{align*}
As $\mathcal M$ is $K$-linearly independent, comparison of the coefficients
of~$t_p e_{\alpha_p}$ for all $p \in \{ 1, \ldots, \mu \}$ yields
\begin{align*}
0 & = -\sum_{\substack{m \in \{ 1, \ldots, \mu \}\\ x_s t_m e_{\alpha_m} \in
\mathcal M}} \delta_{p \varrho_s(m)} c_{mk} + c_{p \ell} - \sum_{w=1}^\nu d_w
c_{pw}\\
& = -\sum_{\substack{m \in \{ 1, \ldots, \mu \}\\ x_s t_m e_{\alpha_m} \in
\mathcal M}} \delta_{p \varrho_s(m)} c_{mk} + c_{p \ell} - \sum_{\substack{w \in
\{ 1, \ldots, \nu \}\\ b_w e_{\beta_w} \in x_s \mathcal M}} c_{mk} c_{pw}\\
& = -\sum_{\substack{m \in \{ 1, \ldots, \mu \}\\ x_s t_m e_{\alpha_m} \in
\mathcal M}} \delta_{p \varrho_s(m)} c_{mk} + c_{p \ell} - \sum_{\substack{m \in
\{ 1, \ldots, \mu \}\\ x_s t_m e_{\alpha_m} \in \partial \mathcal M}} c_{mk}
c_{p \varrho_s(m)},
\end{align*}
i.\,e.\ the equations~$(1)$ of condition~$D_2)$ of Theorem~\ref{thm:commMat} are
satisfied.\\
Given across-the-street neighbors $b_k, b_\ell \in \partial \mathcal M$ with
respect to~$\mathcal M$, say
\begin{align*}
x_s b_k e_{\beta_k} = x_u b_j e_{\beta_j}
\end{align*}
with $s, u \in \{ 1, \ldots, n \}$, $k, j \in \{ 1, \ldots, \nu \}$, and $k \neq
j$, the corresponding across-the-street neighbor syzygy with respect
to~$\mathcal M$ is of the form
\begin{align*}
\sigma_{kj} = x_s \varepsilon_k - x_u \varepsilon_j \in \Syz_P(b_1 e_{\beta_1},
\ldots, b_\nu e_{\beta_\nu})
\end{align*}
by Definition~\ref{defn:neighbors}. If $x_s \mathcal G_k - x_u \mathcal G_j \neq
0$, Definition~\ref{defn:lifting} and Proposition~\ref{thm:index} yield
\begin{align*}
\ind_{\mathcal M}(x_s \mathcal G_k - x_u \mathcal G_j) = 1.
\end{align*}
Let~$\lambda_{kj}$ be a lifting of~$\sigma_{kj}$. Then there exist $d_1,
\ldots, d_\nu \in K$ such that
\begin{align*}
\lambda_{kj} = x_s \varepsilon_k - x_u \varepsilon_j + \sum_{w=1}^\nu d_w
\varepsilon_w,
\end{align*}
and such that $\lambda_{kj}$ is also a syzygy of $(\mathcal G_1, \ldots,
\mathcal G_\nu)$ by to Definition~\ref{defn:lifting}. Thus we have
\begin{align*}
0 & = x_s \mathcal G_k - x_u \mathcal G_j + \sum_{w=1}^\nu d_w \mathcal G_w\\
& = x_s \left( b_k e_{\beta_k} - \sum_{m=1}^\mu c_{mk} t_m e_{\alpha_m}
\right) - x_u \left( b_j e_{\beta_j} - \sum_{m=1}^\mu c_{mj} t_m e_{\alpha_m}
\right)\\
& \alignLongFormula +\sum_{w=1}^\nu d_w \left( b_w e_{\beta_w} - \sum_{m=1}^\mu
c_{mw} t_m e_{\alpha_m} \right)\\
& = -\sum_{m=1}^\mu c_{mk} (x_s t_m e_{\alpha_m}) + \sum_{m=1}^\mu c_{mj}
(x_u t_m e_{\alpha_m})\\
& \alignLongFormula +\sum_{w=1}^\nu d_w b_w e_{\beta_w} - \sum_{w=1}^\nu d_w
\sum_{m=1}^\mu c_{mw} t_m e_{\alpha_m}\\
& = -\sum_{\substack{m \in \{ 1, \ldots, \mu \}\\ x_s t_m e_{\alpha_m} \in
\mathcal M}} c_{mk} t_{\varrho_s(m)} e_{\alpha_{\varrho_s(m)}} -
\sum_{\substack{m \in \{ 1, \ldots, \mu \}\\ x_s t_m e_{\alpha_m} \in \partial
\mathcal M}} c_{mk} b_{\varrho_s(m)} e_{\beta_{\varrho_s(m)}}\\
& \alignLongFormula +\sum_{\substack{m \in \{ 1, \ldots, \mu \}\\ x_u t_m
e_{\alpha_m} \in \mathcal M}} c_{mj} t_{\varrho_u(m)} e_{\alpha_{\varrho_u(m)}} +
\sum_{\substack{m \in \{ 1, \ldots, \mu \}\\ x_u t_m e_{\alpha_m} \in \partial
\mathcal M}} c_{mj} b_{\varrho_u(m)} e_{\beta_{\varrho_u(m)}}\\
& \alignLongFormula +\sum_{w=1}^\nu d_w b_w e_{\beta_w} - \sum_{w=1}^\nu d_w
\sum_{m=1}^\mu c_{mw} t_m e_{\alpha_m}\\
\end{align*}
As~$\partial \mathcal M$ is $K$-linearly independent, comparison of the
coefficients of~$b_w e_{\beta_w}$ for all $w \in \{ 1, \ldots, \nu \}$ yields
\begin{align*}
d_w = \begin{cases}
c_{mk} - c_{mj} & \text{if $b_w e_{\beta_w} \in x_s \mathcal M \cap x_u
\mathcal M$,}\\
c_{mk} & \text{if $b_w e_{\beta_w} \in x_s \mathcal M \setminus x_u \mathcal
M$,}\\
-c_{mj} & \text{if $b_w e_{\beta_w} \in x_u \mathcal M \setminus x_s \mathcal
M$,}\\
0 & \text{if $b_w e_{\beta_w} \notin x_s \mathcal M \cup x_u \mathcal M$.}
\end{cases}
\end{align*}
As~$\mathcal M$ is $K$-linearly independent, comparison of the coefficients
of~$t_p e_{\alpha_p}$ for all $p \in \{ 1, \ldots, \mu \}$ yields
\begin{align*}
0 & = -\sum_{\substack{m \in \{ 1, \ldots, \mu \}\\ x_s t_m e_{\alpha_m} \in
\mathcal M}} \delta_{p \varrho_s(m)} c_{mk} + \sum_{\substack{m \in \{ 1,
\ldots, \mu \}\\ x_u t_m e_{\alpha_m} \in \mathcal M}} \delta_{p \varrho_u(m)}
c_{mj} - \sum_{w=1}^\nu d_w c_{pw}\\
& = -\sum_{\substack{m \in \{ 1, \ldots, \mu \}\\ x_s t_m e_{\alpha_m} \in
\mathcal M}} \delta_{p \varrho_s(m)} c_{mk} + \sum_{\substack{m \in \{ 1,
\ldots, \mu \}\\ x_u t_m e_{\alpha_m} \in \mathcal M}} \delta_{p \varrho_u(m)}
c_{mj}\\
& \alignLongFormula -\sum_{\substack{w \in \{ 1, \ldots, \nu \}\\ b_w
e_{\beta_w} \in x_s \mathcal M}} c_{mk} c_{pw} + \sum_{\substack{w \in \{ 1,
\ldots, \nu \}\\ b_w e_{\beta_w} \in x_u \mathcal M}} c_{mj} c_{pw}\\
& = -\sum_{\substack{m \in \{ 1, \ldots, \mu \}\\ x_s t_m e_{\alpha_m} \in
\mathcal M}} \delta_{p \varrho_s(m)} c_{mk} + \sum_{\substack{m \in \{ 1,
\ldots, \mu \}\\ x_u t_m e_{\alpha_m} \in \mathcal M}} \delta_{p \varrho_u(m)}
c_{mj}\\
& \alignLongFormula -\sum_{\substack{m \in \{ 1, \ldots, \mu \}\\ x_s t_m
e_{\alpha_m} \in \partial \mathcal M}} c_{mk} c_{p \varrho_s(m)} +
\sum_{\substack{m \in \{ 1, \ldots, \mu \}\\ x_u t_m e_{\alpha_m} \in \partial
\mathcal M}} c_{mj} c_{p \varrho_u(m)},
\end{align*}
i.\,e.\ the equations~$(2)$ of condition~$D_2)$ of Theorem~\ref{thm:commMat} are
satisfied.\\
Altogether, we see that condition~$D_2)$ of Theorem~\ref{thm:commMat} is
satisfied and thus~$\mathcal G$ is the $\mathcal M$-module border basis
of~$\langle \mathcal G \rangle$.
\end{proof}

%
%

At last, we use all previous characterizations above to prove Buchberger's
Criterion for Module Border Bases similar to Buchberger's Criterion for Border
Basis~\cite[Prop.~6.4.34]{KR2}. It allows us to easily determine, whether a
given module border prebasis is a module border basis, or not. But before that,
we have to define the $\SV$-vector of two border vectors in~$\mathcal G$ in the
usual way, cf.\ \cite[Page~438]{KR2}.

\begin{defn}
\label{defn:SVector}
Let $i, j \in \{ 1, \ldots, \nu \}$. Then we call the vector
\begin{align*}
\SV(\mathcal G_i, \mathcal G_j) = \tfrac{\lcm(b_i, b_j)}{b_i} \mathcal G_i -
\tfrac{\lcm(b_i, b_j)}{b_j} \mathcal G_j \in \langle \mathcal G \rangle
\subseteq P^r
\end{align*}
the \emph{$\SV$-vector} of~$\mathcal G_i$ and~$\mathcal G_j$.
\end{defn}

\begin{thm}[Buchberger's Criterion for Module Border Bases]
\label{thm:buchbCrit}
The $\mathcal M$-module border prebasis~$\mathcal G$ is the $\mathcal M$-module
border basis of~$\langle \mathcal G \rangle$ if and only if the following
equivalent conditions are satisfied.
\begin{enumerate}
\renewcommand{\labelenumi}{$F_{\arabic{enumi}})$}
  \item We have $\NR_{\mathcal G}(\SV(\mathcal G_i, \mathcal G_j)) = 0$ for all
  $i, j \in \{ 1, \ldots, \nu \}$.
  \item We have $\NR_{\mathcal G}(\SV(\mathcal G_i, \mathcal G_j)) = 0$ for all
  $i, j \in \{ 1, \ldots, \nu \}$ such that the border terms $b_i e_{\beta_i},
  b_j e_{\beta_j} \in \partial \mathcal M$ are neighbors with respect
  to~$\mathcal M$.
\end{enumerate}
\end{thm}

\begin{proof}
We start to show that condition~$F_1)$ is satisfied if~$\mathcal G$ is the
$\mathcal M$-module border basis of~$\langle \mathcal G \rangle$. Let $i, j \in
\{ 1, \ldots, \nu \}$. If we have $\SV(\mathcal G_i, \mathcal G_j) = 0$, we
trivially get $\NR_{\mathcal G}(\SV(\mathcal G_i, \mathcal G_j)) = 0$ by
Definition~\ref{defn:NR}. Thus suppose that $\SV(\mathcal G_i, \mathcal G_j)
\neq 0$. We apply the Module Border Division Algorithm~\ref{thm:divAlg}
to~$\SV(\mathcal G_i, \mathcal G_j)$ and~$\mathcal G$ to obtain a representation
\begin{align*}
\SV(\mathcal G_i, \mathcal G_j) = \mathcal V + \NR_{\mathcal G}(\SV(\mathcal
G_i, \mathcal G_j)) \in \langle \mathcal G \rangle
\end{align*}
with $\mathcal V \in \langle \mathcal G \rangle$ and $\NR_{\mathcal
G}(\SV(\mathcal G_i, \mathcal G_j)) \in \langle \mathcal M \rangle_K$. Then we
have
\begin{align*}
\NR_{\mathcal G}(\SV(\mathcal G_i, \mathcal G_j)) = \SV(\mathcal G_i, \mathcal
G_j) - \mathcal V \in \langle \mathcal G \rangle \cap \langle \mathcal M
\rangle_K = \{ 0 \}
\end{align*}
by Corollary~\ref{thm:char}.

Since~$F_2)$ is a logical consequence of~$F_1)$ by
Definition~\ref{defn:neighbors}, it remains to prove that~$\mathcal G$ is the
$\mathcal M$-module border basis of~$\langle \mathcal G \rangle$ if~$F_2)$
holds. Let $b_i e_{\beta_i}, b_j e_{\beta_j} \in \partial \mathcal M$ be
next-door neighbors with respect to~$\mathcal M$, i.\,e.\ $i \neq j$ and
\begin{align*}
x_s b_i e_{\beta_i} = b_j e_{\beta_j}
\end{align*}
for some $s \in \{ 1, \ldots, n \}$ by Definition~\ref{defn:neighbors}. Let
\begin{align*}
\sigma_{ij} = x_s \varepsilon_i - \varepsilon_j \in \Syz_P(b_1 e_{\beta_1},
\ldots, b_\nu e_{\beta_\nu})
\end{align*}
be the corresponding next-door neighbor syzygy with respect to~$\mathcal M$. If
we have $\SV(\mathcal G_i, \mathcal G_j) = 0$, we see that~$\sigma_{ij}$ is a
lifting of~$\sigma_{ij}$ according to Definition~\ref{defn:lifting}. Thus
suppose that $\SV(\mathcal G_i, \mathcal G_j) \neq 0$. Since $\NR_{\mathcal
G}(\SV(\mathcal G_i, \mathcal G_j)) = 0$ according to~$F_2)$, the Module Border
Division Algorithm~\ref{thm:divAlg} applied to the $\SV$-vector~$\SV(\mathcal
G_i, \mathcal G_j)$ and~$\mathcal G$ yields a representation
\begin{align*}
\SV(\mathcal G_i, \mathcal G_j) = p_1 \mathcal G_1 + \cdots + p_\nu \mathcal
G_\nu + \NR_{\mathcal G}(\SV(\mathcal G_i, \mathcal G_j)) = p_1 \mathcal G_1 +
\cdots + p_\nu \mathcal G_\nu
\end{align*}
with polynomials $p_1, \ldots, p_\nu \in P$ such that
\begin{align*}
\deg(p_\ell) \leq \ind_{\mathcal M}(\SV(\mathcal G_i, \mathcal G_j)) - 1
\end{align*}
for all $\ell \in \{ 1, \ldots, \nu \}$ with $p_\ell \neq 0$. We now prove that
$(P_1, \ldots, P_\nu) \in P^\nu$ defined by $P_i = x_s - p_i$, $P_j = 1 - p_j$,
and $P_\ell = -p_\ell$ for all $\ell \in \{ 1, \ldots, \nu \} \setminus \{ i, j
\}$ is a lifting of~$\sigma_{ij}$. By construction, we see that
\begin{align*}
(P_1, \ldots, P_\nu) \in \Syz_P(\mathcal G_1, \ldots, \mathcal G_\nu).
\end{align*}
Moreover, we have
\begin{align*}
\deg(P_i - x_s) & = \deg(-p_i) \leq \ind_{\mathcal M}(\SV(\mathcal G_i, \mathcal
G_j)) - 1\\
\intertext{if $P_i - x_s = -p_i \neq 0$,}
\deg(P_j - 1) & = \deg(-p_j) \leq \ind_{\mathcal M}(\SV(\mathcal G_i, \mathcal
G_j)) - 1\\
\intertext{if $P_j - 1 = -p_j \neq 0$, and}
\deg(P_\ell - 0) & = \deg(-p_\ell) \leq \ind_{\mathcal M}(\SV(\mathcal G_i,
\mathcal G_j)) - 1
\end{align*}
for all $\ell \in \{ 1, \ldots, \nu \} \setminus \{ i, j \}$ with $P_\ell - 0 =
-p_\ell \neq 0$. Hence $(P_1, \ldots, P_\nu)$ is a lifting of the next-door
neighbor syzygy~$\sigma_{ij}$ with respect to~$\mathcal M$ by
Definition~\ref{defn:lifting}.\\
Now suppose that $b_i e_{\beta_i}, b_j e_{\beta_j} \in \partial \mathcal M$ are
across-the-street neighbors with respect to~$\mathcal M$, i.\,e.\ $i \neq j$
and
\begin{align*}
x_s b_i e_{\beta_i} = x_u b_j e_{\beta_j}
\end{align*}
for some $s, u \in \{ 1, \ldots, n \}$ by Definition~\ref{defn:neighbors}. Let
\begin{align*}
\sigma_{ij} = x_s \varepsilon_i - x_u \varepsilon_j \in \Syz_P(b_1 e_{\beta_1},
\ldots, b_\nu e_{\beta_\nu})
\end{align*}
be the corresponding across-the-street neighbor syzygy with respect to~$\mathcal
M$. If $\SV(\mathcal G_i, \mathcal G_j) = 0$, we see that~$\sigma_{ij}$ is a
lifting of~$\sigma_{ij}$ according to Definition~\ref{defn:lifting}. Thus
suppose that $\SV(\mathcal G_i, \mathcal G_j) \neq 0$. Since $\NR_{\mathcal
G}(\SV(\mathcal G_i, \mathcal G_j)) = 0$ according to~$F_2)$, the Module Border
Division Algorithm~\ref{thm:divAlg} applied to the $\SV$-vector~$\SV(\mathcal
G_i, \mathcal G_j)$ and~$\mathcal G$ yields a representation
\begin{align*}
\SV(\mathcal G_i, \mathcal G_j) = p_1 \mathcal G_1 + \cdots + p_\nu \mathcal
G_\nu + \NR_{\mathcal G}(\SV(\mathcal G_i, \mathcal G_j)) = p_1 \mathcal G_1 +
\cdots + p_\nu \mathcal G_\nu
\end{align*}
with polynomials $p_1, \ldots, p_\nu \in P$ such that
\begin{align*}
\deg(p_\ell) \leq \ind_{\mathcal M}(\SV(\mathcal G_i, \mathcal G_j)) - 1
\end{align*}
for all $\ell \in \{ 1, \ldots, \nu \}$ with $p_\ell \neq 0$. We now prove that
$(P_1, \ldots, P_\nu) \in P^\nu$ defined by $P_i = x_s - p_i$, $P_j = x_u -
p_j$, and $P_\ell = -p_\ell$ for all $\ell \in \{ 1, \ldots, \nu \} \setminus \{
i, j \}$ is a lifting of~$\sigma_{ij}$. By construction, we see that
\begin{align*}
(P_1, \ldots, P_\nu) \in \Syz_P(\mathcal G_1, \ldots, \mathcal G_\nu).
\end{align*}
Moreover, we have
\begin{align*}
\deg(P_i - x_s) & = \deg(-p_i) \leq \ind_{\mathcal M}(\SV(\mathcal G_i, \mathcal
G_j)) - 1\\
\intertext{if $P_i - x_s = -p_i \neq 0$,}
\deg(P_j - x_u) & = \deg(-p_j) \leq \ind_{\mathcal M}(\SV(\mathcal G_i, \mathcal
G_j)) - 1\\
\intertext{if $P_j - x_u = -p_j \neq 0$, and}
\deg(P_\ell - 0) & = \deg(-p_\ell) \leq \ind_{\mathcal M}(\SV(\mathcal G_i,
\mathcal G_j)) - 1
\end{align*}
for all $\ell \in \{ 1, \ldots, \nu \} \setminus \{ i, j \}$ with $P_\ell - 0 =
-p_\ell \neq 0$. Hence the vector~$(P_1, \ldots, P_\nu)$ is a lifting of the
across-the-street neighbor syzygy~$\sigma_{ij}$ with respect to~$\mathcal M$ by
Definition~\ref{defn:lifting}.\\
Altogether, we have proven that every neighbor syzygy with respect to~$\mathcal
M$ lifts to a syzygy of~$(\mathcal G_1, \ldots, \mathcal G_\nu)$. Therefore,
condition~$E_2)$ of Theorem~\ref{thm:liftings} yields that~$\mathcal G$ is the
$\mathcal M$-module border basis.
\end{proof}

\begin{exmp}
\label{exmp:buchbCrit}
We consider Example~\ref{exmp:neighbors}, again. Recall, that the $\mathcal
M$-module border prebasis was of the form $\mathcal G = \{
\mathcal G_1, \ldots, \mathcal G_7 \} \subseteq P^2$ with
\begin{align*}
\mathcal G_1 & = x^2 e_1 - y e_1 + e_2,\\
\mathcal G_2 & = xy e_1 - e_2,\\
\mathcal G_3 & = y^2 e_1 - x e_2,\\
\mathcal G_4 & = x^3 e_2 - e_1,\\
\mathcal G_5 & = x^2y e_2 - e_1 - e_2,\\
\mathcal G_6 & = xy e_2 + 3 e_1,\\
\mathcal G_7 & = y e_2 - x e_1 - y e_1 - e_1 - e_2,
\end{align*}
and that $b_1 e_{\beta_1} = x^2 e_1$ and $b_2 e_{\beta_2} = xy e_1$ are
across-the-street neighbors with respect to~$\mathcal M$. We have already seen
in Example~\ref{exmp:commMat} that~$\mathcal G$ is not the $\mathcal M$-module
border basis of~$\langle \mathcal G \rangle$. Since
\begin{align*}
\NR_{\mathcal G}(\SV(\mathcal G_1, \mathcal G_2)) = \NR_{\mathcal G}(y
\mathcal G_1 - x \mathcal G_2) = x e_1 + y e_1 + e_1 + e_2 \neq 0,
\end{align*}
condition~$F_2)$ of Buchberger's Criterion for Module Border
Bases~\ref{thm:buchbCrit} also yields that~$\mathcal G$ is not the $\mathcal
M$-module border basis of~$\langle \mathcal G \rangle$.
\end{exmp}

%
%

\section{The Module Border Basis Algorithm}
\label{sect:moduleBBAlg}

In this final section of Part~\ref{part:freeMod}, we determine the Module
Border Basis Algorithm~\ref{thm:moduleBBAlg} that allows us to compute module
border bases of arbitrary $P$-submodules of~$P^r$ with finite $K$-codimension
in~$P^r$ with linear algebra techniques. The algorithm was originally developed
in~\cite[Prop.~18]{CompBB}. Our algorithm is a straightforward adaption of the
Border Basis Algorithm as described in~\cite[Thm.~6.4.36]{KR2}.

But first of all, we have to describe what we mean by reducing a vector against
a matrix as it was described in~\cite[p.~392]{KR2}.

\begin{defn}
\label{defn:reduceVM}
Let $r, s \in \mathbb N$.
\begin{enumerate}
  \item Let $(c_1, \ldots, c_s) \in K^s$, $(d_1, \ldots, d_s) \in K^s \setminus
  \{ 0 \}$, and let $i \in \{ 1, \ldots, s \}$ be minimal such that $d_i \neq
  0$. Then the $i^\text{th}$~component of
  \begin{align*}
  (c_1, \ldots, c_s) - \tfrac{c_i}{d_i} (d_1, \ldots, d_s) \in K^s
  \end{align*}
  is~$0$. If $c_i \neq 0$ in this situation, we say that $(d_1, \ldots, d_s)$
  is a \emph{reduceer} of~$(c_1, \ldots, c_s)$.
  \item Let $v \in K^s$ and $\mathfrak M \in \Mat_{r,s}(K)$. We say that~$v$ can
  be \emph{reduced} against~$\mathfrak M$, if there is a row vector $w \in K^s$
  of~$\mathfrak M$ such that~$w$ is a reducer of~$v$.
\end{enumerate}
\end{defn}

Before we generalize the Border Basis Algorithm~\cite[Thm.~6.4.36]{KR2} to the
module setting, we have to generalize the auxiliary
algorithm~\cite[Lemma~6.4.35]{KR2}.

\begin{algorithm}[H]
\caption{${\tt computeOrderModule}(d, \{ \mathcal V_1, \ldots, \mathcal
V_\varrho \}, \sigma)$}
\begin{algorithmic}[1]
\label{algo:computeOM}
  \REQUIRE $d \in \mathbb N$,\\
  $\varrho \in \mathbb N \setminus \{ 0 \}$ and $\{ \mathcal V_1, \ldots,
  \mathcal V_\varrho \} \subseteq \langle \mathbb T^n_{\leq d} \langle e_1,
  \ldots, e_r \rangle \rangle_K \setminus \{ 0 \}$,\\
  $V \assign \langle \mathcal V_1, \ldots, \mathcal V_\varrho \rangle_K
  \subseteq \langle \mathbb T^n_{\leq d} \langle e_1, \ldots, e_r \rangle
  \rangle_K$ is a $K$-vector subspace such that 
  \begin{align*}
  (V + x_1 V + \cdots + x_n V) \cap \langle \mathbb T^n_{\leq d} \langle e_1,
  \ldots, e_r \rangle \rangle_K = V,
  \end{align*}
  $\sigma$~is a degree compatible term ordering on~$\mathbb T^n$
  \STATE $L \assign \mathbb T^n_{\leq d} \langle e_1, \ldots, e_r \rangle$
  \STATE Let $\ell_1, \ldots, \ell_s \in \mathbb T^n$ be terms and $u_1, \ldots,
  u_s \in \{ 1, \ldots, r \}$ be indices such that $L = \{ \ell_1 e_{u_1},
  \ldots, \ell_s e_{u_s} \}$ and  $\ell_1 e_{u_1} >_{\sigma\Pos} \ell_2 e_{u_2}
  >_{\sigma\Pos} \cdots >_{\sigma\Pos} \ell_s
  e_{u_s}$.\label{algo:computeOM-terms}
  \STATE Determine a $K$-vector space basis $\{ \mathcal B_1, \ldots, \mathcal
  B_k \} \subseteq \langle L \rangle_K$ of~$V$.\label{algo:computeOM-basis}
  \FOR{$i = 1, \ldots, k$}
    \STATE Determine $a_{i1}, \ldots, a_{is} \in K$ such that $\mathcal B_i =
    a_{i1} \ell_1 e_{u_1} + \cdots + a_{is} \ell_s
    e_{u_s}$.\label{algo:computeOM-coords}
  \ENDFOR
  \STATE $\mathfrak V \assign (a_{ij})_{1 \leq i \leq k, 1 \leq j \leq s} \in
  \Mat_{k,s}(K)$.\label{algo:computeOM-initV}
  \STATE Compute a row echolon form $\mathfrak W \in \Mat_{k,s}(K)$
  of~$\mathfrak V$ using row operations.\label{algo:computeOM-rowEcho}
  \STATE Let $\mathcal M \subseteq L$ be the set of terms in~$L$ corresponding
  to the pivot-free columns of~$\mathfrak W$, i.\,e.\ the columns of~$\mathfrak
  W$ in which no row of~$\mathfrak W$ has its first non-zero
  entry.\label{algo:computeOM-OM}
  \RETURN $\mathcal M$
\end{algorithmic}
\end{algorithm}

\begin{lem}
\label{thm:moduleBBAlgLemma}
Let $d \in \mathbb N$ and let $L = \mathbb T^n_{\leq d} \langle e_1, \ldots, e_r
\rangle$. Moreover, let $\varrho \in \mathbb N \setminus \{ 0 \}$ and $V =
\langle \mathcal V_1, \ldots, \mathcal V_\varrho \rangle \subseteq L$ with $\{
\mathcal V_1, \ldots, \mathcal V_\varrho \} \subseteq \langle L \rangle_K
\setminus \{ 0 \}$ be a $K$-vector subspace such that
\begin{align*}
(V + x_1 V + \cdots + x_n V) \cap \langle L \rangle_K = V.
\end{align*}
Let $\sigma$ be a degree compatible term ordering on~$\mathbb T^n$. Then
Algorithm~\ref{algo:computeOM} is actually an algorithm and the result
\begin{align*}
\mathcal M \assign {\tt computeOrderModule}(d, \{ \mathcal V_1, \ldots, \mathcal
V_\varrho \}, \sigma)
\end{align*}
of Algorithm~\ref{algo:computeOM} applied to the input data $d$, $\{ \mathcal
V_1, \ldots, \mathcal V_\varrho \}$, and $\sigma$ satisfies the following
conditions.
\begin{enumerate}
\renewcommand{\labelenumi}{\roman{enumi})}
  \item The set $\mathcal M \subseteq L$ is an order module.
  \item The set $\varepsilon_V(\mathcal M)$ is a $K$-vector space basis of
  $\langle L \rangle_K / V$.
  \item We have $\# \varepsilon_V(\mathcal M) = \# \mathcal M$.
\end{enumerate}
\end{lem}

\begin{proof}
Firstly, we show that the procedure is actually an algorithm. Since the
operations in line~\ref{algo:computeOM-basis}, line~\ref{algo:computeOM-coords},
and line~\ref{algo:computeOM-rowEcho} can be computed with linear algebra
techniques, and since the procedure is obviously finite, the procedure is an
algorithm.

Secondly, we show the correctness of the algorithm. We start with the proof that
the set $\varepsilon_V(\mathcal M)$ is a $K$-vector space basis of $\langle L
\rangle_K / V$ and $\# \varepsilon_V(\mathcal M) = \# \mathcal M$. Let
\begin{align*}
\mathcal M = \{ \ell_{j_1} e_{u_{j_1}}, \ldots, \ell_{j_w} e_{u_{j_w}} \}
\end{align*}
with $j_1, \ldots, j_w \in \{ 1, \ldots, s \}$, and let $c_1, \ldots, c_w \in K$
be such that
\begin{align*}
\mathcal V = c_1 \ell_{j_1} e_{u_{j_1}} + \cdots + c_w \ell_{j_w} e_{u_{j_w}}
\in V.
\end{align*}
Let $\{ \mathfrak e_1, \ldots, \mathfrak e_s \}$ denote the canonical $K$-vector
space basis of~$K^s$.
Then the corresponding vector
\begin{align*}
c_1 \mathfrak e_{j_1} + \cdots + c_v \mathfrak e_{j_w} \in K^s
\end{align*}
of~$\mathcal V$ has all its non-zero entries in the columns corresponding
to~$\mathcal M$ by line~\ref{algo:computeOM-OM}. As~$\mathfrak W$ is in row
echolon form according to line~\ref{algo:computeOM-rowEcho}, this vector cannot
be further reduced against~$\mathfrak W$ by Definition~\ref{defn:reduceVM}.
Since the rows of~$\mathfrak V$ correspond to the $K$-vector space basis $\{
\mathcal B_1, \ldots, \mathcal B_k \}$ by the construction in
line~\ref{algo:computeOM-coords}, and since~$\mathfrak W$ corresponds to these
basis elements by line~\ref{algo:computeOM-basis} and
line~\ref{algo:computeOM-rowEcho}, and since we also have $\mathcal V \in V$, it
follows that $\mathcal V = 0$. Thus we get $c_1 = \cdots = c_w = 0$, i.\,e.\
$\varepsilon_V(\mathcal M)$ is $K$-linearly independent. In particular, we get
$\# \varepsilon_V(\mathcal M) = \# \mathcal M$. Let
\begin{align*}
\mathcal W = d_1 \ell_1 e_{u_1} + \cdots + d_s \ell_s e_{u_s} \in \langle L
\rangle_K
\end{align*}
with $d_1, \ldots, d_s \in K$. Then the corresponding vector
\begin{align*}
d_1 \mathfrak e_1 + \cdots + d_s \mathfrak e_s \in K^s
\end{align*}
can be reduced against~$\mathfrak W$ to obtain a vector
\begin{align*}
d'_1 \mathfrak e_{j_1} + \cdots + d'_w \mathfrak e_{j_w} \in K^s
\end{align*}
with $d'_1, \ldots, d'_w \in K$ by Definition~\ref{defn:reduceVM}. Let
\begin{align*}
\mathcal W' = d'_1 \ell_{j_1} e_{u_{j_1}} + \cdots + d'_w \ell_{j_w} e_{u_{j_w}}
\in \langle \mathcal M \rangle_K
\end{align*}
be the corresponding element in~$\langle \mathcal M \rangle_K$. Then the
considerations above show that
\begin{align*}
\mathcal W + V = \mathcal W' + V
\end{align*}
and hence~$\varepsilon_V(\mathcal M)$ is also a generating set of the $K$-vector
space~$\langle L \rangle_K / V$. Altogether, we have proven
that~$\varepsilon_V(\mathcal M)$ is a $K$-vector space basis of~$\langle L
\rangle_K / V$.\\
Finally, we prove that~$\mathcal M$ is an order module. Let $\ell_i e_{u_i} \in
L \setminus \mathcal M$ with $i \in \{ 1, \ldots, s \}$ and $t \in \mathbb T^n$
be such that $t \ell_i e_{u_i} \in L$. The set~$\mathcal M$ is an order module
by Definition~\ref{defn:orderModule} if we show that $t \ell_i e_{u_i} \in L
\setminus \mathcal M$. As $\ell_i e_{u_i} \in L \setminus \mathcal M$, one row
of~$\mathfrak W$ has the form
\begin{align*}
(0, \ldots, 0, c_i, \ldots, c_s) \in K^s
\end{align*}
where $c_i \neq 0$ according to the construction of~$\mathcal M$ in
line~\ref{algo:computeOM-OM}. The corresponding vector in $\langle L \rangle_K$
is
\begin{align*}
\mathcal V = c_i \ell_i e_{u_i} + \cdots + c_s \ell_s e_{u_s} \in \langle L
\rangle_K.
\end{align*}
Moreover, line~\ref{algo:computeOM-terms} yields that $\ell_i e_{u_i} =
\LT_{\sigma\Pos}(\mathcal V)$. Hence we see that
\begin{align*}
t \ell_i e_{u_i} = \LT_{\sigma\Pos}(t \mathcal V),
\end{align*}
and the degree compatibility of~$\sigma$ yields $\Supp(t \mathcal V) \subseteq
L$. Since every line of the matrix~$\mathfrak V$ corresponds to a vector in~$V$
by line~\ref{algo:computeOM-coords}, and since~$\mathfrak W$ is constructed
from~$\mathfrak V$ using row operation by line~\ref{algo:computeOM-rowEcho}, we
see that $\mathcal V \in V$. Hence the hypothesis
\begin{align*}
(V + x_1 V + \cdots + x_n V) \cap \langle L \rangle_K = V
\end{align*}
and induction on the degree of~$t$ also yield $t \mathcal V \in V$. Thus we see
that the vector in~$K^s$ corresponding to~$t \mathcal V$ can be reduced
against~$\mathfrak W$ to zero by Definition~\ref{defn:reduceVM}. In particular,
we have to reduce the entry of this vector that corresponds to~$t \ell_i
e_{u_i}$, i.\,e.\ there has to be one row in $\mathfrak W$ which has its first
non-zero entry in the column that corresponds to the term~$t \ell_i e_{u_i}$.
Altogether, line~\ref{algo:computeOM-OM} yields $t \ell_i e_{u_i} \in L
\setminus \mathcal M$ and the claim follows.
\end{proof}

We now have all ingredients to generalize the Border Bases
Algorithm as described in~\cite[Thm.~6.4.36]{KR2} to the module setting.

\begin{algorithm}[H]
\caption{${\tt moduleBB}(\{ \mathcal B_1, \ldots, \mathcal B_k \}, \sigma)$}
\begin{algorithmic}[1]
\label{algo:moduleBB}
  \REQUIRE $k \in \mathbb N$ and $\{ \mathcal B_1, \ldots, \mathcal B_k \}
  \subseteq P^r \setminus \{ 0 \}$,\\
  $\codim_K(\langle \mathcal B_1, \ldots, \mathcal B_k \rangle, P^r) <
  \infty$,\\
  $\sigma$ is a degree compatible term ordering on~$\mathbb T^n$
  \IF{$k = 0$}
    \RETURN $(\emptyset, \emptyset)$\label{algo:moduleBB-trivial}
  \ENDIF
  \STATE $V \assign \langle \mathcal B_1, \ldots, \mathcal B_k
  \rangle_K$\label{algo:moduleBB-initV}
  \STATE $d \assign \max \{ \deg(t) \mid \ell \in \{ 1, \ldots, k \}, t e_u \in
  \Supp(\mathcal B_\ell) \}$\label{algo:moduleBB-initD}
  \REPEAT\label{algo:moduleBB-repeat}
    \STATE $L \assign \mathbb T^n_{\leq d} \langle e_1, \ldots, e_r \rangle$
    \STATE $V' \assign (V + x_1 V + \cdots + x_n V) \cap \langle L
    \rangle_K$\label{algo:moduleBB-initV'}
    \WHILE{$V \neq V'$}\label{algo:moduleBB-while}
      \STATE $V \assign V'$\label{algo:moduleBB-renewV}
      \STATE $V' \assign (V + x_1 V + \cdots + x_n V) \cap \langle L
      \rangle_K$\label{algo:moduleBB-enlargeV'}
    \ENDWHILE
    \STATE Compute $\varrho \in \mathbb N$ and $\{ \mathcal V_1, \ldots,
    \mathcal V_\varrho \} \subseteq \langle L \rangle_K \setminus \{ 0 \}$ with
    $V = \langle \mathcal V_1, \ldots, \mathcal V_\varrho
    \rangle_K$.\label{algo:moduleBB-gensV}
    \STATE $\mathcal M \assign {\tt computeOrderIdeal}(d, \{ \mathcal V_1,
    \ldots, \mathcal V_\varrho \}, \sigma)$\label{algo:moduleBB-OM}
    \STATE $d \assign d+1$\label{algo:moduleBB-incD}
  \UNTIL{$\partial \mathcal M \subseteq L$}
  \STATE Let $t_1, \ldots, t_\mu \in \mathbb T^n$ be terms and $\alpha_1,
  \ldots, \alpha_\mu \in \{ 1, \ldots, r \}$ be indices such that $\mathcal M =
  \{ t_1 e_{\alpha_1}, \ldots, t_\mu e_{\alpha_\mu} \}$.
  \STATE Let $b_1, \ldots, b_\nu \in \mathbb T^n$ be terms and $\beta_1, \ldots,
  \beta_\nu \in \{ 1, \ldots, r \}$ be indices such that $\partial \mathcal M =
  \{ b_1 e_{\beta_1}, \ldots, b_\nu e_{\beta_\nu} \}$.
  \STATE $\mathcal G \assign \emptyset$
    \FOR{$j = 1, \ldots, \nu$}\label{algo:moduleBB-for}
      \STATE Determine $c_{1j}, \ldots, c_{\mu j} \in K$ such that $b_j
      e_{\beta_j} + V = \sum_{i=1}^\mu c_{ij} t_i e_{\alpha_i} +
      V$.\label{algo:moduleBB-findG}
      \STATE $\mathcal G \assign \mathcal G \cup \{ b_j e_{\beta_j} -
      \sum_{i=1}^\mu c_{ij} t_i e_{\alpha_i} \}$\label{algo:moduleBB-setG}
    \ENDFOR
  \RETURN $(\mathcal M, \mathcal G)$
\end{algorithmic}
\end{algorithm}

\begin{thm}[The Module Border Basis Algorithm]
\label{thm:moduleBBAlg}
Let $k \in \mathbb N$, let $U = \langle \mathcal B_1, \ldots, \mathcal B_k
\rangle \subseteq P^r$ with vectors $\{ \mathcal B_1, \ldots, \mathcal B_k \}
\subseteq P^r \setminus \{ 0 \}$ be a $P$-submodule such that $\codim_K(U, P^r)
< \infty$, and let~$\sigma$ be a degree compatible term ordering on~$\mathbb
T^n$. Then Algorithm~\ref{algo:moduleBB} is actually an algorithm and the result
\begin{align*}
(\mathcal M, \mathcal G) \assign {\tt moduleBB}(\{ \mathcal B_1, \ldots,
\mathcal B_k \}, \sigma)
\end{align*}
of Algorithm~\ref{algo:moduleBB} applied to the input data~$\{ \mathcal B_1,
\ldots, \mathcal B_k \}$ and~$\sigma$ satisfies the following conditions.
\begin{enumerate}
\renewcommand{\labelenumi}{\roman{enumi})}
  \item The set $\mathcal M \subseteq \mathbb T^n \langle e_1, \ldots, e_r
  \rangle$ is an order module.
  \item The set $\mathcal G \subseteq P^r$ is the $\mathcal M$-module border
  basis of~$U$.
\end{enumerate}
\end{thm}

\begin{proof}
Firstly, we prove that every step of the procedure can be computed. The maximum
in line~\ref{algo:moduleBB-initD} can be computed as obviously $k \neq 0$ in
this situation. In particular, it follows that $\dim_K(V) \geq 1$ in
line~\ref{algo:moduleBB-initV}. We can compute the intersection of $K$-vector
spaces for the computation of~$V'$ in line~\ref{algo:moduleBB-initV'} and
line~\ref{algo:moduleBB-enlargeV'} with linear algebra techniques. In
line~\ref{algo:moduleBB-OM}, the while-loop in line~\ref{algo:moduleBB-while}
has already been finished. In this situation, the construction of~$V$ in
line~\ref{algo:moduleBB-initV} and during the while-loop in
line~\ref{algo:moduleBB-while} yields $\varrho = \dim_K(V) \geq 1$ and
\begin{align*}
V = V' = (V + x_1 V + \cdots + x_n V) \cap \langle L \rangle_K,
\end{align*}
after the while-loop. In other words, the input data $d$, $\{ \mathcal V_1,
\ldots, \mathcal V_\varrho \}$, and $\sigma$ in line~\ref{algo:moduleBB-gensV}
satisfy the requirements of Algorithm~\ref{algo:computeOM}. Thus we can compute
an order module $\mathcal M \subseteq \langle L \rangle_K$ such that
$\varepsilon_V(\mathcal M)$ is a $K$-vector space basis of $\langle L \rangle_K
/ V$ and such that $\# \varepsilon_V(\mathcal M) = \# \mathcal M$ according to
Lemma~\ref{thm:moduleBBAlgLemma} in line~\ref{algo:moduleBB-OM}. Moreover, the
repeat-until-loop in line~\ref{algo:moduleBB-repeat} only stops if $\partial
\mathcal M \subseteq L$. Thus we can compute the coefficients $c_{1j}, \ldots,
c_{\mu j} \in K$ for all $j \in \{ 1, \ldots, \nu \}$ in
line~\ref{algo:moduleBB-findG} with linear algebra techniques, too. All the
other steps of the procedure can be trivially computed.\\
Secondly, we show that the procedure stops after a finite amount of time. We
start to show that the while-loop in line~\ref{algo:moduleBB-while} eventually
terminates. By the construction of~$V$ and~$V'$ in
line~\ref{algo:moduleBB-renewV} and line~\ref{algo:moduleBB-enlargeV'}, we see
that we always have $V \subseteq V' \subseteq \langle L \rangle_K$. Assume that
$V \neq V'$ in this situation, i.\,e.\ the while-loop in
line~\ref{algo:moduleBB-while} is executed at least one time. For every $i
\in \mathbb N$, we let $V'_i$ be the $K$-vector subspace $V' \subseteq \langle
L \rangle_K$ after the $i^{\text{th}}$~iteration of the while-loop. Since we
have
\begin{align*}
\# L = \# \mathbb T^n_{\leq d} \langle e_1, \ldots, e_r \rangle < \infty,
\end{align*}
i.\,e.\ $\dim_K(\langle L \rangle_K) < \infty$, the chain
\begin{align*}
V'_0 \subseteq V'_1 \subseteq V'_2 \subseteq \cdots
\end{align*}
must eventually get stationary. In this situation, we have $V'_i = V'_{i-1}$ for
some $i \in \mathbb N \setminus \{ 0 \}$ and therefore $V = V'$ in
line~\ref{algo:moduleBB-while}. Thus the while-loop terminates after the
$i^{\text{th}}$~iteration.\\
Thirdly, we prove that the repeat-until-loop in line~\ref{algo:moduleBB-repeat}
stops after a finite amount of time. Let $\mathcal H = \{ \mathcal H_1, \ldots,
\mathcal H_v \} \subseteq P^r$ with $v \in \mathbb N$ be the reduced $\sigma
\Pos$-Gröbner basis of~$U$. For all $j \in \{ 1, \ldots, v \}$, there is a
representation
\begin{align*}
\mathcal H_j = p_{j1} \mathcal B_1 + \cdots + p_{jk} \mathcal B_k
\end{align*}
with $p_{j1}, \ldots, p_{jk} \in P$. Let
\begin{align*}
d' = \max \{ \deg(t) \mid j \in \{ 1, \ldots, v \}, l \in \{ 1, \ldots, k \}, t
e_u \in \Supp(p_{j \ell} \mathcal B_\ell) \}.
\end{align*}
This maximum exists as $k > 0$ in this situation as we have seen above. Then we
have $\mathcal H \subseteq \langle L \rangle_K$ after the while-loop in the case
that $d = d'$. Now suppose that we are in the situation that $d = d'$ during the
repeat-until-loop. Let~$\mathcal M$ be the result of
Algorithm~\ref{algo:computeOM} computed in line~\ref{algo:moduleBB-OM}.
Then~$\mathcal M$ is an order module such that~$\varepsilon_V(\mathcal M)$ is a
$K$-vector space basis of~$\langle L \rangle_K / V$ and $\#
\varepsilon_V(\mathcal M) = \# \mathcal M$ by Lemma~\ref{thm:moduleBBAlgLemma}.
Let
\begin{align*}
L = \{ \ell_1 e_{u_1}, \ldots, \ell_s e_{u_s} \}
\end{align*}
with $\ell_1, \ldots, \ell_s \in \mathbb T^n$, $u_1, \ldots, u_s \in \{ 1,
\ldots, r \}$, and
\begin{align*}
\ell_1 e_{u_1} >_{\sigma \Pos} \cdots >_{\sigma \Pos} \ell_s e_{u_s}
\end{align*}
be like in line~\ref{algo:computeOM-terms} of Algorithm~\ref{algo:computeOM}
during the computation of~$\mathcal M$ by Algorithm~\ref{algo:computeOM} in
line~\ref{algo:moduleBB-OM}. Furthermore, we let $\mathfrak W \in \Mat_{m,s}(K)$
with $m \in \mathbb N$ be the matrix in row echolon form used during the
computation of $\mathcal M$ by Algorithm~\ref{algo:computeOM} in
line~\ref{algo:moduleBB-OM}. Let $j \in \{ 1, \ldots, v \}$. We write
\begin{align*}
\mathcal H_j = c_1 \ell_1 e_{u_1} + \cdots + c_s \ell_s e_{u_s}
\end{align*}
with $c_1, \ldots, c_s \in K$. Then the vector
\begin{align*}
(c_1, \ldots, c_s) \in K^s
\end{align*}
corresponds to~$\mathcal H_j$. Let
\begin{align*}
\LT_{\sigma \Pos}(\mathcal H_j) = \ell_w e_{u_w}
\end{align*}
with $w \in \{ 1, \ldots, s \}$. Then it follows that
\begin{align*}
(c_1, \ldots, c_s) = (0, \ldots, 0, 1, c_{w+1}, \ldots, c_s).
\end{align*}
by \cite[Defn.~2.4.12]{KR1}. Since $\mathcal H_j \in V$, we see that there
exists a vector in~$K^s$ corresponding to a vector in~$V$ which has its first
non-zero entry in the $w^{\text{th}}$~column of~$\mathfrak W$, namely the vector
\begin{align*}
(0, \ldots, 0, 1, c_{w+1}, \ldots, c_s) \in K^s
\end{align*}
corresponding to~$\mathcal H_j$. Thus the construction of~$\mathcal M$ in
line~\ref{algo:computeOM-OM} of Algorithm~\ref{algo:computeOM} yields that
\begin{align*}
\LT_{\sigma \Pos}(\mathcal H_j) \notin \mathcal M,
\end{align*}
and it follows that
\begin{align*}
\mathcal M \subseteq \mathcal O_{\sigma \Pos}(U).
\end{align*}
Thus
\begin{align*}
\partial \mathcal M \subseteq \oline{\partial \mathcal M} \subseteq
\oline{\partial \mathcal O_{\sigma \Pos}(U)} \subseteq L
\end{align*}
as $\mathcal H \subseteq \langle L \rangle_K$ and hence the repeat-until-loop
terminates in the case that $d = d'$.\\
Altogether, we see that the procedure is actually an algorithm.

It remains to prove the correctness. We have to distinguish two cases. For the
first case, we suppose that the algorithm stops in
line~\ref{algo:moduleBB-trivial}. Then we see that $\langle \mathcal B_1,
\ldots, \mathcal B_k \rangle = \{ 0 \}$ and the assumption $\codim_K(\{ 0 \},
P^r) < \infty$ yields $r = 0$. Thus~$\emptyset$ is obviously the $\emptyset$-module
border basis of~$\langle \mathcal B_1, \ldots, \mathcal B_k \rangle = \{ 0 \}$
by Definition~\ref{defn:moduleBB}. For the second case, we now suppose that $k
\neq 0$. As~$\mathcal M$ is computed in line~\ref{algo:moduleBB-OM} with the use
of Algorithm~\ref{algo:computeOM}, Lemma~\ref{thm:moduleBBAlgLemma} yields that
$\mathcal M$ is an order module.\\
For all $j \in \{ 1, \ldots, \nu \}$, we let
\begin{align*}
\mathcal G_j = b_j e_{\beta_j} - \sum_{i=1}^\nu c_{ij} t_i e_{\alpha_i} \in P^r
\end{align*}
with $c_{1j}, \ldots, c_{\mu j} \in K$ be like in line~\ref{algo:moduleBB-setG}.
Then $\mathcal G = \{ \mathcal G_1, \ldots, \mathcal G_\nu \}$ is an $\mathcal
M$-module border prebasis. We now show that $\LT_{\sigma \Pos}(\mathcal
G_j) = b_j e_{\beta_j}$ for all $j \in \{ 1, \ldots, \nu \}$. Let $j \in \{ 1,
\ldots, \nu \}$. Consider $\mathfrak W \in \Mat_{m,s}(K)$ and $\ell_1 e_{u_1},
\ldots, \ell_s e_{u_s} \in L$ like above, again. We write
\begin{align*}
\mathcal G_j = c_1 \ell_1 e_{u_1} + \cdots + c_s \ell_s e_{u_s}
\end{align*}
with $c_1, \ldots, c_s \in K$. Then the vector
\begin{align*}
(c_1, \ldots, c_s) \in K^s
\end{align*}
corresponding to~$\mathcal G_j$ represents a $K$-linear dependency of the terms
in~$L$. Assume that
\begin{align*}
\LT_{\sigma \Pos}(\mathcal H_j) = \ell_w e_{u_w} \neq b_j e_{\beta_j}
\end{align*}
with $w \in \{ 1, \ldots, s \}$. As~$\sigma$ is degree compatible, $(c_1,
\ldots, c_s) \in K^s$ is a vector which has its first non-zero entry in the
column corresponding to the term $\ell_w e_{u_w} \in \mathcal M$. But this is a
contradiction to the construction of~$\mathcal M$ in
line~\ref{algo:computeOM-OM} of Algorithm~\ref{algo:computeOM}, i.\,e.\ we have
\begin{align*}
\LT_{\sigma \Pos}(\mathcal G_j) = b_j e_{\beta_j}.
\end{align*}
We go on with the proof that the normal remainders of the $\SV$-vectors of all
neighbors with respect to~$\mathcal M$ vanish. Let $b_i e_{\beta_i}, b_j
e_{\beta_j} \in \partial \mathcal M$ with $i, j \in \{ 1, \ldots, \nu \}$ and $i
\neq j$ be neighbors with respect to~$\mathcal M$. We have already shown that
we have $\partial \mathcal M \subseteq L$ at the end of the algorithm, i.\,e.\
we have $\oline{\partial \mathcal M} \subseteq L$. Using the Module Border
Division Algorithm~\ref{thm:divAlg} applied to~$\SV(\mathcal G_i, \mathcal G_j)$
and~$\mathcal G$, we can compute $c_1, \ldots, c_\nu \in K$ such that
\begin{align*}
\NR_{\mathcal G}(\SV(\mathcal G_i, \mathcal G_j)) = \SV(\mathcal G_i, \mathcal
G_j) - \sum_{w=1}^\nu c_w \mathcal G_w \in \langle \mathcal M
\rangle_K.
\end{align*}
Since we have also already seen that $\mathcal G \subseteq V$, i.\,e.\ we have
$\SV(\mathcal G_i, \mathcal G_j) \in V$ as $b_i e_{\beta_i}, b_j e_{\beta_j}
\in \partial \mathcal M$ are neighbors, and $\varepsilon_V(\mathcal M)$ is
a $K$-vector space basis of~$\langle L \rangle_K / V$, it follows that
\begin{align*}
V = \SV(\mathcal G_i, \mathcal G_j) + V = \NR_{\mathcal G}(\SV(\mathcal G_i,
\mathcal G_j)) + V.
\end{align*}
In particular, we see
that
\begin{align*}
\NR_{\mathcal G}(\SV(\mathcal G_i, \mathcal G_j)) \in V \cap \langle \mathcal M
\rangle_K = \{ 0 \}.
\end{align*}
Altogether, we see that condition~$F_2)$ of Buchberger's Criterion for Module
Border Bases~\ref{thm:buchbCrit} is satisfied, i.\,e.\ $\mathcal G$ is the
$\mathcal M$-module border basis of~$\langle \mathcal G \rangle$.\\
Therefore, the claim follows if we show that~$\mathcal G$ generates~$U$. For
every $j \in \{ 1, \ldots, \nu \}$, we have already seen that $\mathcal G_j \in
V \subseteq U$, i.\,e.\ we have $\langle \mathcal G \rangle \subseteq U$. For
the converse inclusion, we let $w \in \{ 1, \ldots, k \}$. We apply the
Module Border Division Algorithm~\ref{thm:divAlg} to~$\mathcal B_w$
and~$\mathcal G$ to obtain a representation
\begin{align*}
\mathcal B_w = \mathcal V + \NR_{\mathcal G}(\mathcal B_w)
\end{align*}
with $\mathcal V \in \langle \mathcal G \rangle$ and $\NR_{\mathcal G}(\mathcal
B_w) \in \langle \mathcal M \rangle_K$. During the Module Border Division
Algorithm~\ref{thm:divAlg}, we always subtract multiples of the form~$t \mathcal
G_j$ with a term $t \in \mathbb T^n$ and $j \in \{ 1, \ldots, \nu \}$ to
eliminate the term $t b_j e_{\beta_j}$. Since~$\sigma$ is a degree compatible
term ordering on~$\mathbb T^n$ and since we have $b_j e_{\beta_j} = \LT_{\sigma
\Pos}(\mathcal G_j)$, it follows that all the vectors that are used for these
reductions satisfy $t \mathcal G_j \in \langle L \rangle_K$.
Thus we have
\begin{align*}
\mathcal V \in V \subseteq \langle L \rangle_K
\end{align*}
because $\mathcal G \subseteq V$. Altogether, we see that
\begin{align*}
V = \mathcal B_w + V = \NR_{\mathcal G}(\mathcal B_w) + V.
\end{align*}
Since $\varepsilon_V(\mathcal M)$ is a $K$-vector space basis of $\langle L
\rangle_K / V$ by Lemma~\ref{thm:moduleBBAlgLemma}, it follows that
\begin{align*}
\NR_{\mathcal G}(\mathcal B_w) \in V \cap \langle \mathcal M \rangle_K = \{ 0
\}.
\end{align*}
Thus it also follows that
\begin{align*}
\mathcal B_w = \mathcal V + \NR_{\mathcal G}(\mathcal B_w) = \mathcal V
\in \langle \mathcal G \rangle,
\end{align*}
Therefore, we see that
\begin{align*}
U = \langle \mathcal B_1, \ldots, \mathcal B_k \rangle \subseteq \langle
\mathcal G \rangle.
\end{align*}
Altogether, we have proven that~$\mathcal G$ is the $\mathcal M$-border basis
of~$U$.
\end{proof}

\begin{exmp}
\label{exmp:moduleBBAlgo}
Let $K = \mathbb Q$, $P = \mathbb Q[x,y]$, and $\sigma = \DegRevLex$.
Furthermore, we let
\begin{align*}
U & = \langle \mathcal B_1, \ldots, \mathcal B_5 \rangle\\
& = \langle (-2, 3x-1), (3x+4, 2), (0, y-1), (y-1, 0), (x+y+1, -x+y) \rangle.
\end{align*}
Since we have $x e_2, x e_1, y e_2, y e_1 \in \LT_{\sigma\Pos}(U)$, we
immediately see that
\begin{align*}
\codim_K(U, P^2) = \#\mathcal O_{\sigma\Pos}(U) \leq \# \{ e_1, e_2 \} = 2 <
\infty
\end{align*}
with Macaulay's Basis Theorem~\cite[Thm.~1.5.7]{KR1}. Thus the requirements of
the Module Basis Algorithm~\ref{algo:moduleBB} are satisfied.

We consider the steps of the Module Border Basis Algorithm~\ref{algo:moduleBB}
applied to the input data~$\{ \mathcal B_1, \ldots, \mathcal B_5 \}$
and~$\sigma$ in detail.
We initialize
\begin{align*}
V = \langle (-2, 3x-1), (3x+4, 2), (0, y-1), (y-1, 0), (x+y+1, -x+y) \rangle_K
\end{align*}
in line~\ref{algo:moduleBB-initV} and thus have $d=1$ and
\begin{align*}
L = \mathbb T^2_{\leq 1} \langle e_1, e_2 \rangle = \{ x e_1 x e_2, y e_1, y
e_2, e_1 e_2 \}
\end{align*}
in line~\ref{algo:moduleBB-initD}. Moreover, we compute that
\begin{align*}
V' = (V + xV + yV) \cap \langle L \rangle_K = V
\end{align*}
in line~\ref{algo:moduleBB-initV'}. Thus the while-loop in
line~\ref{algo:moduleBB-while} does not need to be executed and we have $\{
\mathcal V_1, \ldots, \mathcal V_5 \} = \{ \mathcal B_1, \ldots, \mathcal B_5
\}$ in line~\ref{algo:moduleBB-gensV}.

Next we consider the computation of~$\mathcal M$ in line~\ref{algo:moduleBB-OM}
with Algorithm~\ref{algo:computeOM} applied to the input data~$1$, $\{ \mathcal
B_1, \ldots, \mathcal B_5 \}$, and~$\sigma$. We order the terms in~$L$ according
to~$\sigma\Pos$ descendingly and compute the matrix
\begin{align*}
\mathfrak V = \begin{pmatrix}
0 & 3 & 0 & 0 & -2 & -1\\
3 & 0 & 0 & 0 & 4 & 2\\
0 & 0 & 0 & 1 & 0 & -1\\
0 & 0 & 1 & 0 & -1 & 0\\
1 & -1 & 1 & 1 & 1 & 0
\end{pmatrix} \in \Mat_{5,6}(\mathbb Q)
\end{align*}
like in line~\ref{algo:computeOM-initV} of Algorithm~\ref{algo:computeOM}. The
(reduced) row echolon form needed in line~\ref{algo:computeOM-rowEcho} of
Algorithm~\ref{algo:computeOM} is then
\begin{align*}
\mathfrak W = \begin{pmatrix}
1 & 0 & 0 & 0 & \tfrac{4}{3} & \tfrac{2}{3}\\
0 & 1 & 0 & 0 & -\tfrac{2}{3} & -\tfrac{1}{3}\\
0 & 0 & 1 & 0 & -1 & 0\\
0 & 0 & 0 & 1 & 0 & -1\\
0 & 0 & 0 & 0 & 0 & 0
\end{pmatrix} \in \Mat_{5,6}(\mathbb Q),
\end{align*}
i.\,e.\ we get the order module
\begin{align*}
\mathcal M = \{ e_1, e_2 \} \subseteq \mathbb T^2 \langle e_1, e_2 \rangle
\end{align*}
after line~\ref{algo:moduleBB-OM} of Algorithm~\ref{algo:moduleBB}.

As the border of~$\mathcal M$ satisfies
\begin{align*}
\partial \mathcal M = \{ x e_1, x e_2, y e_1, y e_2 \} \subseteq L,
\end{align*}
we stop the computation of the repeat-until-loop in
line~\ref{algo:moduleBB-repeat}. We then proceed with the computation of the
for-loop in line~\ref{algo:moduleBB-for} and get the set $\mathcal G = \{
\mathcal G_1, \ldots, \mathcal G_4 \} \subseteq P^2$ with
\begin{align*}
\mathcal G_1 & = x e_1 + \tfrac{4}{3} e_1 + \tfrac{2}{3} e_2,\\
\mathcal G_2 & = x e_2 - \tfrac{2}{3} e_1 - \tfrac{1}{3} e_2,\\
\mathcal G_3 & = y e_1 - e_1,\\
\mathcal G_4 & = y e_2 - e_2.
\end{align*}
According to Theorem~\ref{thm:moduleBBAlg}, $\mathcal M$ is an order module
and~$\mathcal G$ is the $\mathcal M$-module border basis of~$U$. In particular,
since the set of all corners of~$\mathcal M$ is $\{ x e_1, x e_2, y e_1, y e_2
\}$ by Definition~\ref{defn:corners}, $\mathcal G$ is also the reduced
$\sigma$-Gröbner basis of~$U$ according to Proposition~\ref{thm:corners}.
\end{exmp}

\begin{rem}\label{rem:needEmptyOrderModule}
In contrast to the theory of border bases in \cite[Section~6.4]{KR2} or
\cite{CharBB}, we explicitly had allowed that order ideals in~$\mathbb T^n$ are
empty by Definition~\ref{defn:orderIdeal}. The reason is as follows:
Let~$\sigma$ be a degree compatible term ordering on~$\mathbb T^n$, and let $k
\in \mathbb N \setminus \{ 0 \}$ and $\{ \mathcal B_1, \ldots, \mathcal B_k \}
\subseteq P^r \setminus \{ 0 \}$ be vectors such that $U = \langle \mathcal
B_1, \ldots, \mathcal B_k \rangle \subseteq P^r$ is a $P$-submodule with
$\codim_K(U, P^r) < \infty$. Moreover, assume that $r \geq 2$ and $e_1 - e_2 \in
U$, i.\,e.\ $U$ contains a $K$-linear dependency of the elements $\{ e_1,
\ldots, e_r \}$. Then the result~$\mathcal M$ in line~\ref{algo:moduleBB-OM} of
the Module Border Bases Algorithm~\ref{thm:moduleBBAlg} applied to~$\{ \mathcal
B_1, \ldots, \mathcal B_k \}$ and~$\sigma$ does not contain all the elements of
$\{ e_1, \ldots, e_r \}$, namely the above $K$-linear dependency yields $e_1
\notin \mathcal M$ as $e_1 >_{\sigma \Pos} e_2$. By allowing empty order ideals
in Definition~\ref{defn:orderIdeal}, this fact does not vanish, but the result
of the algorithm is still an order module according to
Definition~\ref{defn:orderModule}.
\end{rem}

%
%

\part{Module Border Bases of Quotient Modules}
\label{part:quotMod}

In Part~\ref{part:freeMod}, we have defined module border basis of
$P$-submodules~$U \subseteq P^r$ with finite $K$-codimension and $r \in \mathbb
N$. In Part~\ref{part:quotMod}, we extend this theory to module border bases of
$P$-submodules~$U^S \subseteq P^r / S$ of quotient modules of~$P^r$ modulo an
arbitrary given $P$-submodule $S \subseteq P^r$ with finite $K$-codimension
and~$r \in \mathbb N$.

In Section~\ref{sect:quotModuleBB}, we introduce the general concept of quotient
module border bases. We start to generalize order modules to order quotient
modules in Definition~\ref{defn:orderModule-quot}. In
Definition~\ref{defn:moduleBB-quot}, we introduce quotient module border bases
as a straightforward generalization of module border bases. Given a quotient
module border basis, we then construct characterizing module border prebases
in Definition~\ref{defn:charModuleBB}. The central result of this section,
Theorem~\ref{thm:char-quot}, then allows us to characterize quotient module
border bases with the use of these characterizing module border prebases. As a
consequence of this, we will determine an algorithm for the computation of
quotient module border bases in Corollary~\ref{thm:quotModuleBBAlg}. Moreover,
we deduce characterizations of quotient module border bases similar to the
characterizations of module border bases in Section~\ref{sect:chars}. In
particular, we characterize quotient module border bases via the special
generation property in Corollary~\ref{thm:specGen-quot}, via border form modules
in Corollary~\ref{thm:BFMod-quot}, via rewrite rules in
Corollary~\ref{thm:rewrite-quot}, via commuting matrices in
Corollary~\ref{thm:commMat-quot}, via liftings of border syzygies in
Corollary~\ref{thm:liftings-quot}, and via Buchberger's Criterion for Quotient
Module Border Bases in Corollary~\ref{thm:buchbCrit-quot}.

In Section~\ref{sect:subBB}, we apply the notion of quotient module border bases
to subideal border bases, which have been introduced in \cite{SubBB}. To this
end, we then see in Proposition~\ref{thm:subBBmoduleBB} that every subideal
border basis is isomorphic as $P$-module to a quotient module border basis. We
will use this isomorphy to deduce an algorithm for the computation of subideal
border bases in Corollary~\ref{thm:subBBmoduleBB} and to characterize subideal
border bases in Remark~\ref{rem:char-subBB}.

%
%

\section{Quotient Module Border Bases}
\label{sect:quotModuleBB}

In this section, we define module border bases of arbitrary quotient modules of
free $P$-modules of finite rank over the polynomial ring~$P$. To this end, we
always let $r \in \mathbb N$ and $S \subseteq P^r$ be a $P$-submodule. Moreover,
we let $\{ e_1, \ldots, e_r \}$ be the canonical $P$-module basis of the free
$P$-module~$P^r$. Recall, for every $P$-module~$M$ (respectively $K$-vector
space) and for every $P$-submodule $U \subseteq M$ (respectively $K$-vector
subspace), we let
\begin{align*}
\varepsilon_U: M \twoheadrightarrow M / U, \quad m \mapsto m + U
\end{align*}
be the canonical $P$-module epimorphism ($K$-vector space epimorphism).

We start with the generalization of order modules defined in
Definition~\ref{defn:orderModule} and their border in
Definition~\ref{defn:border}.

\begin{defn}
\label{defn:orderModule-quot}
Let $\mathcal M = \mathcal O_1 e_1 \cup \cdots \cup \mathcal O_r e_r$ with order
ideals $\mathcal O_1, \ldots, \mathcal O_r \subseteq \mathbb T^n$ be an order
module.
\begin{enumerate}
  \item We call the set
  \begin{align*}
  \varepsilon_S(\mathcal M) = \mathcal O_1 \cdot (e_1 + S) \cup \cdots \cup
  \mathcal O_r \cdot (e_r + S) \subseteq P^r / S
  \end{align*}
  an~\emph{order quotient module}.
  \item The set
  \begin{align*}
  \partial \varepsilon_S(\mathcal M) = \varepsilon_S(\partial \mathcal M) =
  \partial \mathcal O_1 \cdot (e_1 + S) \cup \cdots \cup \partial \mathcal O_r
  \cdot (e_r + S) \subseteq P^r / S
  \end{align*}
  is called the \emph{(first) border} of~$\varepsilon_S(\mathcal M)$.
\end{enumerate}
\end{defn}

\begin{exmp}
\label{exmp:orderModule-quot}
Let $K = \mathbb Q$, $P = \mathbb Q[x,y]$, let $\{ e_1, e_2 \}$ be the canonical
$P$-module basis of~$P^2$, and let
\begin{align*}
S = \langle x e_1 - y e_2 \rangle \subseteq P^2.
\end{align*}
Additionally, we let $\mathcal O_1 = \{ x, y, 1 \} \subseteq \mathbb T^2$ and
$\mathcal O_2 = \{ x, 1 \} \subseteq \mathbb T^n$. Then both~$\mathcal O_1$
and~$\mathcal O_2$ are order ideals by Definition~\ref{defn:orderIdeal} and
\begin{align*}
\mathcal M^S & = \mathcal O_1 \cdot (e_1 + S) \cup \mathcal O_2 \cdot (e_2 +
S)\\
& = \{ x e_1 + S, y e_1 + S, e_1 + S, x e_2 + S, e_2 + S \}
\end{align*}
is an order quotient module with border
\begin{align*}
\partial \mathcal M^S & = \partial \mathcal O_1 \cdot (e_1 + S) \cup \partial
\mathcal O_2 \cdot (e_ 2 + S)\\
& = \{ x^2, xy, y^2 \} \cdot (e_1 + S) \cup \{ x^2, xy, y \} \cdot (e_2 + S)\\
& = \{ xy e_1 + S, y^2 e_1 + S, x^2 e_2 + S, xy e_2 + S, y e_2 + S \}
\end{align*}
according to Definition~\ref{defn:orderModule-quot} as $x^2 e_1 + S = xy e_2 +
S$. In particular, we also see that
\begin{align*}
x e_1 + S = y e_2 + S \in \mathcal M^S \cap \partial \mathcal M^S.
\end{align*}
\end{exmp}

Similar to the definition of order modules in Definition~\ref{defn:orderModule},
we now define quotient module border prebases to be the images of module border
prebases under the canonical $P$-module epimosphism~$\varepsilon_S$.

\begin{defn}
\label{defn:moduleBB-quot}
Let $\mathcal M = \mathcal O_1 e_1 \cup \cdots \cup \mathcal O_r e_r$ be an
order module with finite order ideals $\mathcal O_1, \ldots, \mathcal O_r
\subseteq \mathbb T^n$. We write $\mathcal M = \{ t_1 e_{\alpha_1}, \ldots,
t_\mu e_{\alpha_\mu} \}$ and the border $\partial \mathcal M = \{ b_1
e_{\beta_1}, \ldots, b_\nu e_{\beta_\nu} \}$ with $\mu, \nu \in \mathbb N$,
$t_i, b_j \in \mathbb T^n$, and $\alpha_i, \beta_j \in \{ 1, \ldots, r \}$ for
all $i \in \{ 1, \ldots, \mu \}$ and $j \in \{ 1, \ldots, \nu \}$. Moreover, we
let $\mathcal G = \{ \mathcal G_1, \ldots, \mathcal G_\nu \} \subseteq P^r$ be
an $\mathcal M$-module border prebasis with $\mathcal G_j = b_j e_{\beta_j} -
\sum_{i=1}^\mu c_{ij} t_i e_{\alpha_i}$ where $c_{1j}, \ldots, c_{\mu j} \in K$
for all $j \in \{ 1, \ldots, \nu \}$.
\begin{enumerate}
  \item We call the set $\varepsilon_S(\mathcal G) = \{ \varepsilon_S(\mathcal
  G_1), \ldots, \varepsilon_S(\mathcal G_\nu) \} \subseteq P^r / S$ where
  \begin{align*}
  \varepsilon_S(\mathcal G_j) = \mathcal G_j + S = b_j e_{\beta_j} -
  \sum_{i=1}^\mu c_{ij} t_i e_{\alpha_i} + S
  \end{align*}
  an \emph{$\varepsilon_S(\mathcal M)$-quotient module border prebasis}.
  \item Let $U^S \subseteq P^r / S$ be a $P$-submodule. The
  $\varepsilon_S(\mathcal M)$-quotient module border prebasis
  $\varepsilon_S(\mathcal G) \subseteq P^r / S$ is called an
  \emph{$\varepsilon_S(\mathcal M)$-quotient module border basis} of~$U^S$ if
  $\varepsilon_S(\mathcal G) \subseteq U^S$, if the set
  \begin{align*}
  (\varepsilon_{U^S} \circ \varepsilon_S)(\mathcal M) = \{ (t_1 e_{\alpha_1} +
  S) + U^S, \ldots, (t_\mu e_{\alpha_\mu} + S) + U^S \}
  \end{align*}
  is a $K$-vector space basis of $(P^r / S) / U^S$, and if $\#
  (\varepsilon_{U^S} \circ \varepsilon_S)(\mathcal M) = \mu$.
\end{enumerate}
\end{defn}

\begin{rem}
\label{rem:problems-quot}
One might think that all the definitions and propositions about module border
bases in Part~\ref{part:freeMod} can be generalized to quotient module border
bases in a straightforward way. Unfortunately, the situation is more complicated
than expected. We have already seen one big difference concerning module border
bases and quotient module border bases in Example~\ref{exmp:orderModule-quot}.
Namely, we have seen that it can happen that the order quotient module and its
border have some elements in common, i.\,e.\ the straightforward, analogous
version of Proposition~\ref{thm:border} is wrong. Since most of the
propositions in Part~\ref{part:freeMod} are based on this proposition, the
theory of quotient module border bases needs much more care in the definitions
and proofs. Moreover, we see that quotient module border bases are indeed a real
extension of module border bases---meaning that the theory cannot be copied from
the usual module border bases in a straightforward way---and thus also of the
usual border bases.
\end{rem}

As we have seen in Remark~\ref{rem:problems-quot}, we cannot proof analogous
versions of every proposition in Part~\ref{part:freeMod}. Fortunately, we can
reuse some of the concepts, e.\,g.\ we can prove the uniqueness of quotient
border bases---if there exists one---similar to the proof in
Proposition~\ref{thm:existUnique}.

\begin{prop}[Uniqueness of Quotient Module Border Bases]
\label{thm:unique-quot}
Let $\mathcal M^S \subseteq P^r / S$ be a finite order quotient module, $U^S
\subseteq P^r / S$ be a $P$-submodule, and let $\mathcal G^S, \mathcal G'^S
\subseteq P^r / S$ be two $\mathcal M^S$-quotient module border bases of~$U^S$.
Then we have $\mathcal G^S = \mathcal G'^S$.
\end{prop}

\begin{proof}
We write the order quotient module $\mathcal M^S = \{ t_1 e_{\alpha_1} + S,
\ldots, t_\mu e_{\alpha_\mu} + S \}$ and its border $\partial \mathcal M^S = \{
b_1 e_{\beta_1} + S, \ldots, b_\nu e_{\beta_\nu} + S \}$ with $\mu, \nu \in
\mathbb N$, terms $t_i, b_j \in \mathbb T^n$, and $\alpha_i, \beta_j \in \{ 1,
\ldots, r \}$ for all $i \in \{ 1, \ldots, \mu \}$ and $j \in \{ 1, \ldots, \nu
\}$. Then the $\mathcal M^S$-quotient module border bases~$\mathcal G^S$
and~$\mathcal G'^S$ are of the form $\mathcal G^S = \{ \mathcal G_1^S, \ldots,
\mathcal G_\nu^S \} \subseteq P^r / S$ with $\mathcal G_j^S = b_j e_{\beta_j} -
\sum_{i=1}^\mu c_{ij} t_i e_{\alpha_i} + S$ where $c_{1j}, \ldots, c_{\mu j} \in
K$ for all $j \in \{ 1, \ldots, \nu \}$, and of the form $\mathcal G'^S = \{
\mathcal G_1'^S, \ldots, \mathcal G_\nu'^S \} \subseteq P^r / S$ with $\mathcal
G_j'^S = b_j e_{\beta_j} - \sum_{i=1}^\mu c'_{ij} t_i e_{\alpha_i} + S$ where
$c'_{1j}, \ldots, c'_{\mu j} \in K$ for all $j \in \{ 1, \ldots, \nu \}$ by
Definition~\ref{defn:moduleBB-quot}. Assume that $c_{ij} \neq c'_{ij}$ for some
$i \in \{ 1, \ldots, \mu \}$ and $j \in \{ 1, \ldots, \nu \}$. Then we have the
non-trivial $K$-linear dependency
\begin{align*}
U^S & = \mathcal G_j^S - \mathcal G_j'^S + U^S\\
& = \left( b_j e_{\beta_j} - \sum_{k=1}^\mu c_{kj} t_k e_{\alpha_k} - b_j
e_{\beta_j} + \sum_{k=1}^\mu c'_{kj} t_k e_{\alpha_k} + S \right) + U^S\\
& = \left( \sum_{k=1}^\mu (-c_{kj} + c'_{kj}) t_k e_{\alpha_k} + S \right) + U^S
\end{align*}
in contradiction to Definition~\ref{defn:moduleBB-quot} and the claim follows.
\end{proof}

For the remainder of this section, we let
\begin{align*}
\mathcal M^S = \{ t_1 e_{\alpha_1} + S, \ldots, t_\mu e_{\alpha_\mu} + S \}
\subseteq P^r / S
\end{align*}
be a finite order quotient module with border
\begin{align*}
\partial \mathcal M^S = \{ b_1 e_{\beta_1} + S, \ldots, b_\nu e_{\beta_\nu} + S
\} \subseteq P^r / S
\end{align*}
where $\mu, \nu \in \mathbb N$, $t_i, b_j \in \mathbb T^n$, and $\alpha_i,
\beta_j \in \{ 1, \ldots, r \}$ for all $i \in \{ 1, \ldots, \mu \}$ and $j \in
\{ 1, \ldots, \nu \}$. Moreover, we let $\mathcal G^S = \{ \mathcal G_1^S,
\ldots, \mathcal G_\nu^S \} \subseteq P^r / S$ with
\begin{align*}
\mathcal G_j^S = b_j e_{\beta_j} - \sum_{i=1}^\mu c_{ij} t_i e_{\alpha_i} + S,
\end{align*}
where $c_{ij} \in K$ for all $i \in \{ 1, \ldots, \mu \}$ and $j \in \{ 1,
\ldots, \nu \}$, be an $\mathcal M^S$-quotient module border prebasis .

We have already seen in Definition~\ref{defn:moduleBB-quot} that quotient module
border bases are very similar to module border bases, cf.\
Definition~\ref{defn:moduleBB}. We now want to study their similarity
in more detail. In order to do this, we regard special order modules that
correspond to~$\mathcal M^S$ in Definition~\ref{defn:charOrderModule} and
construct special module border prebasis $\mathcal G \subseteq P^r$ that
correspond to~$\mathcal G^S$ in Definition~\ref{defn:charModuleBB}. If there
exists such a corresponding module border prebasis~$\mathcal G \subseteq P^r$,
we can reduce the question whether~$\mathcal G^S$ is the $\mathcal
M^S$-quotient module border bases of~$\langle \mathcal G^S \rangle$, or not, in
Theorem~\ref{thm:char-quot} to the question whether $\mathcal G$ is the module
border basis of~$\langle \mathcal G \rangle$ and $S \subseteq \langle \mathcal
G \rangle$. This result is the core of the theory of quotient module border
bases as it implicitly allows us to reuse many results of
Part~\ref{part:freeMod}. In particular, we use this theorem to compute and
characterize quotient module border bases at the end of this section.

\begin{defn}
\label{defn:charOrderModule}
We call an order module $\mathcal M \subseteq \mathbb T^n \langle e_1, \ldots,
e_r \rangle$ an \emph{order module characterizing~$\mathcal M^S$} if
$\varepsilon_S(\mathcal M) = \mathcal M^S$ and if the
restriction~$\restrict{\varepsilon_S}{\mathcal M}$ of~$\varepsilon_S$
to~$\mathcal M$ is injective.
\end{defn}

\begin{defn}
\label{defn:charModuleBB}
Let $\mathcal M \subseteq \mathbb T^n \langle e_1, \ldots, e_r \rangle$ be an
order module characterizing~$\mathcal M^S$. By choosing suitable representatives
of the elements in $\mathcal M^S \subseteq P^r / S$, we can w.\,l.\,o.\,g.\
assume that
\begin{align*}
\mathcal M = \{ t_1 e_{\alpha_1}, \ldots, t_\mu e_{\alpha_\mu} \} \subseteq
\mathbb T^n \langle e_1, \ldots, e_r \rangle
\end{align*}
as $\varepsilon_S(\mathcal M) = \mathcal M^S$ and
as~$\restrict{\varepsilon_S}{\mathcal M}$ is injective by
Definition~\ref{defn:charOrderModule}. The border of~$\mathcal M$ then has the
form
\begin{align*}
\partial \mathcal M = \{ b_1 e_{\beta_1}, \ldots, b_\nu e_{\beta_\nu}, b_{\nu +
1} e_{\beta_{\nu + 1}}, \ldots, b_\omega e_{\beta_\omega} \} \subseteq \mathbb
T^n \langle e_1, \ldots, e_r \rangle
\end{align*}
with an $\omega \in \mathbb N$ such that $\omega \geq \nu$, $b_{\nu + 1},
\ldots, b_{\omega} \in \mathbb T^n$, and $\beta_{\nu + 1}, \ldots, \beta_\omega
\in \{ 1, \ldots, r \}$. For all $j \in \{ 1, \ldots, \nu \}$, we define
\begin{align*}
\mathcal G_j = b_j e_{\beta_j} - \sum_{i=1}^\mu c_{ij} t_i e_{\alpha_i} \in P^r.
\end{align*}
For all $j \in \{ \nu+1, \ldots, \omega \}$, there exists a unique index $k \in
\{ 1, \ldots, \nu \}$ such that
\begin{align*}
b_j e_{\beta_j} + S = \varepsilon_S(b_j e_{\beta_j}) = \varepsilon_S(b_k
e_{\beta_k}) = b_k e_{\beta_k} + S
\end{align*}
according to Definition~\ref{defn:orderModule-quot} and we define
\begin{align*}
\mathcal G_j = b_j e_{\beta_j} - \sum_{i=1}^\nu c_{ik} t_i e_{\alpha_i} \in P^r.
\end{align*}
We call the set $\mathcal G = \{ \mathcal G_1, \ldots, \mathcal G_\omega \}
\subseteq P^r$ the \emph{$\mathcal M$-module border prebasis
characterizing~$\mathcal G^S$}.
\end{defn}

\begin{exmp}
\label{exmp:charModuleBB}
We consider Example~\ref{exmp:orderModule-quot}, again. Recall, that
\begin{align*}
S = \langle x e_1 - y e_2 \rangle \subseteq P^2,
\end{align*}
that the order quotient module has the form
\begin{align*}
\mathcal M^S & = \{ x, y, 1 \} \cdot (e_1 + S) \cup \{ x, 1 \} \cdot (e_2 + S)\\
& = \{ x e_1 + S, y e_1 + S, e_1 + S, x e_2 + S, e_2 + S \},
\end{align*}
and that its border is
\begin{align*}
\partial \mathcal M^S & = \{ x^2, xy, y^2 \} \cdot (e_1 + S) \cup \{ x^2, xy, y
\} \cdot (e_2 + S)\\
& = \{ xy e_1 + S, y^2 e_1 + S, x^2 e_2 + S, xy e_2 + S, y e_2 + S \}.
\end{align*}
Then the set $\mathcal G^S = \{ \mathcal G_1^S, \ldots, \mathcal G_5^S \}
\subseteq P^2 / S$ with
\begin{align*}
\mathcal G_1^S & = xy e_1 - y e_1 + S\\
\mathcal G_2^S & = y^2 e_1 - e_1 - e_2 + S\\
\mathcal G_3^S & = x^2 e_2 + e_1 + S\\
\mathcal G_4^S & = xy e_2 - x e_1 + e_2 + S\\
\mathcal G_5^S & = y e_2 - e_2 + S
\end{align*}
is an $\mathcal M^S$-quotient module border prebasis according to
Definition~\ref{defn:moduleBB-quot}.\\
Moreover, we immediately see that the set
\begin{align*}
\mathcal M & = \{ x, y, 1 \} \cdot e_1 \cup \{ x, 1 \} \cdot e_2\\
& = \{ x e_1, y e_1, e_1, x e_2, e_2 \} \subseteq \mathbb T^n \langle e_1, e_2
\rangle
\end{align*}
is an order module characterizing~$\mathcal M^S$ by
Definition~\ref{defn:charOrderModule}. The border of~$\mathcal M$ is
\begin{align*}
\partial \mathcal M & = \{ x^2, xy, y^2 \} \cdot e_1 \cup \{ x^2, xy, y \}
\cdot e_2\\
& = \{ xy e_1, y^2 e_1, x^2 e_2, xy e_2, y e_2, x^2 e_1 \} \subseteq \mathbb
T^2 \langle e_1, e_2 \rangle.
\end{align*}
Furthermore, the set $\mathcal G = \{ \mathcal G_1, \ldots, \mathcal G_6 \}
\subseteq P^2$ with
\begin{align*}
\mathcal G_1 & = xy e_1 - y e_1\\
\mathcal G_2 & = y^2 e_1 - e_1 - e_2\\
\mathcal G_3 & = x^2 e_2 + e_1\\
\mathcal G_4 & = xy e_2 - x e_1 + e_2\\
\mathcal G_5 & = y e_2 - e_2\\
\mathcal G_6 & = x^2 e_1 - x e_1 + e_2
\end{align*}
is the $\mathcal M$-module border prebasis characterizing~$\mathcal G^S$
according to Definition~\ref{defn:charModuleBB} as
\begin{eqnarray*}
x^2 e_1 + S = xy e_2 + S.
\end{eqnarray*}
Note that
\begin{align*}
\# \mathcal G^S = 5 < 6 =  \# \mathcal G
\end{align*}
and that the construction yields
\begin{align*}
\varepsilon_S(\mathcal G_6) = x^2 e_1 - x e_1 + e_2 + S = xy e_2 - x e_1 + e_2 +
S = \varepsilon_S(\mathcal G_4) = \mathcal G_4^S.
\end{align*}
\end{exmp}

After the following auxiliary lemma, we are ready to prove our main result,
namely the characterization of quotient module border bases via characterizing
module border bases.

\begin{lem}
\label{thm:charModuleBB}
Let $\mathcal B_1, \ldots, \mathcal B_k \in P^r$ and let $\mathcal B_\ell^S =
\varepsilon_S(\mathcal B_\ell) = \mathcal B_\ell + S \subseteq P^r / S$ be for
every index $\ell \in \{ 1, \ldots, k \}$. Furthermore, we let $U = \langle
\mathcal B_1, \ldots, \mathcal B_k \rangle \subseteq P^r$ and $U^S = \langle
\mathcal B_1^S, \ldots, \mathcal B_k^S \rangle \subseteq P^r / S$.\\
In addition, suppose there exists an order module $\mathcal M \subseteq \mathbb
T^n \langle e_1, \ldots, e_r \rangle$ characterizing~$\mathcal M^S$ and let
$\mathcal G \subseteq P^r$ be the $\mathcal M$-module border prebasis
characterizing~$\mathcal G^S$.
\begin{enumerate}
  \item We have $\varepsilon_S(U) = U^S$. In particular, $\varepsilon_S(\langle
  \mathcal G \rangle) = \langle \mathcal G^S \rangle$.
  \item We have $\varepsilon_S^{-1}(U^S) = U + S$. In particular,
  $\varepsilon_S^{-1}(\langle \mathcal G^S \rangle) = \langle \mathcal G \rangle
  + S$.
  \item The set $\mathcal G$ is an $\mathcal M$-module border prebasis.
\end{enumerate}
\end{lem}

\begin{proof}
The definitions of~$U$ and~$U^S$ immediately yield the first claim of~a). Thus
we go on with the proof of the first claim of~b). For the first inclusion, we
let
\begin{align*}
\mathcal V = p_1 \mathcal B_1 + \cdots + p_k \mathcal B_k + \mathcal S \in U + S
\end{align*}
with polynomials $p_1, \ldots, p_k \in P$ and $\mathcal S \in S$. Then we see
that
\begin{align*}
\varepsilon_S(\mathcal V) & = p_1 \varepsilon_S(\mathcal B_1) + \cdots + p_k
\varepsilon_S(\mathcal B_k) + \varepsilon_S(\mathcal S)\\
& = p_1 \mathcal B_1^S + \cdots + p_k \mathcal B_k^S \in U^S.
\end{align*}
For the converse inclusion, we let $\mathcal W \in \varepsilon_S^{-1}(U^S)
\subseteq P^r$. Then there exist polynomials $q_1, \ldots, q_k \in P$ such that
\begin{align*}
\varepsilon_S(\mathcal W) = q_1 \mathcal B_1^S + \cdots + q_k \mathcal B_k^S =
\varepsilon_S(q_1 \mathcal B_1 + \cdots + q_k \mathcal B_k).
\end{align*}
Thus we have
\begin{align*}
\mathcal W - (q_1 \mathcal B_1 + \cdots + q_k \mathcal B_k) \in
\ker(\varepsilon_S) = S
\end{align*}
and this yields~$\mathcal W \in U + S$.

Let $\mathcal G = \{ \mathcal G_1, \ldots, \mathcal G_\omega \} \subseteq P^r$
with $\omega \geq \nu$ be written like in Definition~\ref{defn:charModuleBB}.
Then the second claims of~a) and~b) follow by the construction of~$\mathcal G$
according to Definition~\ref{defn:charModuleBB} and the above proofs as for all
$j \in \{ \nu + 1, \ldots, \omega \}$, there exists a $k \in \{ 1, \ldots, \nu
\}$ such that $\varepsilon_S(\mathcal G_j) = \mathcal G_k^S$.

Finally, the construction of~$\mathcal M$ in
Definition~\ref{defn:charOrderModule} and~$\mathcal G$ in
Definition~\ref{defn:charModuleBB} yield that~$\mathcal M$ is an order module
and that~$\mathcal G$ is an $\mathcal M$-module border prebasis.
\end{proof}

We now use Lemma~\ref{thm:charModuleBB} to proof the main result of this
section, namely the characterization of quotient module border bases via
characterizing module border bases.

\begin{thm}[Characterization of Quotient Module Border Bases]
\label{thm:char-quot}
Suppose that there exists an order module $\mathcal M \subseteq \mathbb T^n
\langle e_1, \ldots, e_r \rangle$ characterizing~$\mathcal M^S$ and let
$\mathcal G \subseteq P^r$ be the $\mathcal M$-module border prebasis
characterizing~$\mathcal G^S \subseteq P^r / S$. Then the following conditions
are equivalent.
\begin{enumerate}
\renewcommand{\labelenumi}{\roman{enumi})}
  \item The $\mathcal M^S$-quotient module border prebasis~$\mathcal G^S$ is the
  $\mathcal M^S$-quotient module border basis of~$\langle \mathcal G^S \rangle$.
  \item The $\mathcal M$-module border prebasis~$\mathcal G$ is the $\mathcal
  M$-module border basis of~$\langle \mathcal G \rangle$ and we have~$S
  \subseteq \langle \mathcal G \rangle$.
\end{enumerate}
\end{thm}

\begin{proof}
Let $\mathcal M \subseteq \mathbb T^n \langle e_1, \ldots, e_r \rangle$ and
$\mathcal G = \{ \mathcal G_1, \ldots, \mathcal G_\omega \} \subseteq P^r$ with
$\omega \geq \nu$ be written like in Definition~\ref{defn:charModuleBB}. In
particular, i.\,e.\ $\mathcal M = \{ t_1 e_{\alpha_1}, \ldots, t_\mu
e_{\alpha_\mu} \}$.

Firstly, we prove that $S \subseteq \langle \mathcal G \rangle$ if~$\mathcal
G^S$ is the $\mathcal M^S$-quotient module border basis of~$\langle \mathcal G^S
\rangle$. Suppose that~$\mathcal G^S$ is the $\mathcal M^S$-quotient module
border basis of~$\langle \mathcal G^S \rangle$. Assume that $S \not\subseteq
\langle \mathcal G \rangle$. Let $\mathcal S \in S \setminus \langle \mathcal G
\rangle$. We apply the Module Border Division Algorithm~\ref{thm:divAlg}
to~$\mathcal S$ and~$\mathcal G$ to obtain a representation
\begin{align*}
\mathcal S = p_1 \mathcal G_1 + \cdots + p_\omega \mathcal G_\omega + c_1 t_1
e_{\alpha_1} + \cdots + c_\mu t_\mu e_{\alpha_\mu}
\end{align*}
with $p_1, \ldots, p_\omega \in P$ and $c_1, \ldots, c_\mu \in K$. Since
$\mathcal S \notin \langle \mathcal G \rangle$, we see that
\begin{align*}
c_1 t_1 e_{\alpha_1} + \cdots + c_\mu t_\mu e_{\alpha_\mu} \neq 0
\end{align*}
and thus there exists an index $i \in \{ 1, \ldots, \mu \}$ such that $c_i \neq
0$. Moreover, Lemma~\ref{thm:charModuleBB} yields
\begin{align*}
\langle \mathcal G^S \rangle & = \varepsilon_S(\mathcal S) + \langle \mathcal
G^S \rangle\\
& = \varepsilon_S(p_1 \mathcal G_1 + \cdots + p_\omega \mathcal G_\omega +
c_1 t_1 e_{\alpha_1} + \cdots + c_\mu t_\mu e_{\alpha_\mu}) + \langle \mathcal
G^S \rangle\\
& = (p_1 \varepsilon_S(\mathcal G_1) + \cdots + p_\omega \varepsilon_S(\mathcal
G_\omega) + (c_1 t_1 e_{\alpha_1} + \cdots + c_\mu t_\mu e_{\alpha_\mu}) + S) +
\langle \mathcal G^S \rangle\\
& = (c_1 t_1 e_{\alpha_1} + \cdots + c_\mu t_\mu e_{\alpha_\mu} + S) + \langle
\mathcal G^S \rangle.
\end{align*}
As $\varepsilon_S(\mathcal M) = \mathcal M^S$
and~$\restrict{\varepsilon_S}{\mathcal M}$ is injective according to
Definition~\ref{defn:charOrderModule} and as $c_i \neq 0$, it follows that
$\varepsilon_{\langle \mathcal G^S \rangle}(\mathcal M^S) \subseteq (P^r / S) /
\langle \mathcal G^S \rangle$ is $K$-linearly dependent in contradiction to
Definition~\ref{defn:moduleBB-quot}. Altogether, we have proven that $S
\subseteq \langle \mathcal G \rangle$.

Secondly, we prove the claimed equivalence. Suppose that $S \subseteq \langle
\mathcal G \rangle$. Then Lemma~\ref{thm:charModuleBB} yields
$\varepsilon_S^{-1}(\langle \mathcal G^S \rangle) = \langle \mathcal G \rangle$
and hence~$\varepsilon_S$ induces a $P$-module isomorphism
\begin{align*}
P^r / \langle \mathcal G \rangle = P^r / \varepsilon_S^{-1}(\langle \mathcal
G^S \rangle) \cong \varepsilon_S(P^r) / \langle \mathcal G^S \rangle = (P^r / S)
/ \langle \mathcal G^S \rangle
\end{align*}
according to the Second Noether Isomorphism Theorem. As $\varepsilon_S(\mathcal
M) = \mathcal M^S$ and as~$\restrict{\varepsilon_S}{\mathcal M}$ is injective by
Definition~\ref{defn:charOrderModule}, it follows that~$\mathcal G$ is the
$\mathcal M$-module border basis of~$\langle \mathcal G \rangle$ if and only
if~$\mathcal G^S$ is the $\mathcal M^S$-quotient module border basis
of~$\langle \mathcal G^S \rangle$ by Definition~\ref{defn:moduleBB} and
Definition~\ref{defn:moduleBB-quot}.
\end{proof}

\begin{exmp}
\label{exmp:char-quot}
Let $K = \mathbb Q$ and $P = \mathbb Q[x,y]$. Furthermore, we let
\begin{align*}
S = \langle (x+y+1, -x+y) \rangle \subseteq P^2
\end{align*}
be a $P$-submodule,
\begin{align*}
\varepsilon_S: P^2 \twoheadrightarrow P^2 / S, \quad \mathcal V \mapsto \mathcal
V + S,
\end{align*}
be the canonical $P$-module epimorphism, and $\{ e_1, e_2 \}$ be the
canonical $P$-module basis of~$P^2$. Consider the order module
\begin{align*}
\mathcal M = \{ e_1, e_2 \} \subseteq \mathbb T^2 \langle e_1, e_2 \rangle,
\end{align*}
and the $\mathcal M$-module border prebasis $\mathcal G = \{ \mathcal G_1,
\ldots, \mathcal G_4 \} \subseteq P^2$ with
\begin{align*}
\mathcal G_1 & = x e_1 + \tfrac{4}{3} e_1 + \tfrac{2}{3} e_2,\\
\mathcal G_2 & = x e_2 - \tfrac{2}{3} e_1 - \tfrac{1}{3} e_2,\\
\mathcal G_3 & = y e_1 - e_1,\\
\mathcal G_4 & = y e_2 - e_2.
\end{align*}
In Example~\ref{exmp:moduleBBAlgo}, we have shown that~$\mathcal G$ is the
$\mathcal M$-module border basis of
\begin{align*}
U = \langle (x+y+1, -x+y), (-2, 3x-1), (3x+4, 2), (0, y-1), (y-1, 0) \rangle
\subseteq P^2.
\end{align*}
In particular, we hence have $U = \langle \mathcal G \rangle$ by
Corollary~\ref{thm:genSetModBB}. Let
\begin{align*}
\mathcal M^S = \varepsilon_S(\mathcal M) = \{ e_1 + S, e_2 + S \}
\end{align*}
and $\mathcal G^S = \{ \mathcal G_1^S, \ldots, \mathcal G_4^S \} =
\varepsilon_S(\mathcal G) \subseteq P^2 / S$ with
\begin{align*}
\mathcal G_1^S & = \varepsilon_S(\mathcal G_1) = x e_1 + \tfrac{4}{3} e_1 +
\tfrac{2}{3} e_2 + S,\\
\mathcal G_2^S & = \varepsilon_S(\mathcal G_2) = x e_2 - \tfrac{2}{3} e_1 -
\tfrac{1}{3} e_2 + S,\\
\mathcal G_3^S & = \varepsilon_S(\mathcal G_3) = y e_1 - e_1 + S,\\
\mathcal G_4^S & = \varepsilon_S(\mathcal G_4) = y e_2 - e_2 + S.
\end{align*}
Then we see that~$\mathcal M$ is an order module characterizing the
order quotient module~$\mathcal M^S$ by Definition~\ref{defn:charOrderModule}
and~$\mathcal G$ is the $\mathcal M$-module border prebasis characterizing
the~$\mathcal M$-quotient module border prebasis~$\mathcal G^S$ by
Definition~\ref{defn:charModuleBB}. As we also have
\begin{align*}
S = \langle (x+y+1, -x+y) \rangle \subseteq \langle \mathcal G \rangle,
\end{align*}
and as~$\mathcal G$ is the $\mathcal M$-module border basis of~$\langle
\mathcal G \rangle$, Theorem~\ref{thm:char-quot} yields that~$\mathcal G^S$ is
the $\mathcal M^S$-quotient module border basis of
\begin{align*}
\langle \mathcal G^S \rangle & = \langle \varepsilon_S(\mathcal G) \rangle\\
& = \langle (-2, 3x-1) + S, (3x+4, 2) + S, (0, y-1) + S, (y-1, 0) + S \rangle
\subseteq P^2 / S.
\end{align*}
\end{exmp}

\begin{rem}
\label{rem:charOrderModule}
The assumption in Theorem~\ref{thm:char-quot} that an order module~$\mathcal
M$ characterizing~$\mathcal M^S$ exists, is crucial as the following example
shows.

Let $P = \mathbb Q[x,y]$, $\{ e_1, e_2 \}$ be the canonical $P$-module basis
of~$P^2$, and
\begin{align*}
S = \langle x e_1 - y e_2 \rangle \subseteq P^2.
\end{align*}
We consider the order quotient module
\begin{align*}
\mathcal M^S & = \{ x^2, x, 1\} \cdot (e_1 + S) \cup \{ y^2, y, 1 \} \cdot (e_2
+ S)\\
& = \{ x^2 e_1 + S, x e_1 + S, e_1 + S, y^2 e_2 + S, y e_2 + S, e_2 + S \}\\
& = \{ e_1 + S, y^2 e_2 + S, y e_2 + S, e_2 + S, xy e_2 + S \}\\
& = \{ 1 \} \cdot (e_1 + S) \cup \{ xy, y^2, y, 1 \} \cdot (e_2 + S) \subseteq
P^2 / S.
\end{align*}
We see that $\# \mathcal M^S = 5$ as
\begin{align*}
x e_1 + S = y e_2 + S \in \mathcal M^S,
\end{align*}
and that
\begin{align*}
x^2 e_1 + S = xy e_2 + S \in \mathcal M^S.
\end{align*}
Assume that there exists an order module $\mathcal M \subseteq \mathbb T^2
\langle e_1, e_2 \rangle$ characterizing~$\mathcal M^S$. Then we have $\#
\mathcal M^S = 5$, $y^2 e_2 \in \mathcal M$, and $x^2 e_1 \in \mathcal M$ or
$xy e_2 \in \mathcal M$. If $x^2 e_1 \in \mathcal M$,
Definition~\ref{defn:charOrderModule} and Definition~\ref{defn:orderModule}
yield the contradiction
\begin{align*}
\# \mathcal M \geq \#(\{ x^2, x, 1 \} \cdot e_1 \cup \{ y^2, y, 1 \} \cdot e_2)
= 6 > 5 = \# \mathcal M^S.
\end{align*}
If $xy e_2 \in \mathcal M$, Definition~\ref{defn:orderModule} yields that $x
e_2 \in \mathcal M$ and hence we get the contradiction $\varepsilon_S(x e_2) = x
e_2 + S \in \mathcal M^S$ by Definition~\ref{defn:charOrderModule}.\\
Thus there exists no order module characterizing~$\mathcal M^S$ by
Definition~\ref{defn:charOrderModule}. In particular, we see that we cannot use
Theorem~\ref{thm:char-quot} for every $\mathcal M^S$-quotient module border
basis.
\end{rem}

\begin{rem}
\label{rem:indexDivAlg-quot}
Although we have seen in Remark~\ref{rem:problems-quot} that we cannot reuse the
the results of module border bases in Part~\ref{part:freeMod} in a
straightforward way, we can identify quotient module border bases with their
characterizing module border bases---if they exist---according to
Theorem~\ref{thm:char-quot}. This allows us to define similar concept like the
Module Border Division Algorithm~\ref{thm:divAlg}---and thus the normal
remainder---for quotient module border bases the following way.

Suppose that there exists an order module $\mathcal M \subseteq \mathbb T^n
\langle e_1, \ldots, e_r \rangle$ characterizing~$\mathcal M^S$. Firstly, we
have to compute an order module~$\mathcal M$ characterizing~$\mathcal M^S$ like
in Definition~\ref{defn:charOrderModule} and the $\mathcal M$-module border
prebasis~$\mathcal G$ characterizing~$\mathcal G^S$ like in
Definition~\ref{defn:charModuleBB}. Secondly, we have to determine one
representative $\mathcal V \in P^r$ of a given element $\mathcal V^S \in P^r /
S$ and apply the Module Border Division Algorithm~\ref{thm:divAlg} to~$\mathcal
V$ and~$\mathcal G$ to obtain a representation
\begin{align*}
\mathcal V = p_1 \mathcal G_1 + \cdots + p_\omega \mathcal G_\omega + c_1 t_1
e_{\alpha_1} + \cdots + c_\mu t_\mu e_{\alpha_\mu}
\end{align*}
with $p_1, \ldots, p_\omega \in P$ and $c_1, \ldots, c_\mu \in K$. Finally, we
have to apply~$\varepsilon_S$ to this result and get a similar representation of
$\mathcal V^S \in P^r / S$, namely
\begin{align*}
\mathcal V^S & = \varepsilon_S(\mathcal V)\\
& = p_1 \varepsilon_S(\mathcal G_1) + \cdots + p_\omega \varepsilon_S(\mathcal
G_\omega) + (c_1 t_1 e_{\alpha_1} + \cdots + c_\mu t_\mu e_{\alpha_\mu}) + S\\
& = q_1 \mathcal G_1^S + \cdots + q_\nu \mathcal G_\nu^S + (c_1 t_1 e_{\alpha_1}
+ \cdots + c_\mu t_\mu e_{\alpha_\mu}) + S \in P^r / S
\end{align*}
with polynomials $q_1, \ldots, q_\nu \in P$, cf.\
Lemma~\ref{thm:charModuleBB} and Definition~\ref{defn:charModuleBB}. We can then
define the element
\begin{align*}
c_1 t_1 e_{\alpha_1} + \cdots + c_\mu t_\mu e_{\alpha_\mu} + S \in P^r / S,
\end{align*}
which is a representative of the residue class $\mathcal V^S + \langle \mathcal
G^S \rangle \in (P^r / S) / \langle \mathcal G^S \rangle$, to be the normal
remainder of $\mathcal V^S \in P^r / S$ with respect to~$\mathcal G^S$
and~$\mathcal V$. In particular, we are then able to generalize
Corollary~\ref{thm:genSetModBB} and Corollary~\ref{thm:genSet} to quotient
module border bases using this construction.

Many other concepts can be defined for quotient module border bases the same
way, e.\,g.\ an $\mathcal M^S$-index like in Definition~\ref{defn:index} or the
special generation property in Theorem~\ref{thm:specGen}. But note that the
result of the last step, namely applying~$\varepsilon_S$ to the result in~$P^r$,
can lead to inconsistencies if we do not distinguish between different
representatives of the given residue class, as the following example shows.\\
We consider Example~\ref{exmp:charModuleBB}, again. Recall, that we have had
\begin{align*}
x e_1 + S = y e_2 + S \in \mathcal M^S \cap \partial \mathcal M^S.
\end{align*}
Then the above construction assigns $x e_1 + S$ the $\mathcal M^S$-index
$\ind_{\mathcal M}(x e_1) = 0$, whereas the same residue class $y e_2 + S$ is
also assigned the $\mathcal M^S$-index $\ind_{\mathcal M}(y e_2) = 1$.

Altogether, we see that we can reuse the concepts of Part~\ref{part:freeMod} but
we sometimes must not define these concepts for a given residue class in~$P^r/S$
itself, but only for a specific representative of it in~$P^r$.
\end{rem}

We now use Theorem~\ref{thm:char-quot} to determine an algorithm that computes a
quotient module border basis of an arbitrary $P$-submodule $U^S \subseteq P^r
/ S$ with finite $K$-codimension in~$P^r / S$. The idea of this algorithm is,
that we first compute an $\mathcal M$-module border basis~$\mathcal G \subseteq
P^r$ of the $P$-submodule $\varepsilon_S^{-1}(U^S) \subseteq P^r$ and then
apply~$\varepsilon_S$ to the result to derive the $\varepsilon_S(\mathcal
M)$-quotient module border prebasis $\varepsilon_S(\mathcal G)$, again. Then we
see that~$\mathcal M$ is an order module characterizing~$\varepsilon_S(\mathcal
M)$ and that~$\mathcal G$ is the $\varepsilon_S(\mathcal M)$-module border
prebasis characterizing~$\varepsilon_S(\mathcal G)$, and hence the claim follows
with Theorem~\ref{thm:char-quot}.

\begin{algorithm}[H]
\caption{${\tt quotModuleBB}(\{ \mathcal B_1, \ldots, \mathcal B_k \}, \{
\mathcal S_1, \ldots, \mathcal S_\ell \}, \sigma)$} \begin{algorithmic}[1]
\label{algo:quotModuleBB}
  \REQUIRE $k \in \mathbb N$ and $\{ \mathcal B_1, \ldots, \mathcal B_k \}
  \subseteq P^r \setminus \{ 0 \}$,\\
  $\ell \in \mathbb N$ and $\{ \mathcal S_1, \ldots, \mathcal S_\ell \}
  \subseteq P^r \setminus \{ 0 \}$,\\
  $\codim_K(\langle \mathcal B_1, \ldots, \mathcal B_k \rangle + \langle
  \mathcal S_1, \ldots, \mathcal S_\ell \rangle, P^r) <
  \infty$,\\
  $\sigma$ is a degree compatible term ordering on~$\mathbb T^n$
  \STATE $S \assign \langle \mathcal S_1, \ldots, \mathcal S_\ell
  \rangle$\label{algo:quotModuleBB-initS}
  \STATE Let $\varepsilon_S: P^r \twoheadrightarrow P^r /S, ~e_k \mapsto e_k +
  S$.\label{algo:quotModuleBB-initEpsilon}
  \STATE $(\mathcal M, \mathcal G) \assign {\tt moduleBB}(\{ \mathcal B_1,
  \ldots, \mathcal B_k, \mathcal S_1, \ldots, \mathcal S_\ell \},
  \sigma)$\label{algo:quotModuleBB-corrModuleBB}
  \STATE $\mathcal M^S \assign \varepsilon_S(\mathcal
  M)$\label{algo:quotModuleBB-orderQuotModule}
  \STATE $\mathcal G^S \assign \varepsilon_S(\mathcal
  G)$\label{algo:quotModuleBB-quotModuleBB}
  \RETURN $(\mathcal M^S, \mathcal G^S)$
\end{algorithmic}
\end{algorithm}

\begin{cor}[The Quotient Module Border Basis Algorithm]
\label{thm:quotModuleBBAlg}
Let $\ell \in \mathbb N$ and $S = \langle \mathcal S_1, \ldots, \mathcal
S_\ell \rangle \subseteq P^r$ with $\{ \mathcal S_1, \ldots, \mathcal S_\ell \}
\subseteq P^r \setminus \{ 0 \}$ be a $P$-submodule. Let $k \in \mathbb N$ and
$U^S = \langle \mathcal B_1 + S, \ldots, \mathcal B_k + S \rangle \subseteq P^r
/ S$ with $\{ \mathcal B_1, \ldots, \mathcal B_k \} \subseteq P^r
\setminus \{ 0 \}$ be a $P$-submodule such that $\codim_K(U^S, P^r / S) <
\infty$. Moreover, let $\sigma$ be a degree compatible term ordering on~$\mathbb
T^n$. Then Algorithm~\ref{algo:quotModuleBB} is actually an algorithm and the
result
\begin{align*}
(\mathcal M^S, \mathcal G^S) \assign {\tt quotModuleBB}(\{ \mathcal B_1, \ldots,
\mathcal B_k \}, \{ \mathcal S_1, \ldots, \mathcal S_\ell \}, \sigma)
\end{align*}
of Algorithm~\ref{algo:quotModuleBB} applied to the input data~$\{ \mathcal B_1,
\ldots, \mathcal B_k \}$, $\{ \mathcal S_1, \ldots, \mathcal S_\ell \}$,
and~$\sigma$ satisfies the following conditions.
\begin{enumerate}
\renewcommand{\labelenumi}{\roman{enumi})}
  \item The set $\mathcal M^S \subseteq P^r / S$ is an order quotient module.
  \item The set $\mathcal G^S \subseteq P^r / S$ is the $\mathcal M^S$-quotient
  module border basis of~$U^S$.
\end{enumerate}
\end{cor}

\begin{proof}
Let $U = \langle \mathcal B_1, \ldots, \mathcal B_k \rangle + S \subseteq
P^r$.

We start to prove that the procedure is actually an algorithm. First, we check
that the requirements of the algorithm are satisfied, i.\,e.\ if $\codim_K(U,
P^r) < \infty$ holds. Since Lemma~\ref{thm:charModuleBB} yields that
\begin{align*}
\varepsilon_S^{-1}(U^S) = \langle \mathcal B_1, \ldots, \mathcal B_k \rangle + S
= U,
\end{align*}
we have
\begin{align*}
P^r / U = P^r / \varepsilon_S^{-1}(U^S) \cong \varepsilon_S(P^r) / U^S = (P^r /
S) / U^S
\end{align*}
by the Second Noether Isomorphism Theorem. In particular, we see that
\begin{align*}
\codim_K(U, P^r) = \codim_K(U^S, P^r / S) < \infty,
\end{align*}
and the requirements of the algorithm are satisfied. Additionally, since we
have $U = \langle \mathcal B_1, \ldots, \mathcal B_k, \mathcal S_1, \ldots,
\mathcal S_\ell \rangle$, it follows that
line~\ref{algo:quotModuleBB-corrModuleBB} can be computed in a finite amount of
time by Theorem~\ref{thm:moduleBBAlg}. As all the other steps can obviously be
computed in a finite amount of time, we see that the procedure is actually an
algorithm.

Next we show the correctness. According to the Module Border Basis
Algorithm~\ref{thm:moduleBBAlg}, we compute an order module $\mathcal M
\subseteq \mathbb T^n \langle e_1, \ldots, e_r \rangle$ such that $\mathcal G
\subseteq P^r$ is the $\mathcal M$-module border basis of~$U$ in
line~\ref{algo:quotModuleBB-corrModuleBB}. In particular, $\mathcal M^S =
\varepsilon_S(\mathcal M)$, which is computed in
line~\ref{algo:quotModuleBB-orderQuotModule}, is an order quotient module by
Definition~\ref{defn:orderModule-quot} and $\mathcal G^S =
\varepsilon_S(\mathcal G)$, which is computed in
line~\ref{algo:quotModuleBB-quotModuleBB}, is an $\mathcal M^S$-quotient module
border prebasis by Definition~\ref{defn:moduleBB-quot}.\\
We now prove that~$\mathcal M$ is an order module characterizing~$\mathcal
M^S$. As~$U$ is generated by~$\mathcal G$ according to
Corollary~\ref{thm:genSetModBB}, Lemma~\ref{thm:charModuleBB} shows that
\begin{align*}
U^S = \varepsilon_S(U) = \varepsilon_S(\langle \mathcal G \rangle) = \langle
\varepsilon_S(\mathcal G) \rangle = \langle \mathcal G^S \rangle.
\end{align*}
We write $\mathcal M = \{ t_1 e_{\alpha_1}, \ldots, t_\mu e_{\alpha_\mu} \}$
with $\mu \in \mathbb N$, terms $t_1, \ldots, t_\mu \in \mathbb T^n$, and
indices $\alpha_1, \ldots, \alpha_\mu \in \{ 1, \ldots, r \}$. Let $i, j \in \{
1, \ldots, \mu \}$ be such that
\begin{align*}
t_i e_{\alpha_i} + S = \varepsilon_S(t_i e_{\alpha_i}) = \varepsilon_S(t_j
e_{\alpha_j}) = t_j e_{\alpha_j} + S.
\end{align*}
According to the definition of~$U$ and the considerations above, we then have
\begin{align*}
t_i e_{\alpha_i} - t_j e_{\alpha_j} \in S \cap \langle \mathcal M \rangle_K
\subseteq U \cap \langle \mathcal M \rangle_K.
\end{align*}
As~$\mathcal G$ is the $\mathcal M$-module border basis of~$U$,
Corollary~\ref{thm:char} yields $t_i e_{\alpha_i} = t_j e_{\alpha_j}$. Hence it
follows that the restriction~$\restrict{\varepsilon_S}{\mathcal M}$ is injective
and thus~$\mathcal M$ is an order module characterizing~$\mathcal M^S =
\varepsilon_S(\mathcal M)$ by Definition~\ref{defn:charOrderModule}.\\
If we show that~$\mathcal G \subseteq P^r$ is the $\mathcal M$-module border
prebasis characterizing~$\mathcal G^S \subseteq P^r / S$,
Theorem~\ref{thm:char-quot} yields that $\mathcal G^S$ is the $\mathcal
M^S$-quotient module border basis of~$U^S$. We write $\partial \mathcal M = \{
b_1 e_{\beta_1}, \ldots, b_\omega e_{\beta_\omega} \}$ with $\omega \in \mathbb
N$, terms $b_1, \ldots, b_\omega \in \mathbb T^n$ and indices $\beta_1, \ldots,
\beta_\omega \in \{ 1, \ldots, r \}$. Additionally, we write $\mathcal G = \{
\mathcal G_1, \ldots, \mathcal G_\omega \}$ with
\begin{align*}
\mathcal G_j = b_j e_{\beta_j} - \sum_{i=1}^\mu c_{ij} t_i e_{\alpha_i}
\end{align*}
where $c_{1j}, \ldots,
c_{\mu j} \in K$ for all $j \in \{ 1, \ldots, \omega \}$. Let $v, w \in \{
1, \ldots, \nu \}$ be such that
\begin{align*}
b_v e_{\beta_v} + S = b_w e_{\beta_w} + S.
\end{align*}
Then we have
\begin{align*}
b_v e_{\beta_v} - b_w e_{\beta_w} \in S \subseteq U.
\end{align*}
Since
\begin{align*}
\mathcal G_v - \mathcal G_w = b_v e_{\beta_v} - b_w e_{\beta_w} -
\sum_{i=1}^\mu (c_{iv} - c_{iw}) t_i e_{\alpha_i} \in \langle \mathcal G
\rangle = U,
\end{align*}
it follows that
\begin{align*}
U = \sum_{i=1}^\mu (-c_{iv} + c_{iw}) t_i e_{\alpha_i} + U.
\end{align*}
As~$\mathcal G$ is the $\mathcal M$-module border basis of~$U$, this equation
yields $c_{iv} = c_{iw}$ for all $i \in \{ 1, \ldots, \mu \}$ by
Definition~\ref{defn:moduleBB}. Altogether, Definition~\ref{defn:charModuleBB}
yields that~$\mathcal G$ is the $\mathcal M$-module border prebasis
characterizing~$\mathcal G^S$. In particular, $\mathcal G^S$ is the $\mathcal
M^S$-quotient module border basis of~$\langle \mathcal G^S \rangle = U^S$ by
Theorem~\ref{thm:char-quot}.
\end{proof}

\begin{exmp}
\label{exmp:quotModuleBBAlgo}
Let $K = \mathbb Q$, $P = \mathbb Q[x,y]$, $\sigma = \DegRevLex$, and $\{ e_1,
e_2 \}$ be the canonical $P$-module basis of $P^2$. Moreover, we let
\begin{align*}
S = \langle \mathcal S_1 \rangle = \langle (x + y + 1, -x + y) \rangle \subseteq
P^2
\end{align*}
and
\begin{align*}
U^S = \langle \mathcal B_1 + S, \ldots, \mathcal B_4 + S \rangle \subseteq P^2 /
S
\end{align*}
with
\begin{align*}
\{ \mathcal B_1, \ldots, \mathcal B_4 \} = \{ (-2, 3x - 1), (3x + 4, 2 e_2), (0,
y - 1), (y - 1, 0) \}.
\end{align*}
Like in Example~\ref{exmp:moduleBBAlgo}, we have
\begin{align*}
\codim_K(\langle \mathcal B_1, \ldots, \mathcal B_4 \rangle + \langle \mathcal
S_1 \rangle, P^r) =
\codim_K(\langle \mathcal B_1, \ldots, \mathcal B_4, \mathcal S_1 \rangle, P^r)
< \infty,
\end{align*}
i.\,e.\ the requirements of the Quotient Module Border Bases
Algorithm~\ref{algo:quotModuleBB} are satisfied.

We now take a closer look at the steps of Algorithm~\ref{algo:quotModuleBB}
applied to~$\{ \mathcal B_1, \ldots, \mathcal B_4 \}$, $\{ \mathcal S_1 \}$,
and~$\sigma$. We see that the construction of~$S$ in
line~\ref{algo:quotModuleBB-initS} and of~$\varepsilon_S$ in
line~\ref{algo:quotModuleBB-initEpsilon} coincides with our definitions of~$S$
and~$\varepsilon_S$. In line~\ref{algo:quotModuleBB-corrModuleBB}, we apply the
Module Border Basis Algorithm~\ref{algo:moduleBB} to~$\{ \mathcal B_1, \ldots,
\mathcal B_4, \mathcal S_1 \}$ and~$\sigma$. As we have already seen in
Example~\ref{exmp:moduleBBAlgo}, the result of Algorithm~\ref{thm:moduleBBAlg}
applied to this input data is the order module
\begin{align*}
\mathcal M = \{ e_1, e_2 \} \subseteq \mathbb T^2 \langle e_1, e_2 \rangle
\end{align*}
and the $\mathcal M$-module border basis $\mathcal G = \{ \mathcal G_1, \ldots,
\mathcal G_4 \} \subseteq P^2$ with
\begin{align*}
\mathcal G_1 & = x e_1 + \tfrac{4}{3} e_1 + \tfrac{2}{3} e_2,\\
\mathcal G_2 & = x e_2 - \tfrac{2}{3} e_1 - \tfrac{1}{3} e_2,\\
\mathcal G_3 & = y e_1 - e_1,\\
\mathcal G_4 & = y e_2 - e_2
\end{align*}
of~$\langle \mathcal B_1, \ldots, \mathcal B_4, \mathcal S_1 \rangle$.
Line~\ref{algo:quotModuleBB-orderQuotModule} now yields the order quotient
module
\begin{align*}
\mathcal M^S = \varepsilon_S(\mathcal M) = \{ e_1 + S, e_2 + S \} \subseteq
P^2 / S
\end{align*}
and the $\mathcal M^S$-quotient module border prebasis $\mathcal G^S =
\varepsilon_S(\mathcal G) = \{ \mathcal G_1^S, \ldots, \mathcal G_4^S \}
\subseteq P^2 / S$ with
\begin{align*}
\mathcal G_1^S & = \varepsilon_S(\mathcal G_1) = x e_1 + \tfrac{4}{3} e_1 +
\tfrac{2}{3} e_2 + S,\\
\mathcal G_2^S & = \varepsilon_S(\mathcal G_2) = x e_2 - \tfrac{2}{3} e_1 -
\tfrac{1}{3} e_2 + S,\\
\mathcal G_3^S & = \varepsilon_S(\mathcal G_3) = y e_1 - e_1 + S,\\
\mathcal G_4^S & = \varepsilon_S(\mathcal G_4) = y e_2 - e_2 + S.
\end{align*}
Then the algorithm returns $(\mathcal M^S, \mathcal G^S)$ and stops.

According to Corollary~\ref{thm:quotModuleBBAlg}, $\mathcal G^S$ is the
$\mathcal M^S$-quotient module border basis of~$U^S$. Note that this result
coincides with Example~\ref{exmp:char-quot}
\end{exmp}

At last, we now derive characterizations of quotient module border bases from
Theorem~\ref{thm:char-quot} and from the characterizations of module border
bases in Section~\ref{sect:chars}.

\begin{cor}[Quotient Module Border Bases and Special Generation]
\label{thm:specGen-quot}
Suppose there exists an order module $\mathcal M \subseteq \mathbb T^n \langle
e_1, \ldots, e_r \rangle$ characterizing~$\mathcal M^S$ and let $\mathcal G = \{
\mathcal G_1, \ldots, \mathcal G_\omega \} \subseteq P^r$ be the $\mathcal
M$-module border prebasis characterizing~$\mathcal G^S$ like in
Definition~\ref{defn:charModuleBB}. Then the $\mathcal M^S$-quotient module
border prebasis~$\mathcal G^S$ is the $\mathcal M^S$-quotient module border
basis of~$\langle \mathcal G^S \rangle$ if and only if $S \subseteq \langle
\mathcal G \rangle$ and the following equivalent conditions are satisfied.
\begin{enumerate}
\renewcommand{\labelenumi}{$A_{\arabic{enumi}})$}
  \item For every vector $\mathcal V \in \langle \mathcal G \rangle \setminus
  \{ 0 \}$, there exist polynomials $p_1, \ldots, p_\omega \in P$ such that
  \begin{align*}
  \mathcal V = p_1 \mathcal G_1 + \cdots + p_\omega \mathcal G_\omega
  \end{align*}
  and
  \begin{align*}
  \deg(p_j) \leq \ind_{\mathcal M}(\mathcal V) - 1
  \end{align*}
  for all $j \in \{ 1, \ldots, \omega \}$ with~$p_j \neq 0$.
  \item For every vector $\mathcal V \in \langle \mathcal G \rangle \setminus
  \{ 0 \}$, there exist polynomials $p_1, \ldots, p_\omega \in P$ such that
  \begin{align*}
  \mathcal V = p_1 \mathcal G_1 + \cdots + p_\omega \mathcal G_\omega
  \end{align*}
  and
  \begin{align*}
  \max \{ \deg(p_j) \mid j \in \{ 1, \ldots, \omega \}, p_j \neq 0 \} =
  \ind_{\mathcal M}(\mathcal V) - 1.
  \end{align*}
\end{enumerate}
\end{cor}

\begin{proof}
This corollary is a direct consequence of Theorem~\ref{thm:specGen} and
Theorem~\ref{thm:char-quot}.
\end{proof}

\begin{cor}[Quotient Module Border Bases and Border Form Modules]
\label{thm:BFMod-quot}
Suppose there exists an order module $\mathcal M \subseteq \mathbb T^n \langle
e_1, \ldots, e_r \rangle$ characterizing~$\mathcal M^S$ and let $\mathcal G = \{
\mathcal G_1, \ldots, \mathcal G_\omega \} \subseteq P^r$ be the $\mathcal
M$-module border prebasis characterizing~$\mathcal G^S$ like in
Definition~\ref{defn:charModuleBB}. Then the $\mathcal M^S$-quotient module
border prebasis~$\mathcal G^S$ is the $\mathcal M^S$-quotient module border
basis of~$\langle \mathcal G^S \rangle$ if and only if $S \subseteq \langle
\mathcal G \rangle$ and the following equivalent conditions are satisfied.
\begin{enumerate}
\renewcommand{\labelenumi}{$B_{\arabic{enumi}})$}
  \item For every $\mathcal V \in \langle \mathcal G \rangle \setminus \{ 0 \}$,
  we have
  \begin{align*}
  \Supp(\BF_{\mathcal M}(\mathcal V)) \subseteq \mathbb T^n \langle e_1, \ldots,
  e_r \rangle \setminus \mathcal M.
  \end{align*}
  \item We have
  \begin{align*}
  \BF_{\mathcal M}(\langle \mathcal G \rangle) = \langle \BF_{\mathcal
  M}(\mathcal G_1), \ldots, \BF_{\mathcal M}(\mathcal G_\omega) \rangle =
  \langle b_1 e_{\beta_1}, \ldots, b_\omega e_{\beta_\omega} \rangle.
  \end{align*}
\end{enumerate}
\end{cor}

\begin{proof}
This corollary is a direct consequence of Theorem~\ref{thm:BFMod} and
Theorem~\ref{thm:char-quot}.
\end{proof}

\begin{cor}[Quotient Module Border Bases and Rewrite Rules]
\label{thm:rewrite-quot}
Suppose there exists an order module $\mathcal M \subseteq \mathbb T^n \langle
e_1, \ldots, e_r \rangle$ characterizing~$\mathcal M^S$ and let $\mathcal G = \{
\mathcal G_1, \ldots, \mathcal G_\omega \} \subseteq P^r$ be the $\mathcal
M$-module border prebasis characterizing~$\mathcal G^S$ like in
Definition~\ref{defn:charModuleBB}. Then the $\mathcal M^S$-quotient module
border prebasis~$\mathcal G^S$ is the $\mathcal M^S$-quotient module border
basis of~$\langle \mathcal G^S \rangle$ if and only if $S \subseteq \langle
\mathcal G \rangle$ and the following equivalent conditions are satisfied.
\begin{enumerate}
\renewcommand{\labelenumi}{$C_{\arabic{enumi}})$}
  \item For all $\mathcal V \in P^r$, we have $\mathcal V \RedR{\mathcal G} 0$
  if and only if~$\mathcal V \in \langle \mathcal G \rangle$.
  \item If $\mathcal V \in \langle \mathcal G \rangle$ is irreducible with
  respect to~$\RedR{\mathcal G}$, then we have~$\mathcal V = 0$.
  \item For all $\mathcal V \in P^r$, there is a unique vector $\mathcal W \in
  P^r$ such that $\mathcal V \RedR{\mathcal G} \mathcal W$ and~$\mathcal W$ is
  irreducible with respect to~$\RedR{\mathcal G}$.
  \item The rewrite relation $\RedR{\mathcal G}$ is confluent.
\end{enumerate}
\end{cor}

\begin{proof}
This corollary is a direct consequence of Theorem~\ref{thm:rewrite} and
Theorem~\ref{thm:char-quot}.
\end{proof}

\begin{cor}[Quotient Module Border Bases and Commuting Matrices]
\label{thm:commMat-quot}
Suppose that there exists an order module $\mathcal M \subseteq \mathbb T^n
\langle e_1, \ldots, e_r \rangle$ characterizing~$\mathcal M^S$ and let
$\mathcal G = \{ \mathcal G_1, \ldots, \mathcal G_\omega \} \subseteq P^r$ be
the $\mathcal M$-module border prebasis corresponding to~$\mathcal G^S$ like in
Definition~\ref{defn:charModuleBB}. For all $s \in \{ 1, \ldots, n \}$, we
define the map
\begin{align*}
\varrho_s: \{ 1, \ldots, \mu \} \to \mathbb N, \quad i \mapsto \begin{cases}
j & \text{if $x_s t_i e_{\alpha_i} = t_j e_{\alpha_j} \in \mathcal M$,}\\
k & \text{if $x_s t_i e_{\alpha_i} = b_k e_{\beta_k} \in \partial \mathcal M$.}
\end{cases}
\end{align*}
Then the $\mathcal M^S$-quotient module border prebasis~$\mathcal G^S$ is the
$\mathcal M^S$-quotient module border basis of~$\langle \mathcal G^S \rangle$ if
and only if $S \subseteq \langle \mathcal G \rangle$ and the following
equivalent conditions are satisfied.
\begin{enumerate}
\renewcommand{\labelenumi}{$D_{\arabic{enumi}})$}
  \item The formal multiplication matrices $\mathfrak X_1, \ldots, \mathfrak X_n
  \in \Mat_\mu(K)$ of~$\mathcal G$ are pairwise commuting.
  \item The following equations are satisfied for every $p \in \{ 1, \ldots, \mu
  \}$ and for every $s, u \in \{ 1, \ldots, n \}$ with $s \neq u$.
  \begin{enumerate}
  \renewcommand{\labelenumi}{(\arabic{enumi})}
    \item If $x_s t_i e_{\alpha_i} = t_j e_{\alpha_j}$, $x_u t_i e_{\alpha_i} =
    b_k e_{\beta_k}$, and $x_s b_k e_{\beta_k} = b_\ell e_{\beta_\ell}$ with
    indices $i, j \in \{ 1, \ldots, \mu \}$ and $k, \ell \in \{ 1, \ldots,
    \omega \}$, we have
    \begin{align*}
    \sum_{\substack{m \in \{ 1, \ldots, \mu \}\\ x_s t_m e_{\alpha_m} \in
    \mathcal M}} \delta_{p \varrho_s(m)} c_{mk} + \sum_{\substack{m \in \{ 1,
    \ldots, \mu \}\\ x_s t_m e_{\alpha_m} \in \partial \mathcal M}} c_{p
    \varrho_s(m)} c_{mk} = c_{p \ell}.
    \end{align*}
    \item If $x_s t_i e_{\alpha_i} = b_j e_{\beta_j}$ and $x_u t_i e_{\alpha_i}
    = b_k e_{\beta_k}$ with indices $i \in \{ 1, \ldots, \mu \}$ and $j, k \in
    \{ 1, \ldots, \omega \}$, we have
    \begin{align*}
    & \alignLongFormula \sum_{\substack{m \in \{ 1, \ldots, \mu \}\\ x_s t_m
    e_{\alpha_m} \in \mathcal M}} \delta_{p \varrho_s(m)} c_{mk} +
    \sum_{\substack{m \in \{ 1, \ldots, \mu \}\\ x_s t_m e_{\alpha_m} \in
    \partial \mathcal M}} c_{p \varrho_s(m)} c_{mk}\\
    & = \sum_{\substack{m \in \{ 1, \ldots, \mu \}\\ x_u t_m e_{\alpha_m} \in
    \mathcal M}} \delta_{p \varrho_u(m)} c_{mj} + \sum_{\substack{m \in \{ 1,
    \ldots, \mu \}\\ x_u t_m e_{\alpha_m} \in \partial \mathcal M}} c_{p
    \varrho_u(m)} c_{mj}.
    \end{align*}
  \end{enumerate}
\end{enumerate}
If the equivalent conditions are satisfied, for all~$s \in \{ 1, \ldots, n \}$,
the formal multiplication matrix~$\mathfrak X_s$ represents the multiplication
endomorphism of the $K$-vector space~$(P^r / S) / \langle \mathcal G^S \rangle$
defined by $\mathcal V^S + \langle \mathcal G \rangle \mapsto x_s \mathcal
V^S + \langle \mathcal G \rangle$, where $\mathcal V^S \in P^r / S$, with
respect to the $K$-vector space basis $\varepsilon_{\langle \mathcal G^S
\rangle}(\mathcal M^S) \subseteq (P^r / S) / \langle \mathcal G^S \rangle$.
\end{cor}

\begin{proof}
This corollary is a direct consequence of Theorem~\ref{thm:commMat} and
Theorem~\ref{thm:char-quot}.
\end{proof}

\begin{cor}[Quotient Module Border Bases and Liftings of Border Syzygies]
\label{thm:liftings-quot}
Suppose there exists an order module $\mathcal M \subseteq \mathbb T^n \langle
e_1, \ldots, e_r \rangle$ characterizing~$\mathcal M^S$ and let $\mathcal G = \{
\mathcal G_1, \ldots, \mathcal G_\omega \} \subseteq P^r$ be the $\mathcal
M$-module border prebasis characterizing~$\mathcal G^S$ like in
Definition~\ref{defn:charModuleBB}. Then the $\mathcal M^S$-quotient module
border prebasis~$\mathcal G^S$ is the $\mathcal M^S$-quotient module border
basis of~$\langle \mathcal G^S \rangle$ if and only if $S \subseteq \langle
\mathcal G \rangle$ and the following equivalent conditions are satisfied.
\begin{enumerate}
\renewcommand{\labelenumi}{$E_{\arabic{enumi}})$}
  \item Every border syzygy with respect to~$\mathcal M$ lifts to a syzygy
  of~$(\mathcal G_1, \ldots, \mathcal G_\omega)$.
  \item Every neighbor syzygy with respect to~$\mathcal M$ lifts to a syzygy
  of~$(\mathcal G_1, \ldots, \mathcal G_\omega)$.
\end{enumerate}
\end{cor}

\begin{proof}
This corollary is a direct consequence of Theorem~\ref{thm:liftings} and
Theorem~\ref{thm:char-quot}.
\end{proof}

\begin{cor}[Buchberger's Criterion for Quotient Module Border Bases]
\label{thm:buchbCrit-quot}
Suppose there exists an order module $\mathcal M \subseteq \mathbb T^n \langle
e_1, \ldots, e_r \rangle$ characterizing~$\mathcal M^S$ and let $\mathcal G = \{
\mathcal G_1, \ldots, \mathcal G_\omega \} \subseteq P^r$ be the $\mathcal
M$-module border prebasis characterizing~$\mathcal G^S$ like in
Definition~\ref{defn:charModuleBB}. Then the $\mathcal M^S$-quotient module
border prebasis~$\mathcal G^S$ is the $\mathcal M^S$-quotient module border
basis of~$\langle \mathcal G^S \rangle$ if and only if $S \subseteq \langle
\mathcal G \rangle$ and the following equivalent conditions are satisfied.
\begin{enumerate}
\renewcommand{\labelenumi}{$F_{\arabic{enumi}})$}
  \item We have $\NR_{\mathcal G}(\SV(\mathcal G_i, \mathcal G_j)) = 0$ for all
  $i, j \in \{ 1, \ldots, \omega \}$.
  \item We have $\NR_{\mathcal G}(\SV(\mathcal G_i, \mathcal G_j)) = 0$ for all
  $i, j \in \{ 1, \ldots, \omega \}$ such that the border terms $b_i
  e_{\beta_i}, b_j e_{\beta_j} \in \partial \mathcal M$ are neighbors with
  respect to~$\mathcal M$.
\end{enumerate}
\end{cor}

\begin{proof}
This corollary is a direct consequence of Theorem~\ref{thm:buchbCrit} and
Theorem~\ref{thm:char-quot}.
\end{proof}

%
%

\section{Subideal Border Bases}
\label{sect:subBB}

In this final section, we apply the theory of quotient module border bases
introduced in Section~\ref{sect:quotModuleBB} to the theory of subideal border
bases of zero-dimensional ideals, which have been recently introduced
in~\cite{SubBB}. To this end, we show that every subideal border basis is
isomorphic to a quotient module border basis in
Proposition~\ref{thm:subBBmoduleBB}. Using this isomorphy, we can apply all
the results of Section~\ref{sect:quotModuleBB} to subideal border bases and get
the following new results about subideal border bases. We start to deduce an
algorithm that computes subideal border bases for an arbitrary zero-dimensional
ideal in Corollary~\ref{thm:subBBcompute} and that does not need to compute any
Gröbner basis. By now, the only way to compute an arbitrary subideal border
basis has been given in \cite[Section~6]{SubBB} and requires the computation of
a Gröbner basis. After that, we deduce characterizations of subideal border
bases in Remark~\ref{rem:char-subBB} that have not been proven in~\cite{SubBB}.
In particular, this yields an alternative proof of the characterization of
subideal border bases via the special generation property, which has already
been proven in \cite[Coro.~3.6]{SubBB}, and we get the new characterizations
via border form modules, via rewrite rules, via commuting matrices, via
liftings of border syzygies, and via Buchberger's Criterion for Subideal Border
Bases.

Similar to the previous sections, we fix some notation. We let $J = \langle F
\rangle \subseteq P$ be the ideal in~$P$ generated by the set $F = \{ f_1,
\ldots, f_r \} \subseteq P \setminus \{ 0 \}$ where $r \in \mathbb N$. Recall,
for every $P$-module~$M$ (respectively $K$-vector space) and every $P$-submodule
$U \subseteq M$ (respectively $K$-vector subspace), we let
\begin{align*}
\varepsilon_U: M \twoheadrightarrow M / U, \quad m \mapsto m + U
\end{align*}
be the canonical $P$-module epimorphism ($K$-vector space epimorphism).

We start to recall the definition of an $\mathcal O_F$-order ideal, its border,
and of an $\mathcal O_F$-subideal border basis as introduced in
\cite[Defn.~2.1]{SubBB} and \cite[Defn.~2.7]{SubBB}.

\begin{defn}
\label{defn:ForderIdeal}
Let $\mathcal O_1, \ldots, \mathcal O_r \subseteq \mathbb T^n$ be order ideals.
\begin{enumerate}
  \item We call the set
  \begin{align*}
  \mathcal O_F = \mathcal O_1 \cdot f_1 \cup \cdots \cup \mathcal O_r \cdot f_r
  \subseteq J
  \end{align*}
  an \emph{$F$-order ideal}.
  \item Let $\mathcal O_F = \mathcal O_1 f_1 \cup \cdots \cup \mathcal O_r f_r
  \subseteq J$ be an $F$-order ideal. Then we call the set
  \begin{align*}
  \partial \mathcal O_F = \partial \mathcal O_1 \cdot f_1 \cup \cdots \cup
  \partial \mathcal O_r \cdot f_r \subseteq J
  \end{align*}
  the \emph{(first) border} of~$\mathcal O_F$.
\end{enumerate}
\end{defn}

\begin{defn}
\label{defn:subBB}
Let $\mathcal O_F = \mathcal O_1 f_1 \cup \cdots \cup \mathcal O_r f_r \subseteq
J$ be an $F$-order ideal where we let $\mathcal O_1, \ldots, \mathcal O_r
\subseteq \mathbb T^n$ be finite order ideals. We write $\mathcal O_F = \{ t_1
f_{\alpha_1}, \ldots, t_\mu f_{\alpha_\mu} \}$ and $\partial \mathcal O_F = \{
b_1 f_{\beta_1}, \ldots, b_\nu f_{\beta_\nu} \}$ with $\mu, \nu \in \mathbb N$,
terms $t_i, b_j \in \mathbb T^n$, and $\alpha_i, \beta_j \in \{ 1, \ldots, r \}$
for all $i \in \{ 1, \ldots, \mu \}$ and $j \in \{ 1, \ldots, \nu \}$.
\begin{enumerate}
  \item A set of polynomials $G = \{ g_1, \ldots, g_\nu \} \subseteq J$ is
  called an \emph{$\mathcal O_F$-subideal border prebasis} if the polynomials
  have the form
  \begin{align*}
  g_j = b_j f_{\beta_j} - \sum_{i=1}^\mu c_{ij} t_i f_{\alpha_i}
  \end{align*}
  with $c_{ij} \in K$ for all $i \in \{ 1, \ldots, \mu \}$ and $j \in \{ 1,
  \ldots, \nu \}$.
  \item Let $G = \{ g_1, \ldots, g_\nu \}$ be an $\mathcal O_F$-subideal border
  prebasis. We call~$G$ an \emph{$\mathcal O_F$-subideal border basis} of the
  ideal $I \subseteq P$ if $G \subseteq I$, if the set
  \begin{align*}
  \varepsilon_{I \cap J}(\mathcal O_F) = \{ t_1 f_{\alpha_1} + I \cap J, \ldots,
  t_\mu f_{\alpha_\mu} + I \cap J \}
  \end{align*}
  is a $K$-vector space basis of $J / I \cap J$, and if~$\# \varepsilon_{I \cap
  J}(\mathcal O_F) = \mu$.
\end{enumerate}
\end{defn}

For the remainder of this section, we always let $\mathcal O_F = \mathcal O_1
f_1 \cup \cdots \cup \mathcal O_r f_r$ be an $F$-order ideal with finite order
ideals $\mathcal O_1, \ldots, \mathcal O_r \subseteq \mathbb T^n$. Moreover, we
write the $F$-order ideal $\mathcal O_F = \{ t_1 f_{\alpha_1}, \ldots, t_\mu
f_{\alpha_\mu} \}$ and its border $\partial \mathcal O_F = \{ b_1 f_{\beta_1},
\ldots, b_\nu f_{\beta_\nu} \}$ with $\mu, \nu \in \mathbb N$, terms $t_i, b_j
\in \mathbb T^n$, and indices $\alpha_i, \beta_j \in \{ 1, \ldots, r \}$ for all
$i \in \{ 1, \ldots, \mu \}$ and $j \in \{ 1, \ldots, \nu \}$. Furthermore, we
let $G = \{ g_1, \ldots, g_\nu \} \subseteq J$ with polynomials $g_j = b_j
f_{\beta_j} - \sum_{i=1}^\mu c_{ij} t_i f_{\alpha_i}$ where $c_{ij} \in K$ for
all $i \in \{ 1, \ldots, \mu \}$ and $j \in \{ 1, \ldots, \nu \}$ be an
$\mathcal O_F$-subideal border prebasis.

We now prove th central result of this section, namely that every subideal
border prebasis is isomorphic as a $P$-module to a quotient module border
prebasis.

\begin{prop}
\label{thm:subBBmoduleBB}
Let $S = \Syz_P(f_1, \ldots, f_r) \subseteq P^r$. Then
\begin{align*}
\varphi: P^r / S \xrightarrow{\sim} J, \quad e_k + S \mapsto f_k
\end{align*}
is a $P$-module isomorphism such that $\varphi^{-1}(\mathcal O_F) \subseteq P^r
/ S$ is an order quotient module and $\varphi^{-1}(G) \subseteq P^r / S$ is a
$\varphi^{-1}(\mathcal O_F)$-quotient module border prebasis.
\end{prop}

\begin{proof}
The two maps
\begin{align*}
\{ e_1, \ldots, e_r \} \to P^r, \quad e_k \mapsto e_k
\end{align*}
and
\begin{align*}
\{ e_1, \ldots, e_r \} \to J, \quad e_k \mapsto f_k
\end{align*}
induce the $P$-module epimorphism
\begin{align*}
\tilde{\varphi}: P^r \twoheadrightarrow J, \quad e_k \mapsto f_k
\end{align*}
by the Universal Property of the Free Module. Additionally, we see that $S =
\ker(\tilde{\varphi})$ and the Isomorphism Theorem induces the $P$-module
isomorphism
\begin{align*}
\varphi: P^r / S \xrightarrow{\sim} J, \quad e_k + S \mapsto f_k.
\end{align*}
Moreover,
\begin{align*}
\varphi^{-1}(\mathcal O_F) = \mathcal O_1 (e_1 + S) \cup \cdots \cup \mathcal
O_r (e_r + S) \subseteq P^r / S
\end{align*}
is a quotient order module by Definition~\ref{defn:orderModule-quot} and
\begin{align*}
\varphi^{-1}(G) = \{ \varphi^{-1}(g_1), \ldots, \varphi^{-1}(g_\nu) \} \subseteq
P^r / S
\end{align*}
with
\begin{align*}
\varphi^{-1}(g_j) = b_j e_{\beta_j} - \sum_{i=1}^\mu c_{ij} t_i e_{\alpha_i} + S
\end{align*}
for all $j \in \{ 1, \ldots, \nu \}$ is an $\varphi^{-1}(\mathcal O_F)$-quotient
module border prebasis according to Definition~\ref{defn:moduleBB-quot}.
\end{proof}

Using the isomorphism constructed in Proposition~\ref{thm:subBBmoduleBB}, the
problem of the computation of subideal border bases is reduced to the problem of
the computation of quotient module border bases, which we have already solved in
Corollary~\ref{thm:quotModuleBBAlg}. We now deduce an algorithm for the
computation of subideal border bases from these results.

\begin{algorithm}[H]
\caption{${\tt subidealBB}(\{ h_1, \ldots, h_s \}, \{ f_1, \ldots, f_r \},
\sigma)$}
\begin{algorithmic}[1]
\label{algo:subBB}
  \REQUIRE $s \in \mathbb N$ and $\{ h_1, \ldots, h_s \} \subseteq P \setminus
  \{ 0 \}$,\\
  $\codim_K(\langle h_1, \ldots, h_s \rangle, P) < \infty$,\\
  $r \in \mathbb N$ and $\{ f_1, \ldots, f_r \} \subseteq P \setminus \{ 0
  \}$,\\
  $\sigma$ is a degree compatible term ordering on~$\mathbb T^n$
  \STATE $I \assign \langle h_1, \ldots, h_s \rangle$\label{algo:subBB-initI}
  \STATE $J \assign \langle f_1, \ldots, f_r \rangle$\label{algo:subBB-initJ}
  \STATE Compute $\{ \mathcal S_1, \ldots, \mathcal S_\ell \} \subseteq P^r
  \setminus \{ 0 \}$ such that $\langle \mathcal S_1, \ldots, \mathcal S_\ell
  \rangle = \Syz_P(f_1, \ldots, f_r)$, where $\ell \in \mathbb
  N$.\label{algo:subBB-gensS}
  \STATE $S \assign \langle \mathcal S_1, \ldots, \mathcal S_\ell
  \rangle$\label{algo:subBB-setS}
  \STATE Let $\varphi: P^r / S \xrightarrow{\sim} J, ~e_k + S \mapsto
  f_k$.\label{algo:subBB-setPhi}
  \STATE Compute $q_{vw} \in P$ for all $v \in \{ 1, \ldots, r \}$ and $w \in
  \{ 1, \ldots, k \}$, where $k \in \mathbb N$, such that $I \cap J = \langle
  \sum_{v=1}^r q_{vw} f_v \mid w \in \{ 1, \ldots, k \}
  \rangle$.\label{algo:subBB-intersect}
  \FOR{$w = 1, \ldots, k$}\label{algo:subBB-for}
    \STATE $\mathcal B_w \assign \sum_{v=1}^r q_{vw} e_v$\label{algo:subBB-setB}
  \ENDFOR
  \STATE $(\mathcal M^S, \mathcal G^S) \assign {\tt quotModuleBB}(\{ \mathcal
  B_1, \ldots, \mathcal B_k \}, \{ \mathcal S_1, \ldots, \mathcal S_\ell \},
  \sigma)$\label{algo:subBB-modBB}
  \STATE $\mathcal O_F \assign \varphi(\mathcal M^S)$\label{algo:subBB-setOF}
  \STATE $G \assign \varphi(\mathcal G^S)$\label{algo:subBB-setG}
  \RETURN $(\mathcal O_F, G)$
\end{algorithmic}
\end{algorithm}

\begin{cor}[Computation of Subideal Border Bases]
\label{thm:subBBcompute}
Let $r \in \mathbb N$, let $F = \{ f_1, \ldots, f_r \} \subseteq P^r \setminus
\{ 0 \}$, and let $J = \langle F \rangle \subseteq P$ be an ideal. Moreover, let
$s \in \mathbb N$, let $I = \langle h_1, \ldots, h_s \rangle \subseteq P$ with
$\{ h_1, \ldots, h_s \} \subseteq P \setminus \{ 0 \}$ be a zero-dimensional
ideal, and let~$\sigma$ be a degree compatible term ordering on~$\mathbb T^n$.
Then Algorithm~\ref{algo:subBB} is actually an algorithm and the result
\begin{align*}
(\mathcal O_F, G) \assign {\tt subidealBB}(\{ h_1, \ldots, h_s \}, \{ f_1,
\ldots, f_r \}, \sigma)
\end{align*}
of Algorithm~\ref{algo:subBB} applied to the input data~$\{ h_1, \ldots, h_s
\}$, $\{ f_1, \ldots, f_r \}$, and~$\sigma$ satisfies the following conditions.
\begin{enumerate}
\renewcommand{\labelenumi}{\roman{enumi})}
  \item The set $\mathcal O_F \subseteq J$ is an $F$-order ideal.
  \item The set $G \subseteq J$ is an $\mathcal O_F$-subideal border basis
  of~$I$.
\end{enumerate}
\end{cor}

\begin{proof}
We start to prove that the procedure is actually an algorithm. We can compute a
generating system $\{ \mathcal S_1, \ldots, \mathcal S_\ell \} \subseteq P^r
\setminus \{ 0 \}$ with $\ell \in \mathbb N$ of~$\Syz_P(f_1, \ldots, f_r)$ in
line~\ref{algo:subBB-gensS} according to \cite[Thm.~3.1.8]{KR1}. Let $S =
\langle \mathcal S_1, \ldots, \mathcal S_\ell \rangle$ be like in
line~\ref{algo:subBB-setS}, and let
\begin{align*}
\varphi: P^r / S \xrightarrow{\sim} J, \quad e_k + S \mapsto f_k
\end{align*}
be like in line~\ref{algo:subBB-setPhi}. Then we see that~$S$ and~$\varphi$
coincide with~$S$ and~$\varphi$ in Proposition~\ref{thm:subBBmoduleBB}. In
particular, $\varphi$ is a $P$-module isomorphism. We can compute the
polynomials $q_{vw} \in P$ for all $v \in \{ 1, \ldots, r \}$ and $w \in \{ 1,
\ldots, k \}$, where $k \in \mathbb N$, such that
\begin{align*}
I \cap J = \Bigl\langle \sum_{v=1}^r q_{vw} f_v \Bigm| w \in \{ 1, \ldots, k
\} \Bigr\rangle
\end{align*}
as proposed in
line~\ref{algo:subBB-intersect} according to \cite[Prop.~3.2.3]{KR1}. Thus it
only remains to prove, that line~\ref{algo:subBB-modBB} can be computed. For
every $w \in \{ 1, \ldots, k \}$, we let
\begin{align*}
\mathcal B_w = \sum_{v=1}^r q_{vw} e_v
\end{align*}
be like in line~\ref{algo:subBB-setB} and we let $U^S = \langle \mathcal B_1 +
S, \ldots, \mathcal B_w + S \rangle \subseteq P^r / S$.\\
If we show that $\codim_K(U^S, P^r / S) < \infty$, then
Corollary~\ref{thm:quotModuleBBAlg} yields that we can compute
line~\ref{algo:subBB-modBB}. The First Noether Isomorphism Theorem yields
\begin{align*}
J / I \cap J \cong I + J / I \subseteq P / I.
\end{align*}
As~$\varphi$ is an $P$-module isomorphism, and as we have $\varphi(U^S) = I \cap
J$ and $\varphi(P^r/S) = J$ by construction, it follows that
\begin{align*}
\codim_K(U^S, P^r /S) = \codim_K(I \cap J, J) \leq \codim_K(I, P) < \infty.
\end{align*}
Thus the conditions of Corollary~\ref{thm:quotModuleBBAlg} are satisfied in
line~\ref{algo:subBB-modBB}. In particular, in line~\ref{algo:subBB-modBB}, it
follows that $\mathcal M^S \subseteq P^r / S$ is an order quotient module and
that $\mathcal G^S \subseteq P^r / S$ is the $\mathcal M^S$-quotient module
border bases of~$U^S$. Thus the claim follows with
Definition~\ref{defn:ForderIdeal} and Definition~\ref{defn:subBB}
since~$\varphi$ is a $P$-module isomorphism.
\end{proof}

\begin{exmp}
\label{exmp:computeSubBB}
Let $K = \mathbb Q$ and $P = \mathbb Q[x,y]$.Moreover, let $F = \{ f_1, f_2 \}
\subseteq P \setminus \{ 0 \}$ be with
\begin{align*}
f_1 & = x - y, & f_2 & = x + y + 1,
\end{align*}
let $J = \langle F \rangle$, let
\begin{align*}
I = \langle h_1, h_2 \rangle = \langle x^2 + xy, y - 1 \rangle \subseteq P,
\end{align*}
and let $\sigma = \DegRevLex$. Then we have $x^2, y \in \LT_\sigma(I)$, i.\,e.\
$I \subseteq P$ has finite $K$-codimension in~$P$ according to the Finiteness
Criterion~\cite[Prop.~3.7.1]{KR1}.

We now take a closer look at the steps of Algorithm~\ref{algo:subBB}
applied to~$\{ h_1, h_2 \}$, $\{ f_1, f_2 \}$, and~$\sigma$. The definition
of~$I$ and~$J$ above obviously coincide with the definition of~$I$ and~$J$ in
the lines~\ref{algo:subBB-initI}--\ref{algo:subBB-initJ}. In the
lines~\ref{algo:subBB-gensS}--\ref{algo:subBB-setS}, we compute the syzygy
module
\begin{align*}
S = \Syz_P(f_1, f_2) = \langle \mathcal S_1 \rangle = \langle x + y + 1, -x + y
\rangle \subseteq P^2
\end{align*}
with the help of \cite[Thm.~3.1.8]{KR1}. Let $\{ e_1, e_2 \} \subseteq P^2$ be
the canonical $P$-module basis of the free $P$-module~$P^2$ and let
\begin{align*}
\varphi: P^2 / S \xrightarrow{\sim} J, \quad e_k + S \mapsto f_k 
\end{align*}
be like in line~\ref{algo:subBB-setPhi}. Then we use \cite[Prop.~3.2.3]{KR1} to
compute
\begin{align*}
q_{11} & = -2,   & q_{12} & = 3x-1,\\
q_{21} & = 3x+4, & q_{22} & = 2,\\
q_{31} & = 0,    & q_{32} & = y-1,\\
q_{41} & = y-1,  & q_{42} & = 0,
\end{align*}
in line~\ref{algo:subBB-intersect}. Therefore, we have
\begin{align*}
\{ \mathcal B_1, \ldots, \mathcal B_4 \} = \{ (-2, 3x-1), (3x+4, 2), (0, y-1),
(y-1,0) \} \subseteq P^2
\end{align*}
after the for-loop in line~\ref{algo:subBB-for}.

Note that~$\{ B_1, \ldots, B_4 \}$, $\{ \mathcal S_1 \}$, and~$\sigma$ is
exactly the input data of Algorithm~\ref{algo:quotModuleBB} in
Example~\ref{exmp:quotModuleBBAlgo}. Thus it follows that
\begin{align*}
\mathcal M^S = \{ e_1 + S, e_2 + S \} \subseteq P^2 / S
\end{align*}
is an order quotient module and $\mathcal G^S = \{ \mathcal G_1^S, \ldots,
\mathcal G_4^S \} \subseteq P^2 / S$ with
\begin{align*}
\mathcal G_1^S & = x e_1 + \tfrac{4}{3} e_1 + \tfrac{2}{3} e_2 + S,\\
\mathcal G_2^S & = x e_2 - \tfrac{2}{3} e_1 - \tfrac{1}{3} e_2 + S,\\
\mathcal G_3^S & = y e_1 - e_1 + S,\\
\mathcal G_4^S & = y e_2 - e_2 + S.
\end{align*}
is the $\mathcal M^S$-quotient module border basis of~$\langle \mathcal B_1 + S,
\ldots, \mathcal B_4 + S \rangle$ as computed in line~\ref{algo:subBB-modBB}.
Altogether, we compute
\begin{align*}
\mathcal O_F = \varphi(\mathcal M^S) = \{ f_1, f_2 \} \subseteq J
\end{align*}
in line~\ref{algo:subBB-setOF} and
\begin{align*}
G = \varphi(\mathcal G^S) = \{ \varphi(\mathcal G_1^S), \ldots, \varphi(\mathcal
G_4^S) \} \subseteq J
\end{align*}
with
\begin{align*}
\varphi(\mathcal G_1^S) & = x f_1 + \tfrac{4}{3} f_1 + \tfrac{2}{3} f_2,\\
\varphi(\mathcal G_2^S) & = x f_2 - \tfrac{2}{3} f_1 - \tfrac{1}{3} f_2,\\
\varphi(\mathcal G_3^S) & = y f_1 - f_1,\\
\varphi(\mathcal G_4^S) & = y f_2 - f_2
\end{align*}
in line~\ref{algo:subBB-setG}.

According to Corollary~\ref{thm:subBBcompute}, the set $\mathcal O_F \subseteq
J$ is an $F$-order ideal and~$G \subseteq J$ is an $\mathcal O_F$-subideal
border basis of~$I$.
\end{exmp}

As Remark~\ref{rem:problems-quot} has already shown for quotient module border
bases, we cannot hope to reuse all results of module border bases in a
straightforward for subideal border bases. The following remark explicitly shows
some of these problems---using the examples of a $\mathcal O_F$-index similar
to Definition~\ref{defn:index} and of a subideal border division algorithm
similar to Theorem~\ref{thm:divAlg}---and compares these concepts with the
corresponding concepts introduced in~\cite{SubBB}.

\begin{rem}
\label{rem:subBBindexEtc}
Using the isomorphism in Proposition~\ref{thm:subBBmoduleBB}, we can try to
define and proof the concepts about module border bases for subideal border
bases as described in Remark~\ref{rem:indexDivAlg-quot}. Unfortunately, there
are the same problems, e.\,g.\ this construction does not yield an $\mathcal
O_F$-index that is well-defined for any arbitrary polynomial $p \in J$.\\
But one can also define some of the concepts directly as it has been done in
\cite{SubBB}. In that paper, many concepts we have seen in
Section~\ref{sect:divAlg} have directly been introduced for subideal border
bases, e.\,g.\ an $\mathcal O_F$-index, a Subideal Border Division Algorithm, a
normal remainder, and a characterization of subideal border bases via the
special generation property.\\
In the author's master's thesis~\cite{MA}, several other characterizations have
been proven following the approach of~\cite{SubBB}, e.\,g.\ a direct version of
characterizations via border form modules, via rewrite rules, via liftings of
border syzygies, and via Buchberger's Criterion.
\end{rem}

Although we cannot define all the concepts that we introduced in
Part~\ref{part:freeMod} for subideal border bases according to
Remark~\ref{rem:subBBindexEtc}, we can characterize subideal border bases via
the isomorphism constructed in Proposition~\ref{thm:subBBmoduleBB} the same way
we have characterized quotient module border bases in
Section~\ref{sect:quotModuleBB}. In particular, this yields a new proof of the
characterization via special generation, which has first been proven in
\cite[Coro.~3.6]{SubBB}, and new characterizations via border form modules, via
rewrite rules, via commuting matrices, via liftings of border syzygies and
via (a weaker version of) Buchberger's Criterion for Subideal Border Bases for
Subideal Border Bases.

\begin{rem}[Characterizations of Subideal Border Bases]
\label{rem:char-subBB}
Using the isomorphism in Proposition~\ref{thm:subBBmoduleBB}, all the
characterizations of quotient module border bases in
Section~\ref{sect:quotModuleBB} immediately yield analogous characterizations of
subideal border bases. In particular, we can characterize subideal border bases
via the following properties:
\begin{enumerate}
\renewcommand{\labelenumi}{$\Alph{enumi})$}
  \item special generation, cf.\ Corollary~\ref{thm:specGen-quot},
  \item border form modules, cf.\ Corollary~\ref{thm:BFMod-quot},
  \item rewrite rules, cf.\ Corollary~\ref{thm:rewrite-quot},
  \item commuting matrices, cf.\ Corollary~\ref{thm:commMat-quot},
  \item liftings of border syszygies, cf.\ Corollary~\ref{thm:liftings-quot},
  \item Buchberger's Criterion for Subideal Border Bases, cf.\
  Corollary~\ref{thm:buchbCrit-quot}.
\end{enumerate}
\end{rem}

%
%


\begin{thebibliography}{KPR09}

\bibitem[KK05]{CharBB}
Achim Kehrein and Martin Kreuzer, \emph{{Characterizations of border bases}},
  Journal of Pure and Applied Algebra \textbf{196} (2005), 251--270.

\bibitem[KK06]{CompBB}
\bysame, \emph{{Computing Border Bases}}, Journal of Pure and Applied Algebra
  \textbf{205} (2006), 279--295.

\bibitem[KP11]{SubBB}
Martin Kreuzer and Henk Poulisse, \emph{{Subideal border bases}}, Mathematics
  of Computation \textbf{80} (2011), 1135--1154.

\bibitem[KPR09]{OilHilbert}
Martin Kreuzer, Hennie Poulisse, and Lorenzo Robbiano, \emph{{From oil fields
  to Hilbert schemes}}, Approximate Commutative Algebra (John Abbot and Lorenzo
  Robbiano, eds.), Springer, Vienna, 2009, pp.~1--54.

\bibitem[KR00]{KR1}
Martin Kreuzer and Lorenzo Robbiano, \emph{{Computational Commutative Algebra
  1}}, Springer, Heidelberg, 2000.

\bibitem[KR05]{KR2}
\bysame, \emph{{Computational Commutative Algebra 2}}, Springer, Heidelberg,
  2005.

\bibitem[Kri11]{MA}
Markus Kriegl, \emph{{Beitr\"{a}ge zur Theorie der Unterideal-Randbasen}},
  Master's thesis, Universit\"{a}t Passau, 2011, p.~105.

\bibitem[Mou07]{Mourrain-pythagore}
Bernard Mourrain, \emph{{Pythagore’s dilemma , symbolic-numeric computation,
  and the border basis method}}, Symbolic-Numeric Computation (Dongming Wang
  and Lihong Zhi, eds.), Birkh\"{a}user, 2007, pp.~223--243.

\bibitem[MT08]{Mourrain-stableNF}
Bernard Mourrain and Philippe Tr\'{e}buchet, \emph{{Stable normal forms for
  polynomial system solving}}, Theoretical Computer Science \textbf{409}
  (2008), 229--240.

\bibitem[Ste04]{Stetter2004}
Hans~J. Stetter, \emph{{Numerical polynomial algebra}}, SIAM, Philadelphia,
  2004.

\end{thebibliography}

\providecommand{\bysame}{\leavevmode\hbox to3em{\hrulefill}\thinspace}
\providecommand{\MR}{\relax\ifhmode\unskip\space\fi MR }
\providecommand{\MRhref}[2]{%
  \href{http://www.ams.org/mathscinet-getitem?mr=#1}{#2}
}
\providecommand{\href}[2]{#2}

\end{document}